\documentclass[10pt, 
]{amsart}
\usepackage{amssymb,mathrsfs}
\usepackage{amssymb}
\usepackage{amsmath}
\usepackage{mathtools}
\usepackage{amsfonts, mathrsfs}
\usepackage{epsfig}
\usepackage{color}
\usepackage{epstopdf}
\usepackage{graphicx}
\usepackage{srcltx}
\usepackage{amsthm}
\usepackage{enumerate}
\usepackage[mathscr]{eucal}
\usepackage{verbatim}
\usepackage[draft]{todonotes}
\usepackage{relsize}
\usepackage{enumitem}
\usepackage{setspace}
\usepackage{fullpage}
\usepackage[titletoc,toc,title]{appendix}
\usepackage{hyperref}
\usepackage{caption}
\usepackage{mathtools}
\usepackage{tikz}
\usepackage{empheq}

\hypersetup{
	colorlinks=false,       
	linkcolor=blue,          
	citecolor=red,        
	filecolor=magenta,      
	urlcolor=cyan           
}

\newcommand{\Be}{\begin{equation}}
\newcommand{\Ee}{\end{equation}}
\newcommand{\Bea}{\begin{eqnarray}}
\newcommand{\Eea}{\end{eqnarray}}
\newcommand{\Bln}{\begin{align}}
\newcommand{\Eln}{\end{align}}
\newcommand{\Beas}{\begin{eqnarray*}}
	\newcommand{\Eeas}{\end{eqnarray*}}
\newcommand{\Benu}{\begin{enumerate}}
	\newcommand{\Eenu}{\end{enumerate}}
\newcommand{\Bi}{\begin{itemize}}
	\newcommand{\Ei}{\end{itemize}}

\newcommand{\Rang}{\big\rangle}
\newcommand{\Lang}{\big\langle}
\def\avint{\mathop{\,\rlap{$-$}\!\!\int}\nolimits}

\def\wt#1{\widetilde{#1}}
\def\wh#1{\widehat{#1}}

\newcommand{\B}{\Big}

\numberwithin{equation}{section}
\newcommand{\supp} {{\rm supp\, }}
\newcommand{\dist} {{\rm dist}}

\newcommand{\C}{\mathbb C}

\newcommand\divergence{\operatorname{div}}
\newcommand{\La}{\Lambda}
\newcommand{\la}{\lambda}
\newcommand{\ga}{\gamma}
\newcommand{\Om}{\Omega}
\newcommand{\R}{\mathbb R}
\newcommand{\F}{\mathcal F}
\newcommand{\Z}{\mathbb Z}
\newcommand{\ta}{\tau}
\newcommand{\Si}{\Sigma}

\newcommand{\de}{\delta}
\newcommand{\diam}{\text{diam}}

\theoremstyle{plain}
\newtheorem{thm}{Theorem}[section]
\newtheorem{cor}[thm]{Corollary}
\newtheorem{lem}[thm]{Lemma}
\newtheorem{prop}[thm]{Proposition}

\newtheorem{conj}{Conjecture}

\theoremstyle{remark}
\newtheorem{rmk}{Remark}

\theoremstyle{definition}
\newtheorem{defn}{Definition}[section]

\newcommand{\Rangle}{\big\rangle}
\newcommand{\Langle}{\big\langle}
\newcommand{\inp}[1]{\langle #1\rangle}

\usepackage{accents}
\newcommand*{\bdot}[1]{%
   \accentset{\mbox{\scriptsize\bfseries .}}{#1}} 
\DeclareRobustCommand{\rchi}{{\mathpalette\irchi\relax}}
\newcommand{\irchi}[2]{\raisebox{\depth}{$#1\chi$}} 

\newcommand{\avm}{\avint_{\!\!\!\! M} }

\newcommand{\cF}{\mathcal F}
\newcommand{\jc}{{j_\circ}}
\newcommand{\ja}{{j_\ast}}
\newcommand{\ec}{\epsilon_\circ}
\newcommand{\I}{\mathrm I}

\newcommand{\fu}{\mathfrak u}
\newcommand{\bs}{\mathbb S}
\newcommand{\epc}{\epsilon_\circ}
\newcommand\bu{\mathbf {u}}
\newcommand\bv{\mathbf {v}}
\newcommand{\bA}{\mathbf A}

\newcommand{\Rd}{\mathbb R^d}
\newcommand{\dep}{\delta_2^{-\epsilon}}

\begin{document}

\title[]
{Uniqueness in the Calder\'on  problem \\ and bilinear restriction estimates }

\author{Seheon Ham}
\author{Yehyun Kwon}
\author{Sanghyuk Lee}


\address{Department of Mathematical Sciences and RIM, Seoul National University, Seoul 08826, Republic of Korea}
\email{seheonham@snu.ac.kr}
\email{shklee@snu.ac.kr}

\address{School of Mathematics, Korea Institute for Advanced Study,  Seoul 02455, Republic of Korea}
\email{yhkwon@kias.re.kr}


\begin{abstract} 
Uniqueness in the Calder\'on problem in dimension bigger than two was  usually studied under the assumption that  conductivity has bounded gradient.  For conductivities with unbounded gradients  uniqueness results  have not been known until recent years.  The latest result due to  Haberman basically relies on  the optimal $L^2$ restriction estimate for hypersurface which is known as the Tomas-Stein restriction theorem.  In the course  of developments of the Fourier restriction problem bilinear and multilinear generalizations of the (adjoint) restriction estimates under suitable transversality condition between surfaces have played  important roles. Since such advanced machineries usually provide strengthened estimates,  it seems natural to attempt to utilize these estimates to improve the known results.   In this paper, we make use of  the sharp  bilinear restriction estimates,  which is due to Tao,  and  relax the  regularity assumption on conductivity.  We also consider the inverse problem for the Schr\"odinger operator with potentials contained in the Sobolev spaces of negative orders. 
\end{abstract}

\maketitle

\section{introduction}\label{intro}
For $d\ge3$, let $\Omega\subset \mathbb R^{d}$ be a bounded domain with Lipschitz boundary, and let  $A(\Omega)$ denote the set of all functions $\gamma \in L^\infty(\Omega)$ satisfying $ \gamma\ge c$ in  $ \Omega$  for some  $ c>0$.  Throughout the paper, we assume $\gamma\in A(\Omega)$.  For $f\in H^{1/2}(\partial\Omega)$ and $\gamma\in A(\Omega)$,  we consider the Dirichlet problem:
\begin{equation}\label{eq}
\left\{ 
\begin{aligned}
\, \divergence (\gamma \nabla u)  &= 0   &\text{in}& \quad \Omega, \\[3pt] 
\,u  &= f &\text{on}& \quad  \partial\Omega.
\end{aligned}
\right. 
\end{equation}
Let $\partial/\partial\nu$ denote the outward normal derivative on the boundary $\partial \Omega$.  The Dirichlet-to-Neumann map  $\Lambda_\gamma$ is formally defined by 
$\Lambda_\gamma(f) = \gamma \frac{\partial u_f}{\partial \nu} \big |_{\partial \Omega}$.  Since the boundary value problem \eqref{eq} has a unique solution $u_f\in H^1(\Omega)$ (for example,  see  \cite[Theorem 2.52]{FSU}),  by the trace theorem and Green's formula, the operator can be formulated in the weak sense.   Precisely,  for $ f\in H^{1/2}(\partial\Omega)$ and $g\in   H^{1/2}(\partial\Omega)$, 
\[ (\Lambda_\gamma(f), g) =\int_{\Omega}\gamma \nabla u_f \cdot {\nabla v} dx  \]
where $v\in H^1(\Omega)$ and $v|_{\partial \Omega}=g$.  It is well known that $\Lambda _\gamma$ is well defined and $\Lambda _\gamma$ is continuous from $ H^{1/2}(\partial\Omega)$ to $H^{-1/2}(\partial\Omega)$.

\subsubsection*{Calder\'on's problem}  Calder\'on's inverse conductivity problem concerns whether $\gamma$ can be uniquely determined  from $\Lambda_\gamma$, that is to say,  whether the map  $\gamma\mapsto \Lambda_\gamma$ is injective.  The problem was introduced by Calder\'on \cite{Calderon} who showed uniqueness for the linearized problem.  Afterwards, numerous works have been devoted to  extending the function class $X (\Omega)\subset A(\Omega)$ for which  the map $X(\Omega) \ni \gamma \mapsto \Lambda_\gamma$  is injective (\cite{Uhlmann}).  Kohn and Vogelius \cite{KV} showed that if $\partial \Omega$ is smooth and $\La_{\ga}=0$ then $\ga$ vanishes to infinite order at $\partial \Omega$ provided that  $\ga\in C^\infty(\overline{\Omega})$ (also, see \cite{SU88}).  Consequently, the mapping $\gamma \mapsto \Lambda_\gamma$ is injective if we choose $X(\Omega)$ to be the space of analytic functions on $\overline{\Omega}$. Sylvester and Uhlmann in their influential work \cite{SU} proved that $\ga$ is completely determined by $\La_\ga$ if  $\gamma\in C^2 (\overline{\Omega})$ for $d\ge 2$.  They made use of the complex geometrical optics solutions which become most predominant tool not only in the Calder\'on problem but also in  various related problems.  Afterward, it has been shown that  regularity on conductivity can be lowered further.  The $C^2$ regularity  assumption  was relaxed to $C^{3/2 +\epsilon}$ by Brown \cite{Brown96}.  P\"aiv\"arinta,  Panchenko, and   Uhlmann \cite{PPU}  showed global uniqueness of conductivities in $W^{3/2,\infty}$, and  results with conductivities in $W^{3/2,p}$, $p>2d$ were obtained by Brown and Torres \cite{BT}. Nguyen and Spirn \cite{NS} obtained a result with conductivities in $W^{s, 3/s}$ for $3/2<s<2$ when $d=3$.  In two dimensions, the problem has different nature and uniqueness of $L^\infty$ conductivity was established by Astala and P\"aiv\"arinta \cite{AP}. Their  result is an extension of the previous ones in \cite{Nachman, BU}.  Recently, C\^arstea and Wang \cite{CW} obtained uniqueness of unbounded conductivities.  (See \cite{ALP} and references therein for related results.)  For $d\ge 3$, the regularity condition was remarkably improved by Haberman and Tataru  \cite{HT}. By making use of  Bourgain's $X^{s,b}$ type spaces, they proved uniqueness when  $\gamma\in C^1(\overline\Omega)$, or  $\gamma \in W^{1,\infty}(\Omega)$ with the assumption that $\|\nabla \log\gamma\|_{L^\infty(\overline \Omega)}$ is small.   This smallness assumption was later removed by Caro and Rogers \cite{CR}.

As already mentioned  before,  for $d\ge 3$,  most of the previous results were obtained under the assumption that  $\gamma$ has  bounded gradient.  Since the equation $\divergence(\gamma\nabla u)=0$ can  be rewritten as $\Delta u+ W \cdot \nabla u = 0$ with $W=\nabla \log \gamma$, it naturally relates to the unique continuation problem for $u$ satisfying  $|\Delta u|\le W|\nabla u|$.  Meanwhile,  it is known that the unique continuation property holds  with $W\in L_{loc}^{d}$ \cite{Wolff} and generally fails  if $W\in L_{loc}^{p}$ for $p<d$ \cite{KT}. In this regards Brown \cite{BT}  proposed a conjecture  that uniqueness should be valid for $\gamma \in W^{1,d}(\Omega)$, but no counterexample which shows the optimality of this conjecture has been known yet.  Recently, Brown's conjecture  was verified  by  Haberman \cite{Haberman15} for $d=3,4$,  and he also showed that  uniqueness remains valid even if $\nabla\gamma$ is unbounded when $d=5,6$. More precisely,  he showed that $\gamma \mapsto \Lambda_\gamma$ is injective  if $\gamma$ belongs to $W^{s,p}(\Omega)$ with $d \le p \le \infty$ and $s = 1$ for $d=3,4$,  and $d \le p < \infty$ and $s = 1 + \frac{d-4}{2p}$ for $d=5,6$.

For a given function $q$, let $\mathcal M_q$ be the multiplication operator $f\mapsto qf$ and let $O_d$ be the orthonormal  group in $\mathbb R^d$. Most important part of the argument  in Haberman  \cite{Haberman15}(\cite{HT})  is to show that there are sequences $\{U_j\}$ in $O_d$ and $\{\tau_j\}$ in $(0,\infty)$  such that 
\Be \label{asymptotic}
\lim_{j\to \infty} \| \mathcal M_{(\nabla f)\circ U_j}\|_{X_{\zeta(\tau_j)}^{1/2}\to X_{\zeta(\tau_j)}^{-1/2} }=0
\Ee
and  $\tau_j\to \infty$ as $j\to \infty$.  We refer  the reader forward to Section \ref{sec_pre} for the definition of  the spaces $X_{\zeta(\tau)}^{1/2}$ and $X_{\zeta(\tau)}^{-1/2}$.  If $f\in L^\frac d2$ and has compact support, it is not difficult to show $\lim_{\tau\to \infty}   \| \mathcal M_{f}\|_{X_{\zeta(\tau)}^{1/2}\to X_{\zeta(\tau)}^{-1/2} }=0$, see Remark \ref{nachman}. However, $\| \mathcal M_{\nabla\! f}\|_{X_{\zeta(\tau)}^{1/2}\to X_{\zeta(\tau)}^{-1/2} }$ does not behave as nicely  as $\| \mathcal M_{f}\|_{X_{\zeta(\tau)}^{1/2}\to X_{\zeta(\tau)}^{-1/2} }$. This is also related to the failure of the Carleman estimate of the form $ \|e^{v\cdot x} \nabla u\|_{L^q}\le C\| e^{v\cdot x} \Delta u\|_{L^p}$ when $d\ge 3$.  See \cite{KRS, BKRS, Wolff, JKL}.  To get around the difficulty  averaged estimates  over $O_d$ and $\tau$ were considered (\cite{HT, NS, Haberman15}).  In view of Wolff's work \cite{Wolff}  it still seems plausible to expect \eqref{asymptotic} or its variant  holds with $f\in L^d_{loc}$.

\subsubsection*{Restriction estimate} Let $S\subset \mathbb R^{d-1}$\,\footnote{We use $d-1$ instead of $d$ to avoid confusion in the subsequent discussion. } be a smooth compact  hypersurface with nonvanishing Gaussian curvature and let $d\mu$ be the surface measure on $S$.  The  estimate $\| \widehat f \vert_{S} \|_{ L^{2}(d\mu)} \lesssim \|f\|_{L^r(\mathbb R^{d-1})}$, $r \le 2d/(d+2)$ is known as the Stein-Tomas theorem.  The range of $r$ is optimal since the estimate fails if  $r> 2d/(d+2)$.  The restriction  estimate can be rewritten in its adjoint form:
\begin{equation}\label{adj-lin}
\| \widehat{f d\mu}  \|_{L^q(\mathbb R^{d-1})} \le C(d, p,q, S) \| f\|_{L^p(S , d\mu )}. 
\end{equation}
The restriction conjecture is to determine $(p,q)$ for which \eqref{adj-lin}  holds. Even  for most typical  surfaces such as the sphere and the paraboloid, the conjecture is left open when $d\ge 4$. We refer the reader to \cite{Guth1, Guth2} for the most recent progress.  There have been bilinear and multilinear generalizations of  the linear estimate \eqref{adj-lin} under additional transversality conditions between surfaces (\cite{TVV, Tao1, BCT}), and these estimates  played important roles in development of the restriction problem.  To be precise, let $S_1, S_2 \subset S$ be hypersurfaces in $\mathbb R^{d-1}$ and  let $ d\mu_1$,  $ d\mu_2$ be the surface measures on $S_1,$ $S_2$, respectively.  The following form of estimate is called \emph{bilinear (adjoint) restriction estimate}:
\begin{equation}\label{adj-bil}
\| \widehat{f d\mu_1} \, \widehat{ g d\mu_2} \|_{L^{q/2}(\mathbb R^{d-1})} \le C\| f\|_{L^2(S_1, d\mu_1)} \| g \|_{L^2(S_2, d\mu_2)}.
\end{equation}
Under certain condition between $S_1$ and $S_2$  the estimate \eqref{adj-bil}  remains valid for some $q<\frac{2d}{d-2}$ with which \eqref{adj-lin} fails if $p=2$. (See  Theorem \ref{bilinearthm} and  \cite{Tao2, Lee1} for detailed discussion.)

By duality,  in order to get estimate for  $\| \mathcal M_{\nabla\! f}\|_{X_{\zeta(\tau)}^{1/2}\to X_{\zeta(\tau)}^{-1/2} }$, we  consider the bilinear operator  $\mathcal B_{\nabla\! f}$ which is given by 
\[  X_{\zeta(\tau)}^{1/2} \times X_{\zeta(\tau)}^{1/2}  \ni  (u,v)   \mapsto \mathcal B_{\nabla\! f}(u,v)=\inp{ \nabla\! f\, u, v}.  \]
Compared with the previous results  the main new input in  \cite{Haberman15}  was the $L^2$-Fourier restriction theorem for the sphere which is due to Tomas \cite{Tom} and Stein \cite{St-beijing}  (Theorem \ref{stein-tomas}).  This is natural  in that  the  multiplier  which defines  $X_{\zeta(\tau)}^{-1/2}$ has mass concentrated  near the surface $\Sigma^\tau$ given by \eqref{sigmatau} while  the restriction estimate provides estimates for functions of which Fourier transform concentrates near hypersurface.  The use of the bilinear restriction estimate instead of the linear one has a couple of obvious advantages.  The bilinear restriction estimate not only has a wider range of boundedness  but also naturally fits with  the bilinear  operator $\mathcal B_{\nabla\! f}$.

In this paper we aim to  improve Haberman's  results by making use of the bilinear restriction estimate \eqref{adj-bil} for the elliptic surfaces  (see  Definition \ref{elliptic}  and Theorem \ref{bilinearthm}). However, the  bilinear estimates outside of the range of the $L^2$ restriction estimate are only true under the extra separation condition between the supports of Fourier transforms of the functions (see  Corollary \ref{bilinearco}).   Such estimates cannot be put in use directly. This leads to  considerable technical involvement.  The following is our main result.

\begin{figure}
\centering
\begin{tikzpicture}[scale=0.5]\scriptsize
\begin{scope}
	\path [fill=lightgray] (0,5)--(0,1)--(25/3,1)--(10,1+2/3)--(10,5); 
	\draw (0,1)--(10,1); 
	\draw[dash pattern={on 2pt off 1pt}] (25/3,1)--(10,1+2/3); 
	\draw[->] (0,0) node[below]{$0$} -- (10+1/2,0) node[right]{$\frac1p$}; 
	\draw (0,0)--(0,0.5); 
	\draw (-0.2,0.45)--(0.2,0.55); \draw (-0.2,0.55)--(0.2,0.65); 
	\draw[->] (0,0.6) -- (0,5) node[above]{$s$}; 
	\draw (0,1)node[left]{$1$};
	\draw[dash pattern={on 2pt off 1pt}, ultra thin] (25/3,1)--(25/3,0)node[below]{$\frac16$}; 
	\draw[dash pattern={on 2pt off 1pt}, ultra thin]  (0,1+2/3)node[left]{\tiny$\frac{41}{40}$}--(10,1+2/3)--(10,0)node[below]{$\frac15$}; 
	\draw (5,-0.8) node  {(I) When $d=5$};
\end{scope}
\begin{scope}[shift={(14,0)}]
	\path [fill=lightgray] (0,5)--(0,1)--(60/7,1+40/21)--(10,1+8/3)--(10,5); 
	\draw (0,1)--(10,1); 
	\draw[dash pattern={on 2pt off 1pt}] (0,1)--(60/7,1+40/21)--(10,1+8/3); 
	\draw[->] (0,0) node[below]{$0$} -- (10+1/2,0) node[right]{$\frac1p$}; 
	\draw (0,0)--(0,0.5); 
	\draw (-0.2,0.45)--(0.2,0.55); \draw (-0.2,0.55)--(0.2,0.65); 
	\draw[->] (0,0.6) -- (0,5) node[above]{$s$}; 
	\draw (0,1)node[left]{$1$};
	\draw[dash pattern={on 2pt off 1pt}, ultra thin] (0,1+40/21)node[left]{\tiny$\frac{15}{14}$}--(60/7,1+40/21)--(60/7,0)node[below]{$\frac17$}; 
	\draw[dash pattern={on 2pt off 1pt}, ultra thin]  (0,1+8/3)node[left]{\tiny$\frac{11}{10}$}--(10,1+8/3)--(10,0)node[below]{$\frac16$}; 
	\draw (5,-0.8) node  {(II) When $d=6$};
\end{scope}
\end{tikzpicture}
\caption{The range of $(\frac1p, s)$ in Theorem \ref{main}: the line $s=1$ for $p\ge d$ corresponds to Brown's conjecture.}
\end{figure}
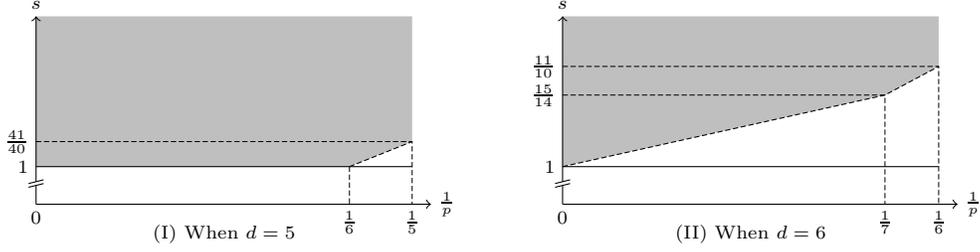

\begin{thm}\label{main}
Let $d=5,6$ and $\Omega$ be a bounded domain with Lipschitz boundary. Then  the map $\gamma\mapsto \Lambda_\gamma$ is injective if $\gamma \in W^{s,p}(\Omega)\cap A(\Omega)$ for $s> s_d(p)$, where 
\begin{equation}\label{sdp}
s_d(p) =
\begin{cases}
\, 	1 +  \frac{d - 5}{2p}   &\text{if }\ \   d+1 \le  p <\infty ,  \\[5pt]
\, 1 +  \frac{d^2-5d +6-p}{2p(d-1)} &\text{if }\ \  d\le  p < d+1. 
\end{cases}
\end{equation} 
Here, $W^{s,p}(\Omega)$ is the Sobolev-Slobodeckij space.
\end{thm}

In particular, uniqueness holds for  $\gamma \in W^{s,5}(\Omega)  \cap A(\Omega)$  if $d=5$ and $s> \frac{41}{40}$, and for $\gamma \in  W^{s,6}(\Omega)  \cap A(\Omega)$ if $d=6$ and $s> \frac{11}{10}$.  Since  $W^{s_d(p)+\epsilon, p} \not\hookrightarrow W^{1,\infty}$ if $\epsilon >0$ is small enough,  this result is not covered by the result in \cite{CR}.

Even though our estimates are stronger than those in \cite{Haberman15}, the estimates do not  immediately yield improved results in every dimensions.  As is to be seen later in the paper, our estimates for the low frequency part are especially improved but this is not the case for the high frequency part since we rely on the argument based on the properties of $X_\zeta^b$ spaces (\cite{HT,Haberman15}).

The argument based on the complex geometrical optics solutions shows that  the Fourier transforms of $q_i=\gamma_i^{-1/2}\Delta \gamma_i^{1/2}$, $i=1,2$, are identical as long as $\Lambda_{\gamma_1}=\Lambda_{\gamma_2}$.  As was indicated in \cite{Haberman15} this approach has a drawback when we deal with less regular conductivity. In order to use the Fourier  transform  one has to extend  $\gamma_1-\gamma_2 \in W_0^{s,p}(\Omega)$ to the whole space $\mathbb R^d$ such that $\gamma_1-\gamma_2=0$ on $\Omega^c$.  Such extension  is possible by  exploiting  the trace theorem (\cite[Theorem 1]{Marschall}) but only under the condition  $s - \frac 1 p \le 1$.  This additional restriction  allows new results only  for  $d=5,6$ in Theorem \ref{main}.\footnote{By the inclusion $W^{s_1,p}\subset W^{s_2,p}$ for $s_1\ge s_2$ and $1<p<\infty$, it is enough to show Theorem \ref{main} for $(s,p)$ satisfying $s_d(p)<s\le \frac1p +1$.}  The same was also true with the result in \cite{Haberman15}. However,  as is mentioned in \cite{Haberman15}, if we additionally  impose the condition ${\partial   \gamma_1}/{\partial \nu}={\partial   \gamma_2}/{\partial \nu}$ on the boundary $\partial\Omega$, then Theorem \ref{main} can be extended to higher dimensions $d \ge 7$.  In fact, by \cite[Theorem 1]{Marschall},   the restriction  $s - \frac 1 p \le 1$  can be relaxed so that $s - \frac1p \le 2$, which is valid  for  $s>s_d(p)$ for $d \ge 7$ and $p\ge d$.  See Remark \ref{sd} for the value of $s_d(p)$.  However, the additional condition on the boundary is not known to be true under the assumption $\Lambda_{\gamma_1}=\Lambda_{\gamma_2}$ for $\gamma_1, \gamma_2$ as in Theorem \ref{main}. (In  \cite{BT} Brown and Torres proved that if $\Lambda_{\gamma_1}=\Lambda_{\gamma_2}$, then ${\partial   \gamma_1}/{\partial \nu}={\partial   \gamma_2}/{\partial \nu}$ on $\partial \Omega$ for $\gamma_1, \gamma_2\in W^{3/2,p}$, $p>2d$,  and  $d\ge 3$.)

If we had the endpoint bilinear restriction estimate (i.e.,  the estimate  \eqref{adj-bil} with $q=\frac{2( d+1)}{d-1}$,  see Remark \ref{rmk-end}),  the argument in this paper would allow us to obtain the uniqueness result with $s=1$ and $p\ge 6$ when $d=5$,  and  with $s=1+1/p$ and $p\ge  8$ when $d=7$.  Unfortunately the endpoint bilinear restriction estimate is still left open.

\subsubsection*{Inverse problem for the Schr\"odinger operator} 
For $d\ge3$, let $\Omega\subset \mathbb R^{d}$ be a bounded domain with $C^{\infty}$ boundary. We now consider the Dirichlet problem: 
\begin{equation}\label{seq}
\left\{
\begin{aligned}
\, \Delta u-q u &= 0  &\text{ in }& \ \ \Omega, \\
\, u &= f   &\text{ on }&  \ \ \partial\Omega.
\end{aligned}
\right. 
\end{equation}

Let us  set 
\[  H^{s, p}_c(\Omega)=\{  q\in  H^{s, p}(\mathbb R^d): \supp q\subset \Omega  \}.\] 
Here $H^{s,p}$ is the Bessel potential space, see \emph{Notations} for its definition.
 Since $H^{s, p}$ is defined by Fourier multipliers, the space is more convenient for dealing with various operators which are defined by Fourier transform. If we disregard $\epsilon$--loss of the regularity, the spaces $H^{s,p}$ and $W^{s,p}$ are essentially equivalent because $W^{s_1,p}\hookrightarrow H^{s_2,p}$ and $H^{s_1,p}\hookrightarrow W^{s_2,p}$ provided $s_1>s_2$. (See \cite[Section 2.3]{Triebel} for more details.) Thus, the statement of Theorem \ref{Schrodinger} does not change if $H^{s, p}$ is replaced by $W^{s,p}$. 

Let $q\in H^{s, p}_c(\Omega)$ with $s, p$ satisfying \eqref{concon}.   We assume that  zero is not a Dirichlet eigenvalue of $\Delta -q$.  Then, by the standard argument (\cite{Mc})  we see that there is a unique solution $u_f\in H^1(\Omega) $ for every $f\in H^\frac12(\partial \Omega)$. In fact, this can be shown by a slight modification of the argument in \cite[Appendix A]{KU} (also see the proof of Lemma \ref{DtoN} where $\int quv\, dx $ is controlled while $q\in H^{s, p}_c(\Omega)$ and $u,v\in H^1(\Omega)$).

For  $q\in  H^{s, p}_c(\Omega)$  let $\mathcal L_{q}$ denote  the Dirichlet-to-Neumann map  given by 
\Be \label{DtN}   (\mathcal L_{q} f,g)= \int_\Omega \nabla u\cdot \nabla v+ q uv\, dx,\Ee
where $u$ is the unique solution  to \eqref{seq} and $v\in  H^1(\Omega)$ with $v|_{\partial \Omega}=g$.   As in the Calder\'on problem,  one may ask whether $q\mapsto \mathcal L_q$  is injective.  As is well known the problem is closely related to the Calder\'on problem. In fact,  the Calder\'on problem can be reduced to the inverse problem for $\Delta-q$ with $q = \gamma^{-1/2}\Delta \gamma^{1/2}$  (see \cite{SU}).  The  problem of injectivity  of $q\mapsto \mathcal L_q$  was originally considered with  $q\in H^{0, p}_c(\Omega)$, but it is not difficult to see that we may consider   $q_1, q_2\in H^{s, p}_c(\Omega)$ with $s<0$.  (See, for example, Brown-Torres \cite{BT}.)  Since $u,$ $v\in H^1(\Omega)$, it is natural to impose $s\ge -1$. In fact,  $\mathcal L_ {q} f$ is well defined  provided that $q\in H_c^{s,p}(\Omega)$ with
\Be \label{concon}
    \max\big\{ -2+\frac dp,\, -1\big\} \le s. 
\Ee
The standard argument shows  that $\mathcal L_ {q} : H^\frac12(\partial \Omega)\to H^{-\frac12}(\partial \Omega)$  is continuous.

The injectivity of the mapping 
\begin{equation}\label{Sch_DN}
H_c^{s,p}(\Omega) \ni  q\mapsto \mathcal L_q
\end{equation} 
was shown with $s=0,$ $p=\infty$ by Sylvester and Uhlmann \cite{SU}. The result was extended to include unbounded potential $q\in L^{\frac d2+\epsilon}$ by  Jerison and Kenig (see Chanillo \cite{ch}).  The  injectivity  for $q$ contained in the Fefferman-Phong class with small norm  was shown by Chanillo  \cite{ch} and the result  for $q\in L^\frac d2$ was announced by Lavine and Nachman in \cite{N}.\footnote{Their result can also be recovered by the argument in this paper. See Remark \ref{nachman}.} Their result was  recently  extended to compact Riemannian manifolds by Dos Santos Ferreira, Kenig, and Salo \cite{FKS}. Also see   \cite{KU} for extensions to the polyharmonic operators.

The regularity requirement for $q$ can be  relaxed.  Results in this direction were obtained by Brown \cite{Brown96}, P\"aiv\"arinta,  Panchenko, and   Uhlmann \cite{PPU},  Brown and Torres \cite{BT} in connection with the Calder\'on problem.  Those results can be improved to less regular $q$. In fact, Haberman's result implies that the injectivity holds with $q \in H^{-1, d}$ when $d=3,4$ (see \cite{Haberman16}).  It seems natural to conjecture that the same is true in any higher dimensions. Interpolating this conjecture with the result due to  Lavine and Nachman \cite{N} ($q\in L^\frac d2$) leads  to the following: 

\begin{conj} \label{conj} 
Let $d/2\le  p < \infty$, and $\Omega$ be a bounded domain with Lipschitz boundary. Suppose  $q_1, q_2\in  H_c^{s, p}(\Omega) $ and  $\mathcal L_ {q_1}=\mathcal L_{q_2}$.  If  $s\ge s^*_d(p) := \max\{-1, -2+\frac{d}{p}\}$, then $q_1 =q_2$. 
\end{conj}

We define $r_d: [\frac{d}{2},\infty) \to \mathbb R $. For $3\le d \le   6$, set
\begin{align*} 
r_d(p) &= \begin{cases}
	-1+ \frac{d-5}{2p}   &\text{if }\ \   p \ge d+1,  \\[5pt] 
	-\frac 32 + \frac{d-2}p   &\text{if }\ \   d+1 > p\ge   4 ,   \\[5pt] 
	 -2 + \frac dp  &\text{if }\ \  4 > p\ge \frac{d}{2}, 
\end{cases} 
\intertext{ 
and, for $d\ge 7$, set}
r_d(p)  & =
\begin{cases}
\, 	-1 + \frac{d-5}{2p}   &\text{if }\ \   p \ge \frac{d+9}{2},  \\[5pt]
\,  -\frac 32 + \frac{3d-1}{4p}  &\text{if }\ \  \frac{d+9}{2} > p\ge \frac{3d+7}{6}, \\[5pt]
\, -\frac{12}{7} + \frac{6d}{7p}  &\text{if }\ \    \frac{3d+7}{6} > p\ge \frac d2. 
\end{cases}
\end{align*} 

The  following  is   a partial result  concerning  Conjecture \ref{conj} when $d\ge 5$.    
\begin{thm}\label{Schrodinger}  
 Let $d\ge3$ and $d/2\le p< \infty $.  The mapping \eqref{Sch_DN} is injective if $s>  \max\{-1,\, r_d(p)\}$.
\end{thm}

\begin{figure}
\centering
\begin{tikzpicture}[scale=0.5]\scriptsize
\begin{scope}
	\path [fill=lightgray] (0,1/2)--(0,-4)--(5,-4)--(10,0)--(10,1/2); 
	\draw (0,-4)--(5,-4)--(10,0)--(10,1/2); 
	\draw[->] (0,0) node[left]{$0$} -- (10+1/2,0) node[right] {$\frac1p$}; 
	\draw[->] (0,-4.5) -- (0,1/2) node[above] {$s$}; 
	\draw (0,-4)node[left]{$-1$};
	\draw (10,0)node[below]{$\frac2d$};
	\draw[dash pattern={on 2pt off 1pt}, ultra thin] (5,0)node[below]{$\frac1d$}--(5,-4); 
	\draw (5,-4.8) node  {(I) When $d=3$ or $d=4$};
\end{scope}
\begin{scope}[shift={(14,0)}]
	\path [fill=lightgray] (0,1/2)--(0,-4)--(25/6,-4)--(25/4,-3)--(10,0)--(10,1/2); 
	\draw (0,-4)--(5,-4)--(10,0)--(10,1/2); 
	\draw[dash pattern={on 2pt off 1pt}] (25/6,-4)--(25/4,-3); 
	\draw[->] (0,0) node[left]{$0$} -- (10+1/2,0) node[right] {$\frac1p$}; 
	\draw[->] (0,-4.5) -- (0,1/2) node[above] {$s$}; 
	\draw (0,-4)node[left]{$-1$};
	\draw (10,0)node[below]{$\frac25$};
	\draw[dash pattern={on 2pt off 1pt}, ultra thin] (5,0)node[below]{$\frac15$}--(5,-4); 
	\draw[dash pattern={on 2pt off 1pt}, ultra thin]  (25/6,0)node[above]{$\frac16$}--(25/6,-4); 
	\draw[dash pattern={on 2pt off 1pt}, ultra thin]  (25/4,0)node[above]{$\frac14$}--(25/4,-3); 
	\draw (5,-4.8) node  {(II) When $d=5$};
\end{scope}
\begin{scope}[shift={(0,-7)}]
	\path [fill=lightgray] (0,1/2)--(0,-4)--(30/7,-26/7)--(15/2,-2)--(10,0)--(10,1/2); 
	\draw (0,-4)--(5,-4)--(10,0)--(10,1/2); 
	\draw[dash pattern={on 2pt off 1pt}] (0,-4)--(30/7,-26/7)--(15/2,-2); 
	\draw[->] (0,0) node[left]{$0$} -- (10+1/2,0) node[right] {$\frac1p$}; 
	\draw[->] (0,-4.5) -- (0,1/2) node[above] {$s$}; 
	\draw (0,-4)node[left]{$-1$};
	\draw (10,0)node[below]{$\frac2d$};
	\draw[dash pattern={on 2pt off 1pt}, ultra thin] (5,0)node[below]{$\frac1d$}--(5,-4); 
	\draw[dash pattern={on 2pt off 1pt}, ultra thin]  (30/7,0)node[above]{$\frac1{d+1}$}--(30/7,-26/7); 
	\draw[dash pattern={on 2pt off 1pt}, ultra thin]  (15/2,0)node[above]{$\frac14$}--(15/2,-2); 
	\draw (5,-4.8) node  {(III) When $d=6$};
\end{scope}
\begin{scope}[shift={(14,-7)}]
	\path [fill=lightgray] (0,1/2)--(0,-4)--(80/17,-56/17)--(240/31,-48/31)--(10,0)--(10,1/2); 
	\draw (0,-4)--(5,-4)--(10,0)--(10,1/2); 
	\draw[dash pattern={on 2pt off 1pt}] (0,-4)--(80/17,-56/17)--(240/31,-48/31)--(10,0); 
	\draw[->] (0,0) node[left]{$0$} -- (10+1/2,0) node[right] {$\frac1p$}; 
	\draw[->] (0,-4.5) -- (0,1/2) node[above] {$s$}; 
	\draw (0,-4)node[left]{$-1$};
	\draw (10,0)node[below]{$\frac2d$};
	\draw[dash pattern={on 2pt off 1pt}, ultra thin] (5,0)node[below]{$\frac1d$}--(5,-4); 
	\draw[dash pattern={on 2pt off 1pt}, ultra thin]  (80/17,0)node[above]{$\frac2{d+9}$}--(80/17,-56/17); 
	\draw[dash pattern={on 2pt off 1pt}, ultra thin]  (240/31,0)node[above]{$\frac6{3d+7}$}--(240/31,-48/31); 
	\draw (5,-4.7) node  {(IV) When $d\ge 7$};
\end{scope}
\end{tikzpicture}
\caption{The range of $(\frac1p, s)$ in Conjecture \ref{conj} and Theorem \ref{Schrodinger}}
\label{fig2}
\end{figure}
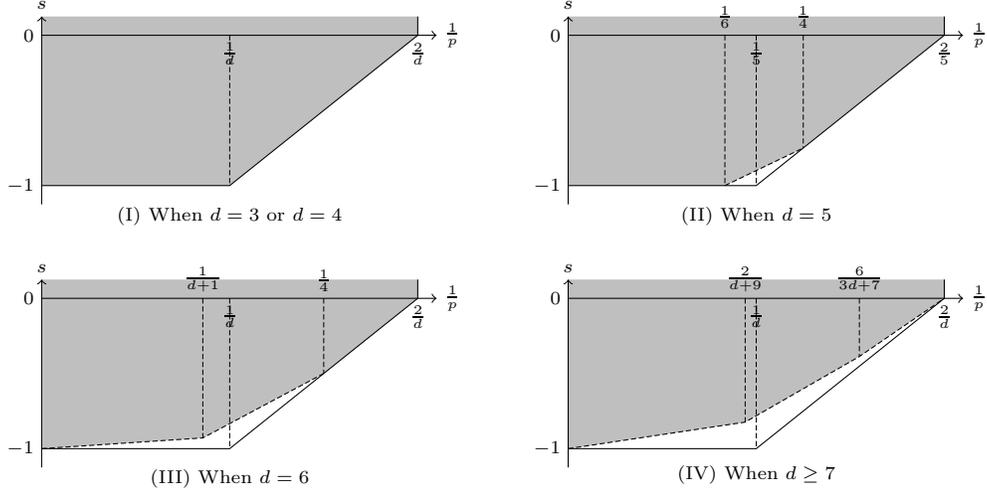

When $d=5$, Theorem \ref{Schrodinger} shows that $s>s^*_d(p) $  for $\frac d 2 \le p <4$ or  $ p \ge d+1$, where Conjecture \ref{conj} is verified   except for the critical case $s = s^*_d(p)$. Similarly, when $d=6$ injectivity of \eqref{Sch_DN} holds if $s > s^*_d(p)$ and $\frac d 2 \le p <4$. We illustrate our result in Figure \ref{fig2}.

\subsubsection*{Organization of the paper}
In Section \ref{sec_pre}, we recall basic properties of the spaces $X_\zeta^b$, and obtain estimates which rely on $L^2$ linear and $L^2$  bilinear restriction estimates. In Section \ref{sec_tech} which is the most technical part of the paper,  we make use  of the bilinear restriction estimates and a Whitney type decomposition  to get the crucial estimates while carefully considering orthogonality among the decomposed pieces. In Section \ref{sec_aver}   we take  average of the estimates from Section \ref{sec_tech}  over dilation and rotation, which allows us to exploit extra cancellation due to frequency localization. Combining the previous estimates together we prove  key estimates in Section \ref{sec_key}.  We prove Theorem \ref{main} and Theorem \ref{Schrodinger} in Section \ref{sec_pf_thm}.

\subsubsection*{Notations}  We list notations which are used throughout the paper. 
\vspace{-6pt}
\begin{itemize}
[leftmargin=0.15in]
         \item For $A, B>0$  we write $A\lesssim B$ if $A\le CB$ with some constant $C>0$. We also use the notation $A\simeq B$ if $A\lesssim B$ and $B\lesssim A$. 
	\item The orthogonal group in $\mathbb R^{d}$ is denoted by $O_d$.
	\item Let $\tau>0$. 
For a function $f:\R^d\to\C$ and a matrix $U\in O_d$  we define $f_{\tau U}(x):=\tau^{-d} f(\tau^{-1} Ux)$. If $U$ is the identity matrix,  
we denote $f_{\tau U}$  by $f_\tau$.
	\item The  Fourier  and inverse Fourier transforms: For an integrable function $u: \R^k \to \C$, we write $\mathcal F u(\xi) = \wh u(\xi) =\int_{\R^k} e^{-ix\cdot\xi} u(x)dx$ and $\mathcal F^{-1}u(\xi)=(2\pi)^{-k}\mathcal Fu(-\xi)$.
	\item For  a measurable function $a$ with polynomial growth, let $a(D)f=\F^{-1}(a\F f)$.
	\item For $E\subset\R^k$ and $x\in \R^k$, we write $\dist(x, E)=\inf\{ |x-y|: y\in E\}$.
	\item For $E\subset\R^k$ and $\delta>0$, we denote by $E+O(\delta)$ the $\delta$-neighborhood of $E$ in $\R^k$, i.e., $E+O(\delta)=\{x\in\R^k: \dist (x, E) < \delta\}$.
	\item We set $\mathbb S^{k-1}=\{x\in\mathbb R^k: |x|=1\}$. Also, for $a\in \R^k$ and $r>0$,  $B_k(a,r)=\{x\in \mathbb R^k: |x-a|<r\}$. 
	\item If $e_1$ and $e_2$ are a pair of orthonormal vectors in $\mathbb R^d$, we write $\xi_1=\xi\cdot e_1$, $\xi_2=\xi\cdot e_2$. 
	\item For $\xi\in\R^d$ we sometimes  write $\xi=(\xi_1, \widetilde \xi)\in \R \times\mathbb R^{d-1}$, $\xi=(\xi_1,\xi_2, \bar\xi)\in \mathbb R\times \R \times\mathbb R^{d-2}$. 
        \item  For $s \ge 0$ and  $p\in [1,\infty]$, we denote by $H^{s,p}$ the Bessel potential space
 $\{\varphi\in \mathcal S': (1+|D|^2)^\frac s2 \varphi \in L^p \}$ which is endowed with 	the norm $\|\varphi\|_{H^{s,p}}=\|	(1+|D|^2)^\frac s2 \varphi\|_{L^p}$. 
	\item For $s< 0$ and $p \in (1,\infty)$, we denote by  $H^{s,p'}(\Omega)$, $p'=\frac{p}{p-1}$, the dual space of $H_0^{-s,p}(\Omega)$. 
	\item We use $\inp{\cdot,\cdot}$ and $(\cdot,\cdot)$ to denote the inner product and the bilinear pairing between distribution and function, respectively.
\end{itemize}

\section{\texorpdfstring{$X^b_\zeta$}{Xb} spaces and \texorpdfstring{$L^2$}{L2space} linear and bilinear restriction estimates}\label{sec_pre}
In this section,  we  recall basic properties of the $X^b_\zeta$ spaces and  linear and bilinear restriction estimates, which are to be used later. 

\subsection{Basic properties of \texorpdfstring{$X^b_\zeta$}{Xbzetaa} spaces}  For a fixed pair of orthonormal vectors  $e_1$, $e_2$ in $\R^d$, let  us set 
\[\zeta(\tau) = \tau (e_1 -i e_2), \quad \tau>0. \] 
For $\zeta\in \C^d$ with $\zeta\cdot\zeta=0$, we denote the symbol of $e^{-x\cdot\zeta}\Delta e^{x\cdot\zeta}=\Delta+2\zeta\cdot\nabla$ by $p_\zeta$, i.e.,
\[  p_\zeta(\xi)=-|\xi|^2 +2 i \zeta\cdot \xi.\] 
By $\Sigma^\tau$ we denote  the zero set of the polynomial $p_{\zeta(\tau)}$, i.e.,  
\begin{equation}\label{sigmatau}
\Sigma^\tau=\{\xi\in\mathbb R^d: p_{\zeta(\tau)}(\xi) = 0 \}= \{ \xi \in \mathbb R^d: \xi_1 =0,\ |\xi - \tau e_2| =\tau\}.
\end{equation}
Clearly,  $\tau^{-1}\Sigma^\tau=\Sigma^1$ and it is easy to check that
\Be\label{geom_symb} 
|p_{\zeta(\tau)}(\xi) | 
	\simeq
\begin{cases}
\, \tau \dist(\xi,\Sigma^\tau) &\text{if }\ \ \dist(\xi,\Sigma^\tau) \le 2^{-7}\, \tau , \\[3pt]
\, \tau^2 + |\xi|^2 &\text{if }\ \ \dist(\xi,\Sigma^\tau) >  2^{-7}\,\tau .
\end{cases} 
\Ee
For  $\sigma,  \tau>0$ and $b\in\R$,  we denote by $X_{\zeta(\tau), \sigma}^b$ and $\bdot{X}_{\zeta(\tau)}^b$  the function spaces which were introduced in \cite{HT}(\cite{Tat, Haberman15}): 
\begin{gather*}
X_{\zeta(\tau), \sigma}^b := \big\{u\in \mathcal S'(\R^d): \|u\|_{X_{\zeta(\tau), \sigma}^b}=\|(|p_{\zeta(\tau)}|+\sigma)^b \wh u\|_{L^2(\R^d)} <\infty \big\}, \\
\bdot{X}_{\zeta(\tau)}^b := \big\{u\in \mathcal S'(\R^d): \|u\|_{\bdot{X}_{\zeta(\tau)}^b}=\| |p_{\zeta(\tau)}|^b \wh u\|_{L^2(\R^d)} <\infty \big\},
\end{gather*}
and for simplicity we also set   $X_{\zeta(\tau)}^b =X_{\zeta(\tau), \tau}^b$.  

Immediately, from the definition of $X_{\zeta(\tau)}^{1/2}$ we have 
\Be\label{x2} 
\| u \|_{L^2(\R^d)} \le C \tau^{-1/2} \| u \|_{X_{\zeta(\tau)}^{1/2}}
\Ee
with $C$ independent of $\tau>0$.

\begin{lem}[{\cite[Lemma 2.2]{HT}}, \cite{Haberman15}] \label{properties}
For $\phi\in\mathcal S(\R^d)$ the estimates  
\begin{align}
\label{lemma2.1}  
\| \phi u \|_{\bdot{X}_{\zeta(\tau)}^{-1/2}} &\le C \|  u \|_{ X_{\zeta(\tau)}^{-1/2}}, \\  
\| \phi u \|_{ X_{\zeta(\tau)}^{1/2} } &\le C \|  u \|_{\bdot{X}_{\zeta(\tau)}^{1/2}} \notag
\end{align}
hold, where $C$ depends on $\phi$, but is independent of $\tau>0$.   Consequently, for a compactly supported function $q$,  there is a constant $C>0$ such that 
\begin{align}\label{equivalence}
\| \mathcal M_q \|_{\bdot{X}_{\zeta(\tau)}^{1/2}\to\bdot{X}_{\zeta(\tau)}^{-1/2}} &\le C  \| \mathcal M_q \|_{  X_{\zeta(\tau)}^{1/2}\to  X_{\zeta(\tau)}^{-1/2}}, \\
\| \phi u \|_{L^2(\R^d)} &\le C \tau^{-1/2} \| u \|_{\bdot{X}_{\zeta(\tau)}^{1/2}}. \notag
\end{align}
\end{lem}

By dilation $\xi \rightarrow \tau\xi$, we see that 
\begin{equation}\label{scaledX}
	\| u \|_{\bdot{X}_{\zeta(\tau)}^{b}} 
	= \tau^{2b-\frac d2}\, \|  u(\tau^{-1}\,\cdot \,) \|_{\bdot{X}_{\zeta(1)}^b}, \ \ 
	\| u \|_{{X}_{\zeta(\tau)}^{b}} 
	=\tau^{2b-\frac d2}\, \|  u(\tau^{-1}\, \cdot \, ) \|_{{X}_{\zeta(1), 1/\tau}^b}. 
\end{equation} 
For any $b\in\R$, $\tau\ge1$, and $u$ with $\mathrm {dist}(\supp \wh u ,\Sigma^\tau)\ge 2^{-7}\tau$, it is easy to check by \eqref{geom_symb} that  
\[	\| u \|_{X_{\zeta(\tau)}^{b}} \simeq \| u \|_{\bdot{X}_{\zeta(\tau)}^{b}}   \] 
uniformly in $\tau\ge1$.  Equivalently, using \eqref{scaledX}, we have $\ \| u \|_{X_{\zeta(1), 1/\tau}^{b}}\simeq \| u \|_{\bdot{X}_{\zeta(1)}^{b}}$ whenever $\mathrm {dist}(\supp \wh u ,\Sigma^1)\ge 2^{-7}$.

\begin{defn}\label{mk}
Let $\kappa \ge 0$. We denote by $m^\kappa$ any (scalar or vector-valued) function which is smooth on $\mathbb R^d\setminus\{0\}$ and satisfy 
\begin{equation}\label{mkmk}
|\partial^\alpha m^\kappa(\xi)|\lesssim 
         \begin{cases} 
          \, |\xi|^{\kappa-|\alpha|}  &\text{if }\ \ |\xi|\ge 1,
            \\[3pt] 
          \,  1  &\text{if }\ \ |\xi|< 1,
         \end{cases}
\end{equation} 
for all multi-indices $\alpha$ with $|\alpha|\le d+1$.  For $\tau> 0$ we also set  $m_\tau^\kappa(\xi):=\tau^{-\kappa} m^\kappa(\tau \xi)$. 
\end{defn}
Particular examples of $m^\kappa(\xi)$ include $(1+|\xi|^2)^\frac\kappa2$ and $\xi$ (when $\kappa=1$), and it is easy to see that 
\begin{equation}\label{prop_symb}
|\partial^\alpha m^\kappa_\tau (\xi)|\lesssim 
         \begin{cases} 
          \,  |\xi|^{\kappa-|\alpha|}  &\text{if } \ \    |\xi|\ge \tau^{-1},
            \\[3pt] 
          \, \tau^{-\kappa+|\alpha|}  &\text{if } \ \  |\xi|< \tau^{-1}.
         \end{cases} 
\end{equation}
 
\begin{lem}\label{equiv}
Let $\tau>0$. The following are equivalent:
\begin{align}
\label{1-sphere}
	|\langle (m^\kappa_\tau (D) f) u,v\rangle| \le \mathcal B \|f\|_{L^{p}(\R^d)} \|u\|_{X_{\zeta(1), 1/\tau}^{1/2}} \| v\|_{X_{\zeta(1),1/\tau}^{1/2}},\\
\label{1-sphere-scaled}
	|\langle (m^\kappa (D) f)  u,  v\rangle| \le  \mathcal B \, \tau^{\frac d {p} -2+\kappa}\, \|f\|_{L^{p}(\R^d)} \|  u\|_{X_{\zeta(\tau)}^{1/2}} \|  v\|_{X_{\zeta(\tau)}^{1/2}}.  
\end{align}
\end{lem}

In particular, if  we take $\kappa=1$ and $m^1(D)=\nabla$, Lemma \ref{equiv} shows that the condition $p \ge d$ is necessary for \eqref{1-sphere-scaled} to hold uniformly in $\tau\ge1$. 

\begin{proof}[Proof of Lemma \ref{equiv}]	
First, we show \eqref{1-sphere-scaled} assuming \eqref{1-sphere}.  By Plancherel's theorem  and dilation $\xi \rightarrow \tau\xi$  we have    
\begin{align*}
\langle (m^\kappa (D) f)  u,  v \rangle 
	&=\int  \mathcal F( m^\kappa (D) f) \mathcal F^{-1}(u\overline v) d\xi
			= \frac{1}{(2\pi)^{2d}} \int m^\kappa (\xi) \widehat{f}(\xi) \int  \wh u(\eta-\xi) \overline{\widehat v}(\eta)d\eta d\xi  \\
	& = \frac{\tau^{\kappa+2d}}{(2\pi)^{2d}}   \int m^\kappa_\tau (\xi)  \wh {f_\tau}(\xi) \int \wh {u_\tau} (\eta - \xi) \overline{\wh{v_\tau}}(\eta) d\eta d\xi 
		= \tau^{\kappa+2d}   \langle (m^\kappa_\tau (D) f_\tau) u_\tau,  v_\tau\rangle. 
\end{align*}
Thus, from the assumption \eqref{1-sphere} 
 it follows that  
\begin{align*}
	 |\langle (m^\kappa (D) f)  u, v\rangle| 
		& \le \mathcal B \tau^{\kappa+2d}  \|f_\tau\|_{L^p(\R^d)}  \|u_\tau \|_{X_{\zeta(1), 1/\tau}^{1/2}} \|v_\tau \|_{X_{\zeta(1),1/\tau}^{1/2}}.
		\end{align*}
This gives the  bound \eqref{1-sphere-scaled} via \eqref{scaledX}. The same argument shows  the reverse implication from \eqref{1-sphere-scaled} to 
\eqref{1-sphere}. We omit the details. 
\end{proof}

\subsection{Linear and bilinear restriction estimates} The following is (the dual form of) the Stein-Tomas restriction theorem. The same estimate holds for any compact smooth surfaces with nonvanishing Gaussian curvature. 

\begin{thm}[\cite{Tom, St-beijing}] \label{stein-tomas} 
Let $d\ge 3$, and let $\mathbb S^{d-2}$ be the unit sphere in  $\mathbb R^{d-1}$ with the surface measure $d\sigma$.  Then  
\[  \bigg\|\int_{\mathbb S^{d-2}} e^{ix\cdot \omega} g(\omega) d\sigma(\omega)\bigg\|_{L^\frac{2d}{d-2}(\mathbb R^{d-1})}\lesssim \|g\|_{L^2(\mathbb S^{d-2})}.\] 
\end{thm}

By the standard argument and Plancherel's theorem, Theorem \ref{stein-tomas} implies the following (see \cite[Corollary 3.2]{Haberman15}), which played a key role in proving the result in \cite{Haberman15}.  

\begin{cor}\label{d-nbd} Let $d\ge 3$ and $0<\delta \ll 1$. If $\supp \widehat f \subset \mathbb S^{d-2} + O(\delta)$, then 
\Be\label{L2foli}   
\|f\|_{L^\frac{2d}{d-2}(\mathbb R^{d-1})}\lesssim \delta^\frac12\|f\|_{L^2(\R^{d-1})}.
\Ee
\end{cor}

Conversely, by a limiting argument it is easy to see that Corollary \ref{d-nbd}  implies  Theorem \ref{stein-tomas}.   

\subsubsection*{Bilinear restriction estimate for the elliptic surfaces}  
For $\varepsilon>0$ and $N\in \mathbb N$ we say $\psi:[-1,1]^{d-2}\to \mathbb R$ is of elliptic type $(\varepsilon, N)$ if  $\psi$ satisfies 
\begin{itemize}
	\item[$(i)$] $\psi(0)=0$ and $\nabla\psi(0)=0$;
	\item[$(ii)$] if $w(\xi')=\psi(\xi')-|\xi'|^2/2 $, then 
\[ \sup_{\xi'\in [-1,1]^{d-2}} \max_{0\le |\alpha|\le N} |\partial^\alpha w (\xi') | \le \varepsilon .\]
\end{itemize}

\begin{defn} \label{elliptic} We say that $S$ is an \emph{elliptic surface of type ($\varepsilon,N)$} if  $S$ is given by $S = \{ (\xi', \xi_{d-1})\in \mathbb R^{d-2}\times \mathbb R : \xi_{d-1} = \psi(\xi'), \, |\xi'|\le 1/2 \}$, where $\psi$ is of elliptic type ($\varepsilon,N)$.  
\end{defn}
Most typical examples are the paraboloid and the surface which is given by parabolic rescaling of a small subset of the sphere.  In general, any convex hypersurface with nonvanishing Gaussian curvature can be rescaled (after being decomposed into sufficiently small pieces and then translated and rotated) so that the resulting surfaces are of  elliptic type ($\varepsilon,N)$. The following sharp bilinear restriction estimate for elliptic surfaces is due to Tao \cite[Theorem 1.1]{Tao1}. 

\begin{thm}[\cite{Tao1}]\label{bilinearthm}
Let $d\ge 3$ and let $q > \frac{2(d+1)}{d-1}$. There are $\varepsilon>0$ and $N\in \mathbb N$ such that the estimate \eqref{adj-bil} holds (with $C$ independent of $S$, $S_1,$ and $S_2$) whenever  $S\subset \mathbb R^{d-1}$  is of type ($\varepsilon,N)$ and  $S_1$, $S_2\subset S$  are hypersurfaces with  $\dist(S_1, S_2) \simeq 1$.
\end{thm}

When $d=3$, the estimate \eqref{adj-bil} is  true with $q=4$. This is an easy consequence of Plancherel's theorem.  Unlike the Stein-Tomas theorem, the bilinear restriction estimate for a surface with nonvanishing Gaussian curvature exhibits different natures depending whether the surface is elliptic or not. If the surface with nonvanishing Gaussian curvature is not elliptic, the separation condition $\dist(S_1,S_2) \simeq 1$ is not sufficient in order for \eqref{adj-bil} to hold for $q< \frac{2d}{d-2}$ (for example, see \cite{Lee1} for more details).
 
\begin{rmk}\label{rmk-end}
The constant $C$ in Theorem \ref{bilinearthm} is clearly stable under small smooth perturbation of $S$. It is known that the estimate \eqref{adj-bil} fails if $q<\frac{2(d+1)}{d-1}$ but the endpoint case $q=\frac{2(d+1)}{d-1}$ is still open when $d\ge 4$.   In this case, under the assumption of Theorem \ref{bilinearthm}  the following local estimate 
\begin{equation}\label{bilinear-local}
\| \widehat{f d\mu_1} \, \widehat{ g d\mu_2} \|_{L^{\frac{d+1}{d-1}}(B_{d-1}(0,R))} \le  CR^\epsilon \| f\|_{L^2(S_1, d\mu_1)} \| g \|_{L^2(S_2, d\mu_2)} 
\end{equation}
holds for any $\epsilon >0$ and $R\ge1$ (see \cite{Tao1}) provided that $S$ is an elliptic surface of type ($\varepsilon,N)$ with small enough $\varepsilon>0$ and large enough $N$. 
\end{rmk}

Making use of the bilinear estimate \eqref{adj-bil} from Theorem \ref{bilinearthm} and interpolation,  we obtain the following. 

\begin{cor}\label{bilinearco}  
Let $d$, $S$, $S_1$, and $S_2$ be as in Theorem \ref{bilinearthm} and let $0<\delta_1,\,  \delta_2\ll 1$. Suppose that $\supp \widehat{u_i}\subset S_i+ O(\delta_i) $, $i=1,2$.  Then, for any  $\epsilon>0$, there exists a constant $C= C(\varepsilon,N,\epsilon ,d)$ such that 
\begin{equation}\label{bibi}
\|u_1 u_2 \|_{L^{\frac{d+1}{d-1}}(\mathbb R^{d-1})} \le C  (\delta_1\delta_2)^{\frac12-\epsilon} \| u_1\|_{ L^2(\R^{d-1})} \| u_2\|_{ L^2(\R^{d-1})}.
\end{equation}
\end{cor}

If $\epsilon=0$, then \eqref{bibi} is equivalent to the endpoint bilinear restriction estimate (\eqref{adj-bil} with $q=\frac{2(d+1)}{d-1}$). 

\begin{proof}
By Theorem \ref{bilinearthm}   we have (see \cite[Proof of Lemma 2.4]{Lee} for the details), for $p>\frac{d+1}{d-1}$, 
\begin{equation*}
\|u_1 u_2 \|_{L^{p}(\mathbb R^{d-1})} \le C  (\delta_1\delta_2)^{\frac12} \| u_1\|_{ L^2(\R^{d-1})} \| u_2\|_{ L^2(\R^{d-1})}.
\end{equation*}
Interpolating this with the trivial estimate $\|u_1 u_2 \|_{L^1}\le \|u_1\|_{L^2} \|u_2\|_{L^2}$ and taking $p$ arbitrarily close to $\frac{d+1}{d-1}$ give \eqref{bibi} for any $\epsilon>0$. 
\end{proof}

\subsection{Frequency localized estimates} \label{loc_est}
We  use Corollary \ref{d-nbd} and Corollary \ref{bilinearco} to show additional estimates which we need  for proving our main estimates  in Section \ref{sec_tech}.   We begin with introducing additional notations.
 
\subsubsection*{Linear estimates} Recalling \eqref{sigmatau}, for $\tau>0$, we define  $\Sigma_\mu^\tau$ and $\Sigma_{\le \mu}^\tau$ by 
\begin{gather*}
\Sigma_\mu^\tau   = \{ \xi\in \mathbb R^{d} : \mu/2 < \dist(\xi, \Sigma^\tau) \le  \mu \}, \quad
\Sigma_{\le \mu}^\tau   = \{ \xi\in \mathbb R^{d} :  \dist(\xi, \Sigma^\tau) \le \mu \} .
\end{gather*}
By $Q_\mu^\tau$ and $Q_{\le\mu}^\tau$ we denote  the multiplier operators given by 
\[ Q_{\mu}^\tau f = \rchi_{\Si^\tau_\mu }(D) f,	\quad  Q_{\le \mu}^\tau f  = \rchi_{\Sigma_{\le\mu}^\tau}(D) f. \]
For an orthonormal basis $\{ e_i \}_{i=1}^d$ for $\mathbb R^d$, the $i$-th coordinate  $\xi_i$ of $\xi$ with respect to $\{e_i\}$ is given by $\xi_i = \xi \cdot e_i$. We write 
\[ \xi = (\xi_1, \widetilde\xi\,) = (\xi_1,\xi_2,\bar\xi\,), \quad  \widetilde \xi\in \mathbb R^{d-1}, \quad \bar \xi\in \mathbb R^{d-2}. \]
For  $0<\de<\tau$ and $h>0$, we also set  
\[ \Sigma_{\le \delta}^{\tau, h} =\{ \xi\in \mathbb R^{d} : |\xi_1| \le h, \ \tau -  \delta \le |\widetilde \xi - \tau \widetilde {e_2}| \le \tau + \delta   \}, \]
and let $Q_{\le \delta}^{\tau, h} $ be  the multiplier operator given by 
\[ Q_{\le \delta}^{\tau, h}  f = \rchi_{\Sigma_{\le \delta}^{\tau, h} }(D) f. \]

\begin{lem}\label{lin} Let $d\ge3$, $1\le\tau \le 2$ and $0<\delta,\, h\le 1$. For $2 \le p \le 2d/(d-2)$
there exists a constant $C>0$, independent of $\tau$, $\delta$ and $h$, such that
\begin{equation}\label{linear}
	\| Q_{\le \delta}^{\tau, h}   u \|_{L^p(\R^d)} \le C \, h^{\frac12-\frac1p} \, \delta^{\frac d2(\frac12 - \frac 1p)} \| u\|_{L^2(\R^d)}.
\end{equation}
\end{lem}
\begin{proof} 
When $p=2$, \eqref{linear} is obvious by Plancherel's theorem.  Thus, in view of interpolation, it suffices to  show  
\begin{equation}\label{end-linear}	
\| Q_{\le \delta}^{\tau, h}  u \|_{L^\frac{2d}{d-2}(\R^d)} \lesssim h^{\frac1d} \, \delta^{\frac12} \, \|u\|_{L^2(\R^d)} .	
\end{equation} 
Since  $\F (Q_{\le \delta}^{\tau, h}u)$ is supported in $\{ \xi: |\xi_1|\le h \}$,     one may use Bernstein's inequality to get 
\[	\| Q_{\le \delta}^{\tau, h} u (\,\cdot \, , \widetilde x) \|_{L^\frac{2d}{d-2}(\R)}
		\lesssim 	h^\frac1d \|Q_{\le \delta}^{\tau, h} u (\,\cdot \, ,   \widetilde x) \|_{L^2(\R)} \]
uniformly in $\widetilde  x\in \mathbb R^{d-1}$.  Applying  Minkowski's inequality and Corollary \ref{d-nbd} we obtain
\[	\| Q_{\le \delta}^{\tau, h} u \|_{L^\frac{2d}{d-2}(\R^d)}
	\lesssim h^\frac1d \big\| \| Q_{\le \delta}^{\tau, h} u(x_1, \, \cdot \, )\|_{ L^\frac{2d}{d-2}(\R^{d-1})} \big\|_{L^2(\R; dx_1)}\lesssim h^\frac1d \delta^\frac12  \|Q_{\le \delta}^{\tau, h} u\|_{ L^2(\R^d)}.	\qedhere\]
\end{proof}

From Lemma \ref{lin} the following is easy to show. 
\begin{lem}[{\cite[Lemma 3.3]{Haberman15}}] \label{hhh} 
For $1 < \mu \le \tau$,  we have
\begin{align} 
\label{tsts}  \| Q_{\le\mu}^\tau f\|_{L^{ \frac{ 2d}{d-2}} (\R^d)} & \lesssim (\mu /  \tau )^{1/d} \| f\|_{X_{\zeta(\tau)}^{1/2}}, \\
\label{x-1/2-0}  \|f\|_{L^\frac{2d}{d-2} (\R^d)}   & \lesssim \|f\|_{X_{\zeta(\tau)}^{1/2}}.
\end{align}
\end{lem}

\begin{proof} By rescaling the estimate \eqref{tsts} is equivalent to 
\Be\label{scaledXX}
\| Q_{\le \delta}^{\,1} f\|_{L^\frac{ 2d}{d-2}(\R^d)} \lesssim \delta^{\frac1d} \| f\|_{X_{\zeta(1),1/\tau}^{1/2}}, \quad 1/\tau <\delta\le1.
\Ee
To show this we decompose $Q_{\le \delta}^{\,1}$ dyadically as follows:
\begin{align*}
\| Q_{\le \delta}^{\,1} f\|_{L^\frac{2d}{d-2}}
	\le \| Q_{\le 1/\tau}^{\,1} f\|_{L^\frac{2d}{d-2}} +\sum_{1/\tau \le 2^j<2\delta} \| Q_{2^j}^{\,1} f\|_{L^\frac{2d}{d-2} }.
\end{align*}
Application of Lemma \ref{lin} gives the bound $\| Q_{\le 1/\tau}^{\,1} f\|_{L^\frac{2d}{d-2}} \lesssim \tau^{-\frac{d+2}{2d}} \|Q^{\,1}_{\le 1/\tau}f\|_{L^2}$, and the definition of $X^{1/2}_{\zeta(1),1/\tau}$ gives $\|Q^{\,1}_{\le 1/\tau}f\|_{L^2}\lesssim \tau^\frac12 \|f\|_{X^{1/2}_{\zeta(1),1/\tau}}$. Combining these  we get
\[	\| Q_{\le 1/\tau}^{\,1} f\|_{L^\frac{2d}{d-2}} \lesssim \tau^{-\frac1d}\|f\|_{X^{1/2}_{\zeta(1),1/\tau}}.	\]
Utilizing Lemma \ref{lin} and \eqref{geom_symb}, the similar argument gives $\| Q_{2^j}^{\,1} f\|_{L^\frac{2d}{d-2}}\lesssim 2^\frac jd \|f\|_{X^{1/2}_{\zeta(1),1/\tau}}$. Now summation of the estimates over $j$ gives the desired bound \eqref{scaledXX}.

The estimate  \eqref{x-1/2-0}  is even easier to prove once we have \eqref{tsts} since $|p_{\zeta(\tau)}(\xi)|\simeq \tau^2+ |\xi|^2$ if $\dist(\xi,\Sigma^\tau)> 2^{-7}\tau$. In fact, $\| f- Q_{\le 2^{-7}\tau}^{\, \tau} f\|_{L^{ \frac{ 2d}{d-2}}(\R^d)} \lesssim \| f\|_{X_{\zeta(\tau)}^{1/2}}$ follows by the Hardy-Littlewood-Sobolev inequality  and Plancherel's theorem.  Combining this and \eqref{tsts}  with $\mu=2^{-7}\tau$ yields \eqref{x-1/2-0}. 
\end{proof}

\subsubsection*{Bilinear estimates}  We obtain some bilinear estimates which are consequences of the bilinear restriction estimate  \eqref{bibi}.

\begin{lem}\label{1separation}
Let $d\ge3$, $0< \delta_2 \le \delta_1 \ll 1 $, $0 <h_2 \le h_1 \ll 1$, and let $S$, $S_1$, and $S_2$ be as in  Theorem \ref{bilinearthm}.  Suppose that 
\[  \supp\widehat {u_j} \subset \big\{  (\xi_1, \widetilde\xi) \in\R\times\R^{d-1} : |\xi_1|\le h_j, \,   \widetilde \xi\in  S_j + O(\delta_j) \big\}, \quad  j=1,2. \]
Then, for any $\epsilon>0$  and $\frac{d+1}{2} \le p\le \infty$, 
\begin{align} 
\label{interpolation}
\| u_1 u_2 \|_{L^{p'} (\mathbb R^d)}
	&\lesssim \,\delta_2^{-\epsilon} \, h_1^{\frac1p} \, ( \delta_1 \, \delta_2)^{\frac{d+1}{4p}} \, \|u_1\|_{L^2(\R^d)} \|u_2\|_{L^2(\R^d)}, \\
\label{interpolation2}
\| u_1 u_2 \|_{L^{p'} (\mathbb R^d)}
	&\lesssim \,\delta_2^{-\epsilon} \, (h_1/h_2)^{\frac{d-3}{4p}}h_2^{\frac1p} \, ( \delta_1 \, \delta_2)^{\frac{d+1}{4p}} \, \|u_1\|_{L^2(\R^d)} \|u_2\|_{L^2(\R^d)}. 
\end{align}
\end{lem} 

When $d=7$ the estimates \eqref{interpolation} and \eqref{interpolation2} are identical.  If $d< 7$,  \eqref{interpolation2} gives a bound  better than \eqref{interpolation}.  When $d> 7$, the bound from \eqref{interpolation} is stronger.  The bounds in \eqref{interpolation} and \eqref{interpolation2} are sharp in that the exponents of $\delta_1, \delta_2$ cannot be improved except for the $\delta_2^{-\epsilon}$ factor.  This can be shown without difficulty by modifying  the (squashed cap) example in \cite{TVV}, especially with $\delta_1\simeq\delta_2$ and $h_1\simeq h_2$. 
\begin{rmk}   
There are  linear counterparts of \eqref{interpolation} and \eqref{interpolation2}.  Let $\tau\simeq 1$, $0<h_2\le h_1$, and $ 0 < \delta_2\le \delta_1 $. Suppose $\supp \widehat{u_i}$ is contained in $\Sigma_{\le \delta_i}^{\tau, h_i}$, $i= 1,2.$  Then \eqref{end-linear} implies that
\begin{equation} \label{LIN}
\|u_1u_2\|_{L^\frac{d}{d-1}(\R^d)} \le \|u_1\|_{L^2} \|u_2\|_{L^\frac{2d}{d-2}} \lesssim  h_2^{\frac1d} \delta_2^{\frac12}\|u_1\|_{L^2} \|u_2\|_{L^2} .
\end{equation}
We may compare this estimate with the estimate \eqref{interpolation}.  In particular, with  $p = d$ in  \eqref{interpolation} and the assumption in Lemma \ref{1separation}, we have 
\begin{equation}\label{BIL} 
\| u_1 u_2\|_{L^\frac{d}{d-1}(\R^d)}  \lesssim \delta_2^{-\epsilon} h_1^{\frac 1 d} (\delta_1\delta_2)^{\frac{d+1}{4d}} \|u_1\|_{L^2} \|u_2\|_{L^2} .
\end{equation}  
If  $\delta_1\simeq\delta_2\simeq h_1\simeq h_2 \simeq \delta$, the bound of \eqref{BIL} is roughly  better than that of  \eqref{LIN} by a factor of $\delta^{\frac 1 {2d}}$.   However, since the  estimate \eqref{BIL} is only possible under additional assumption on the  supports of $\widehat{u_1}$,  $\widehat {u_2}$, we cannot  directly exploit this improvement. Nevertheless, this will be made possible via the bilinear method  which has been used in the study of restriction problem (\cite{TVV}). 
\end{rmk}

\begin{proof}[Proof of Lemma \ref{1separation}]
Let us first assume $d\ge4$. By interpolation with a trivial estimate  
\[
\|u_1 u_2\|_{L^1} \le \| u_1\|_{L^2} \| u_2\|_{L^2},
\]
it suffices to prove \eqref{interpolation} and \eqref{interpolation2}  for $p' = \frac{d+1}{d-1}$.  Since the one-dimensional Fourier transform  $ \mathcal F(u_1(\,\cdot\, , \widetilde x)\,  u_2(\,\cdot\, , \widetilde x))$ is supported in the interval $[-2 h_1, 2 h_1]$,  Bernstein's inequality gives 
\[  \| u_1 (\,\cdot\, , \widetilde x)\,  u_2(\, \cdot \, , \widetilde x)\|_{L^{\frac{d+1}{d-1}} (\mathbb R)}
		\lesssim h_1^{1-\frac{d-1}{d+1}} \| u_1(\,\cdot\, , \widetilde x)\,  u_2(\, \cdot \, , \widetilde x)\|_{L^1(\R)}.	\]
Since 
$ \supp\cF (u_i(x_1, \, \cdot\, ) ) \subset  S_i + O(\delta_i)$ for $i=1,2$, we   have
\[	\| u_1(x_1, \,\cdot\,)\, u_2(x_1, \,\cdot\, ) \|_{L^{\frac{d+1}{d-1}} (\mathbb R^{d-1})}   
		\lesssim \delta_2^{-\epsilon} \delta_1^{\frac12} \delta_2^{\frac12}  \|u_1(x_1, \,\cdot\,)\|_{L^2(\R^{d-1})} \|u_2(x_1, \,\cdot\,)\|_{L^2(\R^{d-1})},	\] 
which follows from Corollary \ref{bilinearco}. 
Thus, by Minkowski's inequality and the Cauchy-Schwarz inequality, we obtain
\begin{align*}
\| u_1 u_2 \|_{L^{\frac{d+1}{d-1}} (\mathbb R^d)} 
	& \lesssim  h_1^{1-\frac{d-1}{d+1}} \int  \| u_1(x_1, \,\cdot\,)\, u_2(x_1, \,\cdot\, ) \|_{L^{\frac{d+1}{d-1}} (\mathbb R^{d-1})}   d x_1  \\
	&  \lesssim \dep h_1^{1-\frac{d-1}{d+1}} \delta_1^{\frac12} \delta_2^{\frac12}\, \|u_1\|_{L^2} \|u_2\|_{L^2}. 
\end{align*}   
This completes the proof of \eqref{interpolation}.

Now we prove \eqref{interpolation2} for $p' = \frac{d+1}{d-1}$, which can be deduced from \eqref{interpolation}. 
Let $\{I_\ell\}$ be a collection of disjoint subintervals of $[- h_1,  h_1]$ each of which has side-length $\simeq h_2$ and $[- h_1,  h_1]=\bigcup_{\ell} I_\ell$.  
It is clear that
\[	u_1u_2= \sum_\ell \F^{-1} \big( \chi_{I_\ell}(\xi_1) \widehat{u_1}(\xi)\big) u_2 = : \sum_\ell u_1^\ell u_2 ,	\] 
and $\F( u_1^\ell u_2)$ is supported in $(I_\ell+[-h_2,h_2])\times \mathbb R^{d-1}$.  
Since $p'\in [1,2]$, from the well-known orthogonality argument (\cite[Lemma 6.1]{TVV}) we see that
\[	\B\|\sum_\ell u_1^\ell u_2 \B\|_{L^{p'} (\mathbb R^d)}
		\lesssim \B(\sum_{\ell}\| u_1^\ell u_2 \|_{L^{p'} (\mathbb R^d)}^{p'} \B)^\frac1{p'}.	\]
The length of the interval $I_\ell+[-h_2,h_2]$ is $\simeq h_2$. Hence, \eqref{interpolation} gives 
\[	\| u_1^\ell u_2 \|_{L^{p'} (\mathbb R^d)} \lesssim \,\dep\, h_2^{\frac1p} \, ( \delta_1 \, \delta_2)^{\frac{d+1}{4p}} \, \|u_1^\ell \|_{L^2} \|u_2\|_{L^2}.\]
Combining the above two inequalities and using H\"older's inequality, we get 
\begin{align*}
	\| u_1  u_2 \|_{L^{p'} (\mathbb R^d)}	
	&\lesssim \dep \, h_2^{\frac1p}  ( \delta_1 \, \delta_2)^{\frac{d+1}{4p}}  \| u_2 \|_{L^2} \B( \sum_{\ell}\|u_1^\ell \|_{L^2}^{p'} \B)^\frac1{p'} 	\\
	&\lesssim \dep (h_1/h_2)^{\frac12-\frac{2}{d+1}}\, h_2^{\frac1p} \, ( \delta_1 \, \delta_2)^{\frac{d+1}{4p}}  \|u_2\|_{L^2} \B(\sum_{\ell}\|u_1^\ell \|_{L^2}^2 \B)^\frac1{2},
\end{align*}
which gives \eqref{interpolation2} with $p' = \frac{d+1}{d-1}$. 

When $d=3$ we have the endpoint bilinear restriction estimate (\eqref{adj-bil} with $d=3$ and $q=4$). Hence the estimate \eqref{bibi} in Corollary \ref{bilinearco} is true without $\dep$.  The same argument gives \eqref{interpolation} and \eqref{interpolation2} without $\dep$.
\end{proof}

\section{Bilinear  ${X}_{\zeta(1), 1/\tau}^{1/2}$  estimates }\label{sec_tech}
As mentioned in the introduction, we regard  $\inp{ (\nabla f) u, v}$ as  a bilinear operator and attempt to obtain estimates while $u,v\in {X}_{\zeta(\tau)}^{1/2}$. In order to make use of  the restriction estimates and its variants we work in frequency local setting after rescaling  $\xi\to \tau\xi$. This enables us to  deal with $\Sigma^1$ instead of $\Sigma^\tau$  which varies along $\tau$.  In this section we use the estimates in the previous section to obtain estimates  for $\inp{(\nabla f) u, v}$  in terms of  ${X}_{\zeta(1), 1/\tau}^{1/2}$.

\subsection{Localization  near $\Sigma^1$} 
Throughout this section (Section \ref{sec_tech})  we assume that 
\[	\supp   \widehat u, \  \supp \widehat v \subset B_d(0,4),	\]  
and  obtain bounds on $\langle (m_\tau^\kappa (D) f) u,v\rangle$ while $u,v$ are in $X^{1/2}_{\zeta(1),1/\ta}$. Note that $\widehat u(-\,\cdot) \ast \overline{\widehat v}$  is supported in $B_d(0,8)$.   

Since $u,v$ are in $X^{1/2}_{\zeta(1),1/\tau}$,  $\widehat u$ and $\widehat v$ exhibit singular behavior  near the set  $\Si^1$.  Meanwhile, the desired estimates  are  easy to show if $\widehat u$ or $\widehat v$ is supported away from $\Si^1$ (see Section \ref{sec_key}).  Thus, for the rest of this section, we assume that
\Be \label{near}
	\supp \widehat u, \  \supp \widehat v \subset \Sigma^1 + O(2^{-2}\epsilon_\circ)	
\Ee
with a fixed small number $\epsilon_\circ\in(0, 2^{-7}/\sqrt d\,]$.   

Let $\beta \in C^\infty_c((\frac12,2))$ be such that $\sum_{j\in\Z} \beta(2^{-j}t) =1$ for $t>0$.  For a dyadic number $\lambda$, we define a Littlewood-Paley projection operator $P_\lambda$ by  $\mathcal F(P_\lambda f)(\xi) = \beta(|\xi|/ \lambda) \widehat f(\xi)$ and write  
\Be \label{littlewood-paley}
\begin{aligned}
	\langle  (m^\kappa_\tau (D) f) u,v\rangle 
		&= \sum_{\lambda\le 8} \langle (m^\kappa_\tau (D) P_\lambda  f) u,v\rangle \\
		&= (2\pi)^{-2d} \sum_{\lambda\le 8}\int_{\R^d}  m^\kappa_\tau (\xi) \beta\Big(\frac{|\xi|}{\lambda}\Big) \widehat f(\xi)  ( \widehat u(-\,\cdot) \ast \overline{\widehat v}) (\xi) d\xi .	
\end{aligned}
\Ee
In order to get estimate for  $\langle (m^\kappa_\tau (D) f) u,v\rangle$, we first obtain  estimate for $\langle (m^\kappa_\tau (D) P_\lambda  f) u,v\rangle$.

\subsubsection*{Primary decomposition} 
Before breaking the bilinear operator \eqref{littlewood-paley} into fine scales by a Whitney type decomposition, we first decompose  the unit sphere $\Si^1-e_2$   (see \eqref{sigmatau})  into small $\epsilon_\circ$-caps.  Let  $\{ \mathbb S_\ell\}$ be a collection of essentially disjoint\footnote{ $\mathbb S_\ell\cap \mathbb S_{\ell'}$ is of measure zero if $\ell\neq \ell'$.} subsets of $\mathbb S^{d-2}\subset \mathbb R^{d-1}$ such that $\diam(\mathbb S_\ell)\in[\frac{\epsilon_\circ}{10}, \epsilon_\circ]$ and $\mathbb S^{d-2}=\bigcup_{\ell} \mathbb S_\ell$. Thus, 
\[	\big(\Si^1-e_2\big) \times \big(\Si^1-e_2\big) = \bigcup_{\ell, \ell'} ( \{0\}\times \mathbb S_\ell )\times ( \{0\}\times  \mathbb S_{\ell'}).	\]
For the products $\mathbb S_\ell\times \mathbb S_{\ell'}$, we distinguish the following three cases: 
\begin{align}
\label{s-case}
	\textit{transversal} \colon &   \dist(\mathbb S_\ell, \mathbb S_{\ell'})\ge  \epsilon_\circ  \  \text{ and } \   \dist(-\mathbb S_\ell, \mathbb S_{\ell'})\ge  \epsilon_\circ, \\
\label{p-case}
	\textit{neighboring}\colon & \dist(\mathbb S_\ell, \mathbb S_{\ell'})<   \epsilon_\circ,  \\
\label{n-case}
	\textit{antipodal}\colon &  \dist(-\mathbb S_\ell, \mathbb S_{\ell'})<   \epsilon_\circ .  
\end{align}
This leads us to the primary decomposition
\begin{align}
 \label{primary}	& \qquad \langle (m^\kappa_\tau (D) P_\lambda  f) u,v\rangle
	= \sum_{(\mathbb S_\ell,\mathbb S_{\ell'}): \textit{transversal}}\langle (m^\kappa_\tau (D) P_\lambda  f) u_\ell,v_{\ell'}\rangle
		\\& + \sum_{(\mathbb S_\ell,\mathbb S_{\ell'}): \textit{neighboring}}\langle (m^\kappa_\tau (D) P_\lambda  f) u_\ell,v_{\ell'}\rangle
	+ \sum_{(\mathbb S_\ell,\mathbb S_{\ell'}): \textit{antipodal}}\langle (m^\kappa_\tau (D) P_\lambda  f) u_\ell,v_{\ell'}\rangle, \nonumber
\end{align}
where
\begin{equation}\label{uv_support}
	\supp \widehat {u_\ell} -e_2  \subset \{0\}\times \mathbb S_\ell + O(2^{-2}\epsilon_\circ) , \quad 
	\supp \widehat {v_{\ell'}} -e_2\subset \{0\}\times \mathbb S_{\ell'}+  O(2^{-2}\epsilon_\circ) .	
\end{equation}
For derivation of linear restriction estimate from bilinear restriction estimate, it is enough to consider the neighboring case \eqref{p-case} only, since we can decompose  a single function into functions with smaller frequency pieces.  However, in our situation  the functions $u$ and $v$ are completely independent.  So, we  cannot localize the supports of $\widehat u$, $\widehat v$ in such a favorable manner.  

In  \emph{transversal} case \eqref{s-case} we can apply the bilinear restriction estimate directly since the separation (transversality) condition is guaranteed.   For the other two cases \eqref{p-case} and  \eqref{n-case},  the  separation (transversality)  condition fails.  
To apply  the bilinear estimates \eqref{interpolation} and \eqref{interpolation2}, we need to decompose further the sets $\mathbb S_\ell$ and $\mathbb S_{\ell'}$ by making use of a Whitney type decomposition (for example, see \cite{TVV}).   This is to be done in the following section.

\subsection{Estimates with dyadic localization} \label{dylocal} 
To get control over the functions of which  Fourier transforms are confined in a narrow neighborhood of $\Sigma^1$ (recall \eqref{near}),  we first decompose the functions $\wh{u_\ell}$ and $\wh{v_{\ell'}}$ in \eqref{uv_support} dyadically away from $\Sigma^1$, and then break the resulting pieces in the angular directions via the Whitney type decomposition.  We obtain estimates for each piece (Lemma \ref{rescaling}) and combine those estimates  together to get the estimate \eqref{Large} in Proposition \ref{bilinear-sum}.

Let $0 < \delta_2\le \delta_1 \le 2\epc$.  In this section (Section \ref{dylocal})  we work with $u$, $v$ satisfying 
\Be\label{del_sup}	\supp\widehat u -e_2\subset  \{0\} \times \mathbb S_\ell + O(\delta_1), \quad \supp\widehat v-e_2\subset \{0\} \times \mathbb S_{\ell'}+O(\delta_2),	\Ee
under the assumption \eqref{p-case} or \eqref{n-case}.

\subsubsection*{Whitney type decomposition}
Let $j_\circ$ be the smallest  integer such that  $2^{-j_\circ+3}\le 1/\sqrt{d}$, and set  $\mathrm I_\circ=[-2^{-\jc}, 2^{-\jc}]$. For each $j\ge \jc$, we denote  by $\{\mathbf  I_k^j\}$ the collection of the dyadic cubes of side length $2^{-j}$ which are contained in $ \I_\circ^{d-2} $.\footnote{It should be noted that the index $k$ is subject to $j$.}  Fix $j_\ast >\jc+3$. For $\jc<j<\ja$,  $k\sim k'$ means that $ \mathbf  I_k^j $ and  $\mathbf  I_{k'}^j$  are not adjacent but have adjacent parent dyadic cubes.  If $j=\ja$,  by $k\sim k'$ we mean $\dist ( \mathbf  I_k^j, \mathbf  I_{k'}^j)\lesssim 2^{-j}$.   By a Whitney type decomposition of  $ \I_\circ^{d-2}\times  \I_\circ^{d-2}$ away from its diagonal, we have
\[	\I_\circ^{d-2} \times \I_\circ^{d-2}= \bigcup_{\jc < j  \le j_\ast} \Big(\bigcup_{k\sim k'} \mathbf  I_k^j\times \mathbf  I_{k'}^j \Big).  \]
The cubes $\mathbf  I_k^j\times \mathbf  I_{k'}^j$ appearing in the above are essentially disjoint.  Thus we may write 
\[    \rchi_{  \I_\circ^{d-2} \times \I_\circ^{d-2} }=    \sum_{\jc < j  \le j_\ast}  \sum _{k\sim k'}  \rchi_{ \mathbf  I_k^j\times \mathbf  I_{k'}^j }, \]
where $\rchi_{A}$ denotes the indicator function of a set $A$.

Since $\mathbb S_\ell \cup \mathbb S_{\ell'}$ or $(-\mathbb S_{\ell})\cup \mathbb S_{\ell'}$ is contained in $B_{d-1}(\theta, 3\ec)$ for some $\theta\in \mathbb S^{d-2}$,  considering $\I_\circ^{d-2}$ to be placed in the hyperplane $H$ in $\R^{d-1}$ that is orthogonal to $\theta$ and contains the origin (sharing the origin),  there is obviously  a  smooth  diffeomorphism    $\mathcal G:  \I_\circ^{d-2}  \to \mathbb S^{d-2}$\,\footnote{the inverse of the projection from $\mathbb S^{d-2}$ to the plane  $H$} such that  $\mathbb S_\ell \cup \mathbb S_{\ell'} \subset   \mathcal G( \I_\circ^{d-2} )$ in the case of \eqref{p-case} and  $(-\mathbb S_\ell) \cup \mathbb S_{\ell'}  \subset   \mathcal G( \I_\circ^{d-2} )$ in the case of \eqref{n-case}.   We now set 
\[ \mathbb S^j_k:=\mathcal G(\mathbf  I_k^j). \] 
Then, it follows that 
\begin{equation}\label{sdecomp}
\rchi_{ \mathcal G( \I_\circ^{d-2} ) \times \mathcal G( \I_\circ^{d-2} )} =    \sum_{\jc < j  \le j_\ast}  \sum _{k\sim k'}  \rchi_{ \bs_k^j\times\bs_{k'}^j }. 
\end{equation}
If $j<\ja$  we have $\dist ( \bs_k^j, \bs_{k'}^j)\simeq  2^{-j}$ and,   for $j=\ja$,    $\dist ( \bs_k^j, \bs_{k'}^j)\lesssim 2^{-j}$.

The following are rather direct consequences of Lemma \ref{1separation} via scaling. 
\begin{lem}\label{rescaling}
Assume  $d\ge 3$,  $0<h_2\le h_1\ll 1$, and $0<\delta_2\le \delta_1 \le 2^{-2j}\le c$ with a constant  $c$ small enough.  Suppose that $\dist ( \bs_{k}^j, \bs_{k'}^j ) \simeq 2^{-j}$,  and  suppose 
\begin{align*}
	\supp \widehat {u_1} &\subset   \{ \xi\in\R^d:  |\xi_1|\le h_1, \,  \widetilde\xi \in  \pm \, \bs_k^j + O(\delta_1)   \}, \\
	\supp \widehat {u_2} &\subset   \{ \xi\in\R^d:  |\xi_1|\le h_2, \,  \widetilde\xi \in  \bs_{k'}^j + O(\delta_2)  \}.
\end{align*}
Then, the following estimates hold for any $\epsilon>0$ and $\frac{d+1}{2} \le p\le \infty\colon$
\begin{align}
\label{j-sep} 
\| u_1{u_2} \|_{L^{p'}(\Rd)} 
	& \lesssim \dep \, 2^{\frac{j}{p}}\,  h_1^{\frac{1}{p}}\,  (\delta_1\delta_2)^{\frac{d+1}{4p}}\|u_1\|_{L^2(\R^d)} \|u_2\|_{L^2(\R^d)} ,\\
\label{j-sep2}
\| u_1 u_2 \|_{L^{p'} (\mathbb R^d)} 
	&\lesssim \,\dep \,    2^{\frac{j}{p}}   (h_1/h_2)^{\frac{d-3}{4p}}h_2^{\frac1p} \, ( \delta_1 \, \delta_2)^{\frac{d+1}{4p}} \, \|u_1\|_{L^2(\R^d)} \|u_2\|_{L^2(\R^d)}. 
\end{align}
\end{lem}

\begin{proof}
Note that if $\supp\widehat {u_1} \subset [-h_1,h_1]\times(-\mathbb S_k^j +O(\delta_1))$, then $\supp\widehat {\overline{u_1}} \subset [-h_1,h_1]\times (\mathbb S_k^j +O(\delta_1) )$. Since $\| u_1 u_2 \|_{L^{p'}} =\| \overline {u_1}u_2 \|_{L^{p'}} $, we need only  to consider  the case 
\[ \supp \widehat {u_1} \subset   \{ \xi\in\R^d:  |\xi_1|\le h_1, \,  \widetilde\xi \in \, \bs_k^j + O(\delta_1)   \}. \] 
We first observe that  the supports of $\wh {u_1}$ and  $\wh {u_2} $ are contained in the set  $\{ \xi\in\R^d:  |\xi_1|\le h_1, \,  |(\widetilde\xi -\theta)\cdot \theta|\lesssim 2^{-2j} , \,  |\widetilde\xi -\theta -((\widetilde\xi -\theta)\cdot \theta) \theta |\lesssim 2^{-j} \}$ for some  $\theta\in \mathbb S^{d-2}$.  After rotation in $\widetilde\xi$ we may assume that  $\theta=-\widetilde {e_2}$, and  by translation $\widetilde\xi \to \widetilde\xi+\widetilde{e_2}$ (since these changes of variables do not affect the estimates)  we may assume that 
\[	\supp \widehat {u_i} \subset \big \{(\xi_1, \xi_2,\bar \xi):  |\xi_1|\le h_i, \,  \xi_2=\psi(\bar\xi)+O(\delta_i), \,  |\bar \xi|\lesssim 2^{-j} \big\}, \quad i=1,2,	\]
where $\psi(\bar\xi) = 1-( 1-|\bar \xi|^2)^{1/2} $.  Note that  $\supp \wh{u_1}$ and $\supp \wh{u_2}$ are separated by $\simeq 2^{-j}$.  By an anisotropic dilation  $(x_2, \bar x)\to (2^{2j}x_2, 2^{j}\bar x)$ we see that
\Be \label{this} 
	\| u_1 u_2 \|_{L^{p'}(\mathbb R^d)} = 2^{\frac{d}{p'}j}\|  \fu_1 \fu_2 \|_{L^{p'}(\Rd) }, 
\Ee
where $\fu_i (x_1, x_2, \bar  x)=u_i(x_1, 2^{2j}x_2, 2^j\bar x)$ for $i=1,2$.  Then it  follows that, for $i=1,2$,  
\[	
\supp \widehat {\fu_i}\subset \big\{ (\xi_1, \xi_2, \bar\xi ):  |\xi_1|\le h_i, \,  \xi_2= 2^{ 2j} \psi(2^{-j} \bar\xi)
+O(\delta_i 2^{2j}) , \,  |\bar\xi| \lesssim {1}\big\}	
\]
while  $\supp \widehat{\fu_1}$ and $\supp \widehat{\fu_2}$ are separated by $\simeq 1$.   Since $2^{2j} \psi(2^{-j}\bar \xi) = \frac12|\bar\xi|^2 + O(2^{-2j})$ by the Taylor expansion,  we see   $2^{2j} \psi(2^{-j}\bar \xi)$ is of elliptic type $(C2^{-2j},N)$ for some $N$.  Thus, for $j$  large enough\footnote{Otherwise,  we may directly use the bilinear estimate without rescaling since the surfaces are well separated.} we may apply  Lemma \ref{1separation}.   Rescaling gives 
\[  \|  \fu_1\fu_2 \|_{L^{p'}(\mathbb R^d) }  \lesssim  2^{-2\epsilon j}\dep 
	h_1^{\frac1p} \, ( 2^{2j} \delta_1 \, 2^{2j} \delta_2)^{\frac{d+1}{4p}} \,   2^{-dj}\|u_1\|_{L^2} \|u_2\|_{L^2}.	\]
Combining this with \eqref{this} we get the desired estimate  \eqref{j-sep}.  The same argument with \eqref{interpolation2} gives  \eqref{j-sep2}. So we omit the details. 
\end{proof}

\subsubsection*{Location  of $\pm\mathbb S_k^j+ \mathbb S_{k'}^j$ for $k\sim k'$}
Let us denote by $c(\bs_k^j)$ the barycenter of $\bs_k^j$ and set 
\[	c_{k, k'}^{j, \pm}=\mp c(\bs_k^j) +  c(\bs^j_{k^\prime}).\] 
Note that every $\bs_k^j$ is contained in a rectangle of dimensions about $2^{-2j}\times \overbrace{2^{-j}\times \dots \times 2^{-j}}^{d-2 \text{ times }} $.

For each $j$ we  observe that, if $k\sim k'$, 
\Be\label{thetas}
	-\bs^j_k+ \bs^j_{k^\prime}\subset  \mathcal R_{k,k'}^{j,+},
\Ee
where 
\begin{equation*} \label{Rj}
\mathcal R_{k,k'}^{j,+} 
	 =\left\{ \widetilde\xi \colon  \left | \left\langle \widetilde\xi- c_{k, k'}^{j,+}, \frac{c_{k, k'}^{j,-}}{|c_{k, k'}^{j,-}|}\right\rangle \right |\lesssim  2^{-2j},  \,
		\left |\widetilde\xi- c_{k, k'}^{j,+}-\left\langle  \widetilde\xi- c_{k, k'}^{j,+}, \frac{c_{k, k'}^{j,-}}{|c_{k, k'}^{j,-}|} \right\rangle  \frac{c_{k, k'}^{j,-}}{|c_{k, k'}^{j,-}|} \right |\lesssim  2^{-j}  \right\}.
\end{equation*} 
Thus, for    ${j_\circ}<j< j_\ast$, we have 
\Be \label{thetatheta} 
	\bigcup_{k\sim k^\prime}(-\bs^j_k+ \bs^j_{k^\prime})
		 \subset    B_{d-1}(0,C_1 2^{-j})\setminus  B_{d-1}(0,C_2 2^{-j})
		 \Ee
with some $C_1, C_2>0$ (see Figure \ref{fig_conv_near}). For $j=j_\ast$, since there is no separation between $ \bs^j_k$ and $\bs^j_{k^\prime}$, we just have 
\Be \label{thetatheta1} 
	\bigcup_{k\sim k^\prime}(-\bs^\ja_k+ \bs^\ja_{k^\prime})
		 \subset    B_{d-1}(0,C2^{-\ja})\Ee
for some $C>0$. 

\usetikzlibrary{calc}
\def\centerarc[#1](#2)(#3:#4:#5)
    { \draw[#1] ($(#2)+({#5*cos(#3)},{#5*sin(#3)})$) arc (#3:#4:#5); }
\begin{figure}
\captionsetup{type=figure,font=footnotesize}
\centering
\begin{tikzpicture}[scale=0.9]\scriptsize
	\centerarc[black](0,0)(0:360:3)	
	\centerarc[red](0,0)(0:360:0.18)
	\centerarc[red](0,0)(0:360:0.9)
	\centerarc[blue, line width=1mm](0,0)(30:35:3)
		\draw (0.22,0.2)+({3*cos(30)}, {3*sin(30)}) node{$\bs_{k}^j$};
	\centerarc[blue, line width=1mm](0,0)(210:215:3)
		\draw ({3*cos(210)}, {3*sin(210)})+(-0.3,-0.2) node{$-\bs_{k}^j$};
	\draw [dash pattern={on 2pt off 1pt}, line width=0.05mm]({3*cos(32.5)}, {3*sin(32.5)})--({3*cos(212.5)}, {3*sin(212.5)});
	\draw[fill] (0,0) circle [radius=0.03] node[below]{$0$};
	\centerarc[purple, line width=1mm](0,0)(40:45:3)
		\draw (0.1,0.3)+({3*cos(40)}, {3*sin(40)}) node{$\bs_{k'}^j$};
	\draw[dash pattern={on 2pt off 1pt}, line width=0.05mm] ({3*cos(42.5)}, {3*sin(42.5)})--(0,0);
	\draw[line width=3.5mm, orange] ({5*cos(212.5)+5*cos(42.5)}, {5*sin(212.5)+5*sin(42.5)}) -- ({cos(212.5)+cos(42.5)}, {sin(212.5)+sin(42.5)});
	\draw[fill] ({3*cos(212.5)}, {3*sin(212.5)})+({3*cos(42.5)}, {3*sin(42.5)}) circle [radius=0.03] node[right]{$c_{k,k'}^{j,+}$};
	\draw ({3*cos(212.5)}, {3*sin(212.5)})+({3*cos(42.5)}, {3*sin(42.5)}) node[above left]{$\mathcal R_{k,k'}^{j,+}$};
	\draw[dash pattern={on 2pt off 1pt}, line width=0.05mm] ({3*cos(212.5)}, {3*sin(212.5)})-- ({3*cos(212.5)+3*cos(42.5)}, {3*sin(212.5)+3*sin(42.5)}) -- ({3*cos(42.5)}, {3*sin(42.5)});
\end{tikzpicture}\caption{ The neighboring case \eqref{p-case}:  the point  $c_{k,k'}^{j,+}=-c(\bs_k^j) + c(\bs^j_{k^\prime})$ and the set $\mathcal R_{k,k'}^{j,+}\subset \R^{d-1}$ (the orange rectangle).}
\label{fig_conv_near}
\end{figure}
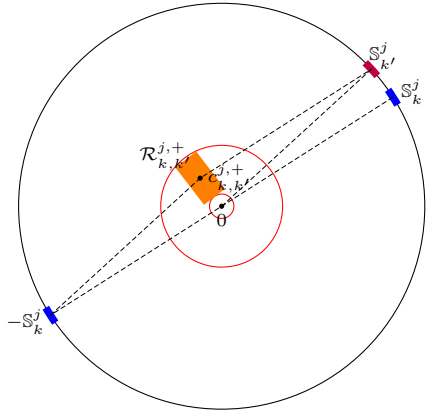

On the other hand, since $\dist ( \bs_k^j, \bs_{k'}^j)\lesssim 2^{-j}$ whenever $k\sim k'$,  we note that if $k\sim k'$ 
\Be\label{thetas_n}
	\bs^j_k+ \bs^j_{k^\prime}\subset  \mathcal R_{k,k'}^{j,-},
\Ee
where 
\begin{equation} \label{Rj_n}
\mathcal R_{k,k'}^{j,-} =\left\{ \widetilde\xi\colon  \left |\left\langle \widetilde\xi- c_{k, k'}^{j,-}, \frac{c_{k, k'}^{j,-}}{|c_{k, k'}^{j,-}|}\right\rangle\right |\lesssim 2^{-2j},  \,
	\left |\widetilde\xi- c_{k, k'}^{j,-}-\left\langle  \widetilde\xi- c_{k, k'}^{j,-}, \frac{c_{k, k'}^{j,-}}{|c_{k, k'}^{j,-}|} \right\rangle  \frac{c_{k, k'}^{j,-}}{|c_{k, k'}^{j,-}|} \right |\lesssim  2^{-j}  \right\}.
\end{equation} 

Clearly, $ \mathcal R_{k,k'}^{j,-} $ is contained in a $C2^{-2j}$-neighborhood of $2\mathbb S^{d-2}$ provided that  $k\sim k'$ (see Figure \ref{fig_conv_anti}).  We also see that  for every $j$ with ${j_\circ}<j\le j_\ast$,
\Be \label{thetatheta-} 
\bigcup_{k\sim k^\prime}(\bs^j_k+ \bs^j_{k^\prime})
	 \subset 2\mathbb S^{d-2} + O( C2^{-2j}).  
\Ee

Let us denote by $\phi_{k,k'}^{j,\pm}$ a smooth function adapted to  $\mathcal R_{k,k'}^{j,\pm}$ such that $\phi_{k,k'}^{j,\pm}$ is supported in the rectangle given by dilating  $\mathcal R_{k,k'}^{j,\pm}$ twice  from its center  and  $\phi_{k,k'}^{j,\pm}=1$  on $\mathcal R_{k,k'}^{j,\pm}$.  We also define the projection operator   $P_{k, k'}^{j,\pm} $ by 
\Be \label{projkk}
	\F\big(  P_{k, k'}^{j, \pm} g\big)(\xi) =\phi_{k,k'}^{j,\pm}(\widetilde\xi\,)\, \widehat g(\xi).
\Ee

In what follows we prove bilinear estimates which are the key ingredients for the main estimates. It  will be done by considering  the  three cases \eqref{p-case}, \eqref{n-case}, and \eqref{s-case}, separately.  For a unit vector $e\in \mathbb S^{d-1}$ and $\de>0$ let $P_{\le \delta}^{e}$ be the Littlewood-Paley projection in the $e$-direction which is defined by 
\[	\mathcal F \big(P_{\le \delta }^{e}  g\big)(\xi) = \beta_0\bigg(\frac{\xi\cdot e}{\delta}\bigg) \widehat g(\xi) \] 
with $\beta_0\in C^\infty_c ((-4,4))$ satisfying $\beta_0=1$ on $[-2,2]$.

\begin{figure}[t]
\captionsetup{type=figure,font=footnotesize}
\centering
\begin{tikzpicture}[scale=0.9]\scriptsize
	\centerarc[black](0,0)(0:360:3)
	\centerarc[red](0,0)(-27:27:6.5)
	\centerarc[red](0,0)(-30:30:5.8)
	\centerarc[blue, line width=1mm](0,0)(-2.5:2.5:3)
		\draw (3.3, -0.2) node{$\bs_{k}^j$};
	\centerarc[blue, line width=1mm](0,0)(180-2.5:180+2.5:3)
		\draw (-3.3,0) node{$-\bs_{k}^j$};
	\draw[fill] (0,0) circle [radius=0.03] node[below]{$0$};
	\centerarc[purple, line width=1mm](0,0)(7.5:12.5:3)
		\draw (0.3,0.2)+({3*cos(10)}, {3*sin(10)}) node{$\bs_{k'}^j$};
	\draw[dash pattern={on 2pt off 1pt}, line width=0.05mm] (3, 0)+({3*cos(10)}, {3*sin(10)}) -- ({3*cos(10)}, {3*sin(10)})--(0,0);
	\draw[line width=6mm, orange] ({3.07+3.07*cos(10)}, {3.07*sin(10)}) -- ({2.92+2.92*cos(10)}, {2.92*sin(10)});
		\draw[fill] (3, 0)+({3*cos(10)}, {3*sin(10)}) circle [radius=0.03] node[right]{$c_{k,k'}^{j,-}$};
		\draw (3, 0.25)+({3*cos(10)}, {3*sin(10)}) node[above]{$\mathcal R_{k,k'}^{j,-}$};
	\draw [dash pattern={on 2pt off 1pt}, line width=0.05mm](3, 0)+({3*cos(10)}, {3*sin(10)})--(3,0)--(-3,0);
\end{tikzpicture}\caption{The antipodal case \eqref{n-case}:  the point $c_{k,k'}^{j,-} =c(\bs_k^j) + c(\bs^j_{k^\prime})$ and the set $\mathcal R_{k,k'}^{j,-}\subset \R^{d-1}$ (the orange rectangle).}
\label{fig_conv_anti}
\end{figure}
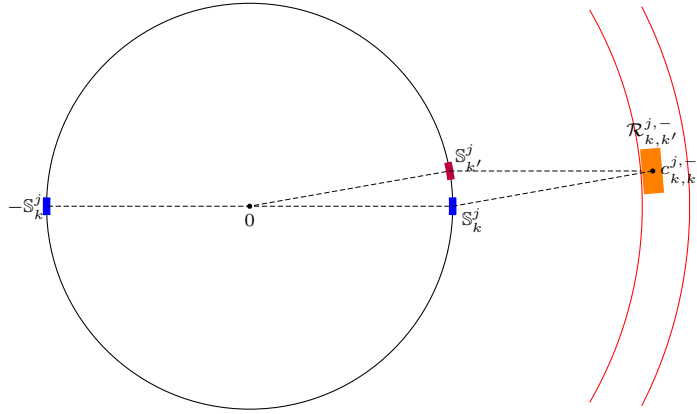

\subsection{Estimates for the neighboring case}\label{sec_neigh}
In this section (Section \ref{sec_neigh}) we consider the neighboring  case in the primary decomposition \eqref{primary}.  So, $\mathbb S_\ell$, $\mathbb S_{\ell'}$ satisfy \eqref{p-case}, and $u$, $v$ satisfy \eqref{uv_support}  in place of $u_\ell$, $v_{\ell'}$.

\subsubsection*{{Decomposition of  $u$ and $v$}}
Let us define $u_k^j$ and $v_{k}^j$ by 
\Be\label{decom_result}
	\F (u_k^j) (\xi) = \rchi_{\bs_k^j}\bigg(\frac{\widetilde\xi - \widetilde{e_2}}{|\widetilde\xi - \widetilde{e_2}|}\bigg) \widehat u(\xi), \quad \F ( v_{k}^j )(\xi) = \rchi_{\bs_{k}^j}\bigg(\frac{\widetilde\xi - \widetilde{e_2}}{|\widetilde\xi - \widetilde{e_2}|}\bigg) \widehat v(\xi).
\Ee 
Then by \eqref{sdecomp} it follows that  
\begin{align}\label{whitney}
\langle (m^\kappa_\tau (D) P_\lambda  f) u,v\rangle  =\sum_{\jc < j  \le j_\ast}  \sum _{k\sim k'}   \langle (m^\kappa_\tau (D) P_\lambda  f) u_k^j,v_{k'}^j \rangle =   \sum_{\jc < j  \le j_\ast}  \mathcal I_{j},
\end{align}
where
\Be\label{jjkk}
	\mathcal I_j= \sum_{k\sim k'} \langle (m^\kappa_\tau (D) P_\lambda  f) u_k^j, {v_{k'}^j}\rangle, \quad  \jc<j\le  j_\ast.
\Ee
For $0<\delta\le 2\ec$, we denote by $j_\ast(\delta)$ the largest integer $j_\ast$ satisfying $\sqrt{\de}\le 2^{-j_\ast}$.  For $\lambda,\, \delta>0$ and $f\in L^p(\R^d)$, let us set 
\[	\Gamma_{\lambda,\delta}^{p, +}(f) = 
		\sup_{\jc< j \le  j_\ast(\delta)} \sup_{k\sim k'} \left\{ \lambda^\frac1p \delta^{-\frac1p} \|P_{k, k'}^{j,+}  P_{\le \delta}^{e_1} P_\lambda f\|_{L^p(\R^d)} \right\}.	\]
\begin{lem}\label{bilinear+} 
Let $d\ge3$, $p\ge \frac{d+1}2$, $\frac1\tau\le\lambda \lesssim 1$,  $0 < \delta_2\le \delta_1 \le 2\epc$, 
and let $f\in L^p(\mathbb R^{d})$. 
Assume that $\mathbb S_\ell$, $\mathbb S_{\ell'}$ satisfy \eqref{p-case} and   $u, $ $v$ satisfy \eqref{del_sup}. 
Then,  for any  $\epsilon>0$, 
\begin{align}
\label{large}  
	\big|\langle (m^\kappa_\tau (D) P_\lambda  f) u,v\rangle\big| 
		&\lesssim \dep \lambda^{\kappa-\frac2p}   
		\delta_1^{\frac{d+9 }{4p} } 		\delta_2^{\frac{d+1 }{4p} }   \Gamma_{\lambda, \delta_1}^{p, +}(f) \|u\|_{L^2(\R^d)} \|v\|_{L^2(\R^d)},    \\ 
\label{large2}
	\big|\langle (m^\kappa_\tau (D) P_\lambda  f) u,v\rangle\big| 
		&\lesssim \dep \lambda^{\kappa-\frac2p}  
		\delta_1^{\frac{d+1 }{2p} } 
		\delta_2^{\frac{2 }{p} }   \Gamma_{\lambda, \delta_1}^{p, +}(f) \|u\|_{L^2(\R^d)} \|v\|_{L^2(\R^d)}. 
\end{align}
\end{lem} 

\begin{proof}  In the decomposition \eqref{whitney}  we take $j_\ast=j_\ast(\delta_1)$.  Since $u, $ $v$ satisfy \eqref{del_sup}, we see that $\supp \F u_k^j \subset \big(\{0\}\times \mathbb S_k^j+e_2\big) +O(\delta_1)$ and $\supp \F v_{k'}^j \subset \big(\{0\}\times \mathbb S_{k'}^j+e_2\big) +O(\delta_2)$. Hence,  we have
\[\supp \big(\F u_k^j(-\,\cdot)\ast \overline{\F v_{k'}^j}\big) \subset \big(\{0\}\times (-\mathbb S_k^j+\mathbb S_{k'}^j)\big) +O(2\delta_1) \subset \{|\xi_1|\lesssim \delta_1\} \times \mathcal R_{k,k'}^{j,+}, \]
where the last inclusion follows from \eqref{thetas} since $2^{-2j}\ge \delta_1$  for $j\le \ja$. This observation enables us to insert the projection operator  $P_{k, k'}^{j,+}$   to write 
\Be \label{decomp2}
	\mathcal I_j = \sum_{k\sim k' } \langle  (P_{k, k'}^{j,+}    m^\kappa_\tau (D) P_\lambda  f)  u_k^j, v_{k'}^j \rangle.
\Ee
We first obtain estimates for each  $\mathcal I_j$.  By  \eqref{thetatheta} and \eqref{thetatheta1},  it follows   that  $\mathcal I_j\neq 0$ only if $\lambda\lesssim 2^{-j}$.  From   \eqref{del_sup} we see
\[ \langle  (P_{k, k'}^{j,+} m^\kappa_\tau (D) P_\lambda  f)  u_k^j, v_{k'}^j \rangle = \langle  (m^\kappa_\tau (D) P_{k, k'}^{j,+}  P_{\le \delta_1}^{e_1} P_\lambda f)  u_k^j, v_{k'}^j \rangle.\]
Thus, H\"older's inequality and Lemma \ref{Lit_Pal} below give us that for every $\jc<j< j_\ast$ and $k\sim k'$,
\[	 |\langle  (P_{k, k'}^{j,+} m^\kappa_\tau (D) P_\lambda    f)  u_k^j, v_{k'}^j \rangle |
	\lesssim \lambda^\kappa  \| P_{k, k'}^{j,+} P_{\le \delta_1}^{e_1}  P_\lambda f\|_{L^p(\R^d)} \, \| u_k^j\, \overline{ v_{k'}^j} \|_{L^{p'}(\Rd)} .	\]
As already mentioned, by \eqref{thetatheta}  we may assume   $ \lambda \lesssim  2^{-j}$ since  $P_{k, k'}^{j,+}P^{e_1}_{\le \delta_1} P_{\lambda} =0$ otherwise.  Translation in the frequency space does not have any effect on the estimates, so we may apply \eqref{j-sep} in Lemma \ref{rescaling} (with $h_1 = \delta_1$ and $h_2 =\delta_2$)  to get, for $\jc<j< \ja$, 
\begin{equation}\label{jjast}
	\| u_k^j\, v_{k'}^j \|_{L^{p'}(\Rd)}  \lesssim  \dep  2^{\frac{j}{p}} \delta_1^{\frac{d+5 }{4p} } \delta_2^{\frac{d+1}{4p} }     \|u_k^j\|_{L^2} \|v_{k'}^j\|_{L^2} .
\end{equation}
Combining these estimates with \eqref{decomp2}, we have,  for $\jc<j< \ja$,
\[ |\mathcal I_j| \lesssim \dep \lambda^\kappa
          2^{\frac j p}\delta_1^{\frac{d+5 }{4p} } \delta_2^{\frac{d+1}{4p} } \sup_{k\sim k'} \| P_{k, k'}^{j,+} P_{\le \delta_1}^{e_1}  P_\lambda f\|_{L^p} \sum_{k\sim k'} 
   \|u_k^j\|_{L^2}\|v_{k'}^j\|_{L^2}.
   \] 
Since $\lambda \lesssim  2^{-j}$ and $\sum_{k\sim k'} \|u_k^j\|_{L^2}\|v_{k'}^j\|_{L^2}\lesssim \|u\|_{L^2}\|v\|_{L^2}$, for $\jc<j< \ja$ it follows that 
\begin{align}\label{q_whi_n}		
|{\mathcal I_j} |  
 \lesssim
\dep  \lambda^{\kappa-\frac{2}p}  \delta_1^{\frac{d+9 }{4p} } \delta_2^{\frac{d+1}{4p} } 
         \Gamma_{\lambda,\delta_1}^{p, +} (f)  \|u \|_{L^2}\|v \|_{L^2} .	\end{align}
Since $j_\ast \simeq \log \de_1^{-1}\lesssim \delta_2^{-\epsilon}$, summation along $j$ gives 
\Be\label{sumjstar}
\sum_{\jc<j<j_\ast}|{\mathcal I_j} | 
	\lesssim \dep \lambda^{\kappa-\frac2{p}}  \delta_1^{\frac{d+9 }{4p} } \delta_2^{\frac{d+1 }{4p} }   \Gamma_{\lambda, \delta_1}^{p, +}(f)  \|u\|_{L^2} \|v\|_{L^2}.	
\Ee
Now we consider  $\mathcal I_{j_\ast}$ in \eqref{whitney}. Since there is no separation between the supports of $\mathcal F(u_k^{j_\ast})$, $\mathcal F(v_{k'}^{j_\ast})$, the bilinear restriction estimates are no longer available. Instead, we use more elementary argument which relies on Lemma \ref{lin}.  In fact, as is clear for the argument in the above,  it is sufficient to  show 
\[ \|  u_k^\ja  v_{k'}^\ja \|_{L^{p'}(\Rd)}  \lesssim  \lambda^{-\frac1p} 
 \delta_1^{\frac{d+5 }{4p} } \delta_2^{\frac{d+1}{4p} }     \| u_k^\ja \|_{L^2} \| v_{k'}^\ja\|_{L^2} , \]
which plays the role of \eqref{jjast}.  For the purpose we may assume $\lambda\lesssim 2^{-\ja}$, otherwise $P_{k,k'}^{\ja,+}P_\lambda=0$.  Since  $2^{-j_\ast} \simeq \sqrt{\delta_1}$ we have $1\lesssim \lambda^{-\frac1p} \delta_1^\frac1{2p}$. Thus, we only need to show that,  for $p\ge \frac{d+1}2$, 
\Be\label{final_scale}
	\| u_k^\ja v_{k'}^\ja  \|_{L^{p'}(\R^d)} \lesssim   \delta_1^{\frac{d+3 }{4p} } \delta_2^{\frac{d+1}{4p} }     \|u_k^\ja\|_{L^2}\|v_{k'}^\ja\|_{L^2} .
	\end{equation}
The estimate is trivial with $p=\infty$. By interpolation it is sufficient to show \eqref{final_scale} with $p=\frac{d+1}2$.  We note that  $\F u_k^{j_\ast}$ is supported in a rectangle of dimensions  approximately 
\[	\delta_1\times\delta_1\times \overbrace{\sqrt{\delta_1} \times \dots \times \sqrt{\delta_1}}^{ d-2 \text{ times}}.	\]
Let $r$ be the number such that $\frac{d-1}{d+1}=\frac1r+\frac{d-2}{2d}$. By H\"older's inequality  $\| u_k^\ja  v_{k'}^\ja  \|_{L^\frac{d+1}{d-1} }\le \|  u_k^\ja  \|_{L^r}  \|  v_{k'}^\ja \|_{L^\frac{2d}{d-2}}$. Then, applying Bernstein's inequality and Lemma  \ref{lin} to $u_k^\ja$ and $v_{k'}^\ja$, respectively,   we have 
\Be\label{bernstein} 
	 \| u_k^\ja  v_{k'}^\ja  \|_{L^\frac{d+1}{d-1}}\lesssim \delta_1^{\frac{d+2}2(\frac12-\frac1r)}  \delta_2^{\frac{d+2}{2d}}\|u_k^\ja\|_{L^2}\|v_{k'}^\ja\|_{L^2} = \delta_1^{\frac{d+3}{2(d+1)}-\frac1d} \delta_2^{\frac{d+2}{2d}} \|u_k^\ja\|_{L^2}\|v_{k'}^\ja\|_{L^2}.  
\Ee
Since $\delta_2\le \delta_1$, this gives the desired \eqref{final_scale} with $p=\frac{d+1}2$. 

We now prove \eqref{large2}.  We claim that,   if $\jc< j\le j_\ast $ and $k\sim k'$,  \begin{equation}
\label{secondcase0}
	\| u_k^j v_{k'}^j \|_{L^{p'}(\Rd)}  \lesssim \dep 2^\frac{j}{p} \delta_1^{\frac{d-1 }{2p} } \delta_2^{\frac{2}{p} }  \|u_k^j\|_{L^2} \|v_{k'}^j\|_{L^2}.	
\end{equation}
Once we have \eqref{secondcase0},  by repeating the same argument as in the above we get  the  bound
\[	\sum_{\jc<j\le j_\ast}| \mathcal I_j | 
	\lesssim \dep \lambda^{\kappa-\frac2p} \delta_1^\frac{d+1}{2p} \delta_2^\frac{2}{p}  
	\Gamma_{\lambda, \delta_1}^{p, +}(f)  \|u\|_{L^2} \|v\|_{L^2}.	 \]
So, it remains to show \eqref{secondcase0} for $\jc< j\le j_\ast$. If $\jc< j< j_\ast$ \eqref{secondcase0} follows by the estimate \eqref{j-sep2} (in Lemma \ref{rescaling}) with $h_2=\delta_2$, $h_1=\delta_1$. Thus, it  remains to show  \eqref{secondcase0} with $j=j_\ast$.

As before, the case $j=\ja$ is handled differently because there is no separation between the supports of  $\mathcal F(u_k^{j_\ast})$, $\mathcal F(v_{k'}^{j_\ast})$. In fact, we  claim that 
\begin{equation}
\label{secondcase}
	\| u_k^\ja v_{k'}^\ja \|_{L^{p'}(\mathbb R^d)}  \lesssim \delta_1^{\frac{d-2 }{2p} } \delta_2^{\frac{2}{p} }  \| u_k^\ja\|_{L^2} \| v_{k'}^\ja\|_{L^2}.
\end{equation}
Since the estimate  is trivial with $p=\infty$, by interpolation it is enough to show  \eqref{secondcase} with $p=\frac{d+1}2$.  For simplicity we set 
	\[\mathbf {u}= u_k^\ja, \quad  \bv = v_{k'}^\ja , \] 
and  we use the argument for \eqref{interpolation2}.  Let $\{I_\ell\}$ be a collection of essentially disjoint intervals  of length $\simeq \delta_2$ such that  $I_\ell\subset [- \delta_1,  \delta_1]$, and $[- \delta_1,  \delta_1]=\bigcup_{\ell} I_\ell$.  Then, we have 
\[ \bu\bv= \sum_\ell \bu^\ell \bv,    \quad  \bu^\ell:= \F^{-1} \big( \chi_{I_\ell}(\xi_1) \widehat{\bu}(\xi)\big).\] 
Since $\F( \bu^\ell \bv)$ is supported in $(I_\ell+[-\delta_2,\delta_2])\times \mathbb R^{d-1}$, by \cite[Lemma 6.1]{TVV}, and successively applying Bernstein's inequality, Minkowski's inequality  and  H\"older's inequality we see that 
\begin{equation}\label{secondcase1}
	\begin{aligned}
	\B\|\sum_\ell \bu^\ell \bv \B\|_{L^{p'} (\mathbb R^d)}
		&\lesssim \B(\sum_{\ell}\| \bu^\ell \bv \|_{L^{p'} (\mathbb R^d)}^{p'} \B)^\frac1{p'} 
		\lesssim \delta_2^{\frac1p}		\B(\sum_{\ell}\| \bu^\ell(x_1, \cdot) \bv(x_1, \cdot) \|_{L_{x_1}^1L^{p'}_{\widetilde x}}^{p'} \B)^\frac1{p'} \\	
		&\lesssim \delta_2^{\frac1p}(\delta_1/\delta_2)^{\frac{1}{2}-\frac1p }
		\B(\sum_{\ell}\| \bu^\ell(x_1, \cdot) \bv(x_1, \cdot) \|_{L_{x_1}^1L^{p'}_{\widetilde x}}^{2} \B)^\frac1{2} .	
	\end{aligned}
\end{equation}
Let $r$ be the number such that $\frac{d-1}{d+1}=\frac1r+\frac{d-2}{2d}$.  By H\"older's inequality 
\[ \| \bu^\ell(x_1, \cdot) \bv(x_1, \cdot)  \|_{L^{\frac{d+1}{d-1}} (\R^{d-1})}\le \|   \bu^\ell(x_1, \cdot)  \|_{L^r(\R^{d-1})}  \|   \bv(x_1, \cdot) \|_{L^{\frac{2d}{d-2}}(\R^{d-1})}.\]
Since $\mathcal F(\bu^\ell(x_1, \cdot)) $ is supported in a rectangle of dimensions about $\delta_1\times \overbrace{\sqrt{\delta_1} \times \dots \times \sqrt{\delta_1}}^{ d-2 \text{ times}}$ and  $\mathcal F(\bv(x_1, \cdot)) $  is supported in $\wt{e_2}+ \mathbb S^{d-2} + O(\delta_2)$,  applying Bernstein's inequality to $\|\bu^\ell(x_1, \cdot)\|_{L^r}$ and Corollary \ref{d-nbd} to $\|\bv(x_1, \cdot)\|_{L^\frac{2d}{d-2}}$,   we have 
\[ \| \bu^\ell(x_1, \cdot) \bv(x_1, \cdot)  \|_{L^{\frac{d+1}{d-1}} (\R^{d-1}) }\lesssim  \delta_1^{\frac d2(\frac12-\frac1r)} \delta_2^\frac12
 \| \bu^\ell(x_1, \cdot)\|_{L^2  (\R^{d-1})}\|\bv(x_1, \cdot)  \|_{L^2 (\R^{d-1})}.\] 
Putting this in \eqref{secondcase1} with $p=\frac{d+1}2$,  we have 
\begin{align*}
\|  \bu \bv \|_{L^\frac{d+1}{d-1}} 
	\lesssim \delta_1^{\frac{d-2}{d+1}}\delta_2^{\frac{4}{d+1}} \B(\sum_{\ell}\big\| \|\bu^\ell(x_1, \cdot)\|_{L^2}\| \bv(x_1, \cdot)\|_{L^2} \big\|_{L_{x_1}^1}^2 \B)^\frac1{2} 
	\le \delta_1^{\frac{d-2}{d+1}}\delta_2^{\frac{4}{d+1}} \|\bu\|_{L^2}\|\bv\|_{L^2}.
\end{align*}
Thus we get \eqref{secondcase} with $p=\frac{d+1}2$.
\end{proof}

\begin{lem}\label{Lit_Pal}
If $\kappa\ge 0$ and $\lambda\gtrsim 1/\tau$, then 
\[	\|m^{\kappa}_\tau(D) P_\lambda f\|_{L^p(\R^d)} \le C\lambda^\kappa \|P_\lambda f\|_{L^p(\R^d)}	\]
with $C$ independent of $\lambda$ and $\tau$.
\end{lem}
\begin{proof}
By scaling and Young's inequality, it suffices to show that
\[	\Big\| \int m_{\tau\lambda}^\kappa (\xi) \bar{\beta}(|\xi|) e^{ix\cdot\xi} d\xi \Big \|_{L^1(\R^d; dx)} \lesssim 1, \]
where $\bar\beta\in C_c^\infty ((1/4, 4))$ such that $\bar \beta \beta= \beta$. Since $(\lambda\tau)^{-1}\lesssim 1$, this readily follows from integration by parts and \eqref{prop_symb}.
\end{proof}

\subsubsection*{Estimates in  ${X}_{\zeta(1), 1/\tau}^b$} We define a function  $\beta_d: [ \frac{d+1}{2},\infty]\to  [0,1]$  as follows. For $3\le d\le 6,$
\begin{align*} \beta_d(p) &  = \begin{cases}
\,	1- \frac{d+5}{2p}  &\text{if }\ \  p \ge d+1,  \\[5pt] 
\,	\frac12 - \frac2p   &\text{if }\ \   d+1 > p\ge   4 ,   \\[5pt] 
\,	 0  &\text{if }\ \  4 > p\ge  \frac{d+1}{2}, 
\end{cases}
\intertext{and,  for $d\ge7$,}
 \beta_d(p) &  =
\begin{cases} 
\, 1- \frac{d+5}{2p}  &\text{if }\ \  p \ge \frac{d+9}{2}, \\[5pt]  
\, \frac12- \frac{d+1}{4p}    &\text{if }\ \   \frac{d+9}{2} > p\ge \frac{d+1}{2} .
\end{cases}
\end{align*}

\begin{prop} \label{bilinear-P}  
Let $d\ge3$, $\tau\gtrsim \epsilon_\circ^{-1}$, $\frac1\tau \le \lambda \lesssim 1$, $p\ge\frac{d+1}{2}$, and let $f\in L^p(\mathbb R^{d})$. Suppose that $\mathbb S_{\ell}$ and $\mathbb S_{\ell'}$ satisfy \eqref{p-case}, and that $u$, $v$ satisfy \eqref{uv_support} in place of $u_\ell$, $v_{\ell'}$, respectively. Then, for  any $\epsilon>0$,  there is a constant  $C=C(\epsilon,p,d)>0$, independent of $\tau$, $\lambda$, $f$, $u$, $v$, $\mathbb S_\ell$, and $\mathbb S_{\ell'}$, such that 
\Be\label{P-est}	
| \langle ({m^\kappa_\tau (D) P_\lambda  f}) u,v\rangle | 
	\le C \lambda^{\kappa-\frac{2}p} \tau^{\beta_d(p)+\epsilon} \sup_{\frac1\tau\le\delta\lesssim1}   \Gamma_{\lambda, \delta}^{p, +}(f)   \|u\|_{X_{\zeta(1), 1/\tau}^{1/2}}\|v\|_{X_{\zeta(1), 1/\tau}^{1/2}}.
\Ee
\end{prop} 
 
\begin{proof}
Let $\delta_\ast$ be the dyadic number such that  $\tau^{-1}\le \delta_\ast < 2\tau^{-1}$. We begin with decomposing $u$ and $v$ as follows: 
\begin{align}
\label{dyadic1}
	u &=\sum_{\delta_\ast\le \delta: \text{dyadic}} u_\delta:= Q^{\,1}_{\le \delta_\ast}u+   \sum_{\delta_\ast< \delta: \text{dyadic}} Q^1_\delta  u,  \\
\label{dyadic2}
	v &=\sum_{\delta_\ast\le \delta: \text{dyadic}} v_\delta:= Q^{\,1}_{\le \delta_\ast}v+   \sum_{\delta_\ast< \delta: \text{dyadic}} Q^1_\delta  v.
\end{align}
Thus, we may write 
\Be\label{dd}	\langle (m^\kappa_\tau (D) P_\lambda  f) u,v \rangle  = I+ I\!I, \Ee
where
\[ I	=\sum_{\delta_\ast\le \delta_2\le \delta_1} \langle (m^\kappa_\tau (D) P_\lambda  f)  u_{\delta_1},v_{\delta_2}\rangle, 
\quad  I\!I= \sum_{\delta_\ast\le \delta_1< \delta_2} \langle (m^\kappa_\tau (D) P_\lambda  f)  u_{\delta_1},v_{\delta_2} \rangle. \] 

We first consider the case $d\ge 7$. Since $\de_\ast\simeq1/\tau$, it is easy to see that
\Be	\label{u_del}
\|u_\de\|_{L^2} \lesssim \min\{\de, 1/\tau\}^{-\frac12} \|u_\de\|_{X_{\zeta(1),1/\tau}^{1/2}} \lesssim \de^{-\frac12} \|u\|_{X_{\zeta(1),1/\tau}^{1/2}}
\Ee
for any $\de\ge\delta_\ast$. Hence, utilizing \eqref{large} we have 
\begin{equation}  \label{key}
	|\langle (m^\kappa_\tau (D) P_\lambda  f)  u_{\delta_1},  v_{\delta_2}\rangle | 
		\lesssim \dep \lambda^{\kappa-\frac2p} \delta_1^{\frac{d+9-2p }{4p} } \delta_2^{\frac{d+1-2p }{4p} }   \Gamma_{\lambda, \delta_1}^{p, +}(f) 
			\|u \|_{X_{\zeta(1),1/\ta}^{1/2}} \|v \|_{X_{\zeta(1),1/\ta}^{1/2}} 	
\end{equation}
for  $p\ge \frac{d+1}2$ and $\delta_\ast\le\delta_2\le \delta_1$.  Considering the cases $p\ge \frac{d+9}2$ and $p< \frac{d+9}2$ separately and taking summation along the dyadic numbers $\delta_1, \delta_2$, we have that, for $p>\frac{d+1}2$, 
\begin{align*}
| I | \le\!  \sum_{\delta_\ast\le \delta_2\le \delta_1}  | \langle (m^\kappa_\tau (D) P_\lambda  f)  u_{\delta_1},v_{\delta_2}\rangle |  
	\lesssim  \lambda^{\kappa-\frac2p} \tau^{\beta_d(p)+\epsilon}  \sup_{\delta_1}  \Gamma_{\lambda, \delta_1}^{p, +}(f)      \|u \|_{X_{\zeta(1), 1/\tau}^{1/2}} \|v \|_{X_{\zeta(1),1/\tau}^{1/2}}. 
\end{align*}
Symmetrically, interchanging the roles of $\delta_1,$ $\delta_2$ and  repeating the argument for $I\!I$ give  the same bound as for $I$.  Thus we get  \eqref{P-est} for $d\ge 7 $.

We now consider the case $3\le d \le6 $. If   $\delta_\ast\le\delta_2\le \delta_1$,   we use \eqref{large2} instead of \eqref{large} to get   
 \begin{equation}\label{key*+}
	|\langle (m^\kappa_\tau (D) P_\lambda  f)  u_{\delta_1},  v_{\delta_2}\rangle | 
		\lesssim \dep \lambda^{\kappa-\frac2p} \delta_1^{\frac{d+1}{2p}-\frac12 } \delta_2^{\frac{2}{ p}-\frac12 }   \Gamma_{\lambda, \delta_1}^{p, +}(f) 
			\|u \|_{X_{\zeta(1),1/\tau}^{1/2}} \|v \|_{X_{\zeta(1),1/\tau}^{1/2}}. 	
\end{equation}
Similarly we also have the estimate for  $\delta_\ast\le\delta_1< \delta_2$.
Recalling \eqref{dd} and summing  along $\delta_1$, $\delta_2$,
we obtain \eqref{P-est} for $3\le d\le 6$. 
\end{proof}

\subsection{Estimates for the antipodal case} \label{sec_anti} 
We now  consider the case in which 
 $u$, $v$ satisfy \eqref{uv_support} in place of $u_\ell$, $v_{\ell'}$, respectively, while $\mathbb S_\ell$ and $\mathbb S_{\ell'}$ satisfy \eqref{n-case}. 
 In this case  $\mathbb S_\ell$ and $\mathbb S_{\ell'}$ are not close to each other, but so are $-\mathbb S_\ell$ and $\mathbb S_{\ell'}$.  To use the  decomposition \eqref{sdecomp}  we need to modify the definition of  $u_k^j$  as follows: 
\[	\F u_k^j (\xi) = \rchi_{\bs_k^j}\bigg(\frac{ \widetilde{e_2}-\widetilde\xi }{|\widetilde\xi - \widetilde{e_2}|}\bigg) \widehat u(\xi). \]
But we  keep the definition of $v_{k}^j$ the same as in  \eqref{decom_result}. 
As before, by \eqref{sdecomp} we have \eqref{whitney} and \eqref{jjkk}.  
Now,  for $\lambda,\, \delta>0$ and $f\in L^p(\R^d)$, we set 
\[	\Gamma_{\lambda,\delta}^{p, -}(f) = 
		\sup_{j_\circ< j\le j_\ast(\delta)}\,\, \sup_{k\sim k'} \left\{ 2^\frac{2j}{p} \delta^{-\frac 1p} \|P_{k, k'}^{j,-}  P_{\le \delta}^{e_1} P_\lambda f\|_{L^p(\R^d)} \right\}.\]

\begin{lem}\label{bilinear-}
Let  $d\ge3$,  $p\ge \frac{d+1}2$,  $\frac1\tau\le\lambda \lesssim 1$, $0< \delta_2\le \delta_1 \le 2\epsilon_\circ$, and let $f\in L^p(\mathbb R^{d})$.    
Suppose that $\mathbb S_\ell$, $\mathbb S_{\ell'}$ satisfy \eqref{n-case}, and $u,$ $v$ satisfy \eqref{del_sup}.  
Then,  for any $\epsilon>0$,
\begin{align}
\label{large_n}  
	\big|\langle (m^\kappa_\tau (D) P_\lambda  f) u,v\rangle\big| 
		&\lesssim \dep  \delta_1^{\frac{d+9 }{4p} }  \delta_2^{\frac{d+1 }{4p} }   \Gamma_{\lambda, \delta_1}^{p, -}(f) \|u\|_{L^2(\R^d)} \|v\|_{L^2(\R^d)},    \\ 
\label{large2_n}
	\big|\langle (m^\kappa_\tau (D) P_\lambda  f) u,v\rangle\big| 
		&\lesssim \dep   \delta_1^{\frac{d +1 }{2p} } \delta_2^{\frac{2 }{p} }   \Gamma_{\lambda, \delta_1}^{p, -}(f) \|u\|_{L^2(\R^d)} \|v\|_{L^2(\R^d)},
\end{align}
when $\lambda \simeq 1$. If $\lambda \le  1/2$, $\langle (m^\kappa_\tau (D) P_\lambda  f) u,v\rangle=0$. 
\end{lem} 
The proof is similar to that of Lemma \ref{bilinear+} except for different support property of $\phi_{k,k'}^{j,-}$ in the frequency domain. So, we shall be brief.
\begin{proof}
As before, we choose the stopping step $j_\ast=\ja(\delta_1)$ in  \eqref{whitney}.  Noting  $\supp\F u_k^j\subset \big(\{0\}\times (-\mathbb S_k^j)+e_2\big) +O(\delta_1)$ and  $\supp \F v_{k'}^j \subset \big(\{0\}\times \mathbb S_{k'}^j+e_2\big) +O(\delta_2)$, we see from \eqref{thetas_n} that 
\[	\supp \big( \F{u_k^j}(-\cdot) \ast \overline{\F v_{k'}^j} \big)\subset \{ |\xi_1 |  \lesssim \de_1\} \times \mathcal R_{k,k'}^{j,-}, \quad \jc< j\le j_\ast	\]
since $2^{-2\ja}\ge \delta_1$. This is the main difference  from the previous neighboring case. See  \eqref{Rj_n}, \eqref{thetatheta-}, and Figure \ref{fig_conv_anti}.  
Hence, we can insert the harmless projection operator  $P_{k, k'}^{j,-}$ and $ P_{\le \delta_1}^{e_1} $ to  write 
\Be \label{decomp22}
 \langle  ( m^\kappa_\tau (D) P_\lambda  f)  u_k^j, v_{k'}^j \rangle	=  \langle  (P_{k, k'}^{j,-}   P_{\le \delta_1}^{e_1}  m^\kappa_\tau (D) P_\lambda  f)  u_k^j, v_{k'}^j \rangle,  \quad   \jc< j\le j_\ast .
\Ee
From  \eqref{thetatheta-},  it follows  that $\langle  ( m^\kappa_\tau (D) P_\lambda  f)  u_k^j, v_{k'}^j \rangle \neq 0$ only if $\lambda\simeq 1$.  Thus,  for the rest of this proof we may assume that $\lambda\simeq 1$, and, for  $\jc<j \le j_\ast$ and  $k\sim k'$, we have
\[	|\langle  (m^\kappa_\tau (D) P_\lambda   f)  u_k^j, v_{k'}^j \rangle |
	\lesssim  \| P_{k, k'}^{j,-} P_{\le \delta_1}^{e_1}  P_\lambda f\|_{L^p(\R^d)} \, \| u_k^j\, { v_{k'}^j} \|_{L^{p'}(\Rd)}  	\]
by Lemma \ref{Lit_Pal}.  Applying \eqref{j-sep} (with $h_1 = \delta_1$, $h_2 =\delta_2$)  we get \eqref{jjast} for $\jc< j< j_\ast$.
Hence,  as in the proof of Lemma  \ref{bilinear+},   from \eqref{jjkk} and \eqref{decomp22}  we see that  
\begin{equation*}   \begin{aligned}	  
|\mathcal I_j| 
&\lesssim   \dep 2^{\frac j p} \delta_1^{\frac{d+5 }{4p} } \delta_2^{\frac{d+1}{4p} }  \sum_{k\sim k'}  \| P_{k, k'}^{j,-} P_{\le \delta_1}^{e_1}  P_\lambda f \|_{L^p}      \|u_k^j\|_{L^2} \|v_{k'}^j\|_{L^2}  \\  
& \lesssim  \dep 2^{-\frac j p}  \delta_1^\frac{d+9}{4p} \delta_2^\frac{d+1}{4p} \Gamma_{\lambda, \delta_1}^{p, -}(f) \|u\|_{L^2} \|v\|_{L^2},
\end{aligned} \end{equation*}
and, consequently, we get
\Be\label{q_whi_a}	
\sum_{j_\circ < j<j_\ast} |\mathcal I_j | 
	\lesssim \dep \delta_1^\frac{d+9}{4p} \delta_2^\frac{d+1}{4p} \Gamma_{\lambda, \delta_1}^{p, -}(f) \|u\|_{L^2} \|v\|_{L^2}.	
\Ee
The same estimates for $|\mathcal I_{j_\ast}|$ can be obtained exactly in the same way as in the proof of Lemma \ref{bilinear+} since \eqref{final_scale} holds with $j=\ja$. This is easy to show using \eqref{bernstein}. So, we omit the details.

On the other hand, applying \eqref{j-sep2} instead of \eqref{j-sep} and using \eqref{secondcase},  we have \eqref{secondcase0} for $\jc<j\le j_\ast$. Thus, following the same argument as in the above  we obtain 
\[
	|\langle (m^\kappa_\tau (D) P_\lambda  f) u,v\rangle | 
	 \le \sum_{j_\circ < j\le j_\ast} |{\mathcal I_j}| 
	\, \lesssim  \, \dep \delta_1^\frac{d+1 }{2p} \delta_2^\frac{2}{p} \Gamma_{\lambda, \delta_1}^{p, -}(f)  \|u\|_{L^2} \|v\|_{L^2}.   \qedhere
\]
\end{proof}

The following  can be shown in  the same way as in the proof of Proposition \ref{bilinear-P} exploiting Lemma \ref{bilinear-} instead of  Lemma \ref{bilinear+}.  
So, we omit its proof. 

\begin{prop} \label{bilinear-N} 
Let $d\ge3$, $\tau\gtrsim \epsilon_\circ^{-1}$, $\lambda\simeq 1$, $p\ge\frac{d+1}{2}$, and let $f\in L^p(\mathbb R^d)$.  
Suppose that $\mathbb S_{\ell}$ and $\mathbb S_{\ell'}$ satisfy \eqref{n-case}, and that $u$, $v$ satisfy \eqref{uv_support} in place of $u_\ell$, $v_{\ell'}$, respectively.
Then, for  any $\epsilon>0$ and     there is a constant $C=C(\epsilon,p,d)>0$  such that 
\[	| \langle ({m^\kappa_\tau (D) P_\lambda  f}) u,v\rangle | 
		\le   C \tau^{\beta_d(p)+\epsilon} \sup_{\frac1\tau\le\delta\lesssim1}   \Gamma_{\lambda, \delta}^{p, -}(f)    \|u\|_{X_{\zeta(1), 1/\tau}^{1/2}}\|v\|_{X_{\zeta(1),1/\tau}^{1/2}}.	\]
If $0<\lambda\le 1/2$, the left side is zero.
\end{prop} 

\subsection{Estimates for  the transversal case} 
When $\mathbb S_\ell$ and $\mathbb S_{\ell'}$ satisfy \eqref{s-case} and  $u$, $v$ satisfy \eqref{uv_support} in place of $u_\ell$, $v_{\ell'}$, respectively, we can obtain bilinear estimates without invoking the decomposition \eqref{sdecomp}  since the supports of  $\widehat u$, $\widehat v$ are well separated.  Also, note that
\Be\label{supp_annul}
	(-\supp\wh u\pm \supp\wh v\,)\cap B_d(0, 2^{-7}\epsilon_\circ)=\emptyset . 
\Ee
For $\lambda,\, \delta>0$ and $f\in L^p(\R^d)$, we set 
\[\Gamma_{\lambda, \delta}^{p}(f) =  \lambda^\frac1p \delta^{-\frac 1p} \| P_{\le \delta}^{e_1} P_\lambda f\|_{L^p(\R^d)}.\]

\begin{prop}\label{bilinear-S}
Let $d\ge3$, $p\ge\frac{d+1}{2}$, $\tau\gtrsim \epsilon_\circ^{-1}$, $\frac1\tau\le\lambda\lesssim 1$,  and let $f\in L^p(\mathbb R^d)$.  Suppose that $\mathbb S_{\ell}$, $\mathbb S_{\ell'}$  satisfy \eqref{s-case} and $u$, $v$ satisfy \eqref{uv_support} in place of $u_\ell$, $v_{\ell'}$, respectively. Then, for any $\epsilon>0$,
\Be \label{est-S} 
	|\langle (m^\kappa_\tau (D) P_\lambda  f)  u,v\rangle | 
	\lesssim \,  \la^{\kappa-\frac1p} \tau^{\beta_d(p)+\epsilon} \sup_{\frac1\tau\le\delta\lesssim1}\Gamma_{\lambda, \delta}^{p}(f) \,\,\, \|u\|_{X_{\zeta(1), 1/\tau}^{1/2}}\|v\|_{X_{\zeta(1),1/\tau}^{1/2}}.
\Ee
\end{prop} 

It should be also noted that $\langle (m^\kappa_\tau (D) P_\lambda  f)  u,v\rangle\neq 0$ only if $ \la\gtrsim \epsilon_\circ$ because of \eqref{supp_annul}. 

\begin{proof}
We follow the same way as in  the proof of Proposition \ref{bilinear-P}. Using  \eqref{interpolation} in  Lemma \ref{1separation},  the dyadic decomposition \eqref{dyadic1}, \eqref{dyadic2}, and \eqref{dd}, we see that 
\begin{align*}	
|\langle (m^\kappa_\tau   (D) P_\lambda  f)  u_{\delta_1},  v_{\delta_2}\rangle | 
	& = |\langle ( m^\kappa_\tau (D) P_\lambda  P_{\le \delta_1}^{e_1} f)   u_{\delta_1},  v_{\delta_2}\rangle |   \lesssim  \lambda^\kappa \| P_\lambda P_{\le \delta_1}^{e_1} f\|_{L^p} \| u_{\delta_1} v_{\delta_2}\|_{L^{p'}} \\
	&\lesssim \dep \lambda^\kappa  \delta_1^{\frac{d+5}{4p}} \delta_2^{\frac{d+1}{4p}} \| P_\lambda P_{\le \delta_1}^{e_1} f\|_{L^p}   \| u_{\delta_1}\|_{L^2} \|v_{\delta_2}\|_{L^2}    \\
	&\lesssim \dep \la^{\kappa-\frac1p} \delta_1^{\frac{d+9-2p }{4p} } \delta_2^{\frac{d+1-2p }{4p}}  \sup_{\delta_1 \ge \frac1\tau} \big\{  \lambda^{\frac1p}\delta_1^{-\frac1p} \|  P_{\le \delta_1}^{e_1}   P_\la f\|_{L^p} \big\} \|u \|_{X_{\zeta(1),1/\tau}^{1/2}} \|v \|_{X_{\zeta(1),1/\tau}^{1/2}} 	
\end{align*}
whenever $\delta_\ast\le \delta_2\le \delta_1$.  This plays the role of \eqref{key} in the proof of Proposition \ref{bilinear-P}.  As before summation over $\delta_\ast \le \delta_2\le  \delta_1$ gives the desired bound on $I$. Interchanging the roles of $\delta_1$ and $\delta_2$  yields  the estimate for $I\!I$.  Thus,   we get \eqref{est-S} for $d\ge 7$. Applying \eqref{interpolation2} instead of \eqref{interpolation}  gives  the estimate \eqref{est-S} for $3\le d\le 6$. 
\end{proof}

Combining  the three (neighboring, antipodal, and transversal) cases  and using  Proposition \ref{bilinear-P}, Proposition \ref{bilinear-N}, and Proposition \ref{bilinear-S},  we obtain the following. 
 
\begin{prop} \label{bilinear-sum} Let $d\ge 3$, $p\ge \frac{d+1}2$,  $\tau\gtrsim \epsilon_\circ^{-1}$, $\frac1\tau\le \lambda \lesssim 1$, and  let  $f\in L^p (\R^d)$. 
 Suppose $u$ and $v$ satisfy \eqref{near}. 
Then
\begin{equation} \label{Large}  
|\langle (m^\kappa_\tau (D) P_\lambda  f) u,v\rangle | 
	\lesssim  \lambda^\nu \tau^{\mu}  \mathcal A_{p}(f, \lambda,  e_1) \|u\|_{X_{\zeta(1),  1/\tau}^{1/2}}\|v\|_{X_{\zeta(1), 1/\tau}^{1/2}},
\end{equation}
where $\nu={\kappa-\frac 2 p}$, $\mu>\beta_d(p)$, and
\[	 \mathcal A_{p}(f, \lambda, e_1)
	= \max_{\substack{(\mathbb S_\ell, \mathbb S_{\ell'}):\\ \text{neighboring}} } \sup_{0<\delta\lesssim 1} \Gamma_{\lambda, \delta}^{p, +}(f) 
		+ \max_{\substack{(\mathbb S_\ell, \mathbb S_{\ell'}):\\ \text{antipodal}} } \sup_{0<\delta\lesssim 1}  \Gamma_{\lambda, \delta}^{p, -}(f) 
		+ \sup_{0<\delta\lesssim 1}  \Gamma_{\lambda, \delta}^{p}(f) . 	\] 
\end{prop}

\begin{proof}
In the primary decomposition \eqref{primary}, the number of pairs $( \mathbb S_\ell, \mathbb S_{\ell'})$ is finite. 
Thus, from the estimates in  Proposition \ref{bilinear-P}, Proposition \ref{bilinear-N}, and  Proposition \ref{bilinear-S}  it follows that
\begin{multline*}
 |\langle (m^\kappa_\tau (D) P_\lambda  f) u,v\rangle |   
\lesssim  
	  \tau^{\beta_d(p)+\epsilon}  \|u\|_{X_{\zeta(1),  1/\tau}^{1/2}}\|v\|_{X_{\zeta(1), 1/\tau}^{1/2}} \times \\
\Big(  \lambda^{\kappa-\frac 2p} \max_{\substack{(\mathbb S_\ell, \mathbb S_{\ell'}):\\ \text{neighboring}} } \sup_{\frac1\tau\le\delta\lesssim 1} \Gamma_{\lambda, \delta}^{p, +}(f)     
		+ \max_{\substack{(\mathbb S_\ell, \mathbb S_{\ell'}):\\ \text{transversal}} } \sup_{\frac1\tau\le\delta\lesssim 1}   \Gamma_{\lambda, \delta}^{p, -}(f)   
		+ \lambda^{\kappa-\frac 1 p}   \sup_{\frac1\tau\le\delta\lesssim 1} \Gamma_{\lambda, \delta}^{p}(f) \Big). 
\end{multline*}
This gives \eqref{Large}  since $\lambda \lesssim 1$. 
\end{proof}

\begin{rmk} When $d=3$, it is possible to remove $\tau^\epsilon$ by replacing $\lambda^{\kappa -\frac2p}$ with $\log(1/\lambda)\lambda^{\kappa-\frac2p}$ since the number of \emph{nonzero} terms in the summation \eqref{sumjstar} is $\simeq \log(1/\lambda)$.
\end{rmk} 

\subsection{Strengthening  the estimates \eqref{Large} when $3\le d\le 8$ and $p\ge d$}\label{sec_str}
The estimates in Proposition \ref{bilinear-P}, Proposition \ref{bilinear-N}, and Proposition \ref{bilinear-S}  can be improved if we restrict the range of $p$ to the interval $[d,\infty]$, and combine them with the following which is a consequence of the linear estimate \eqref{LIN}.  
 
\begin{lem} 
Let $p \ge d\ge 3$,  $\frac1\tau\le \delta_2\le \delta_1\lesssim 1$, and $\frac1\tau\le \lambda \lesssim 1$.  Suppose that $f \in L^p(\mathbb R^{d})$  and   $u$, $v$ satisfy \eqref{del_sup}. Then, 
\begin{equation} \label{linear-key}
	|\langle (m^\kappa_\tau (D) P_\lambda  f)  u,  v \rangle | 
	\lesssim \lambda^{\kappa-\frac 1p} \, \delta_1^{\frac 1 p -\frac 12} \delta_2^{\frac{d+2}{2p} - \frac 12}
	\Gamma_{\lambda, \delta_1}^{p}(f) 
	\|u \|_{X_{\zeta(1),1/\tau}^{1/2}} \|v \|_{X_{\zeta(1),1/\tau}^{1/2}}.
\end{equation} 
\end{lem}
\begin{proof} 
Interpolation between \eqref{LIN} (with $h_2=\delta_2$) and the trivial estimate $\|u v\|_{L^1} \le  \| u\|_{L^2}  \| v\|_{L^2}$ gives
\[ \| u v \|_{L^{p'}(\R^d)} \lesssim \delta_2^{\frac{d+2}{2p}} \| u \|_{L^2 (\R^d)}  \| v\|_{L^2(\R^d)}	\]
for $p \ge d$.
Using this estimate and Lemma \ref{Lit_Pal},  we see that
\begin{equation*} 
\begin{aligned}
	|\langle (m^\kappa_\tau (D) P_\lambda  f)  u,  v \rangle | 
	&= |\langle ( m^\kappa_\tau (D) P_{\le \delta_1}^{e_1} P_\lambda f)  u,  v\rangle | 
		 \lesssim \lambda^\kappa \|P_{\le \delta_1}^{e_1} P_\lambda f\|_{L^p} \| u v \|_{L^{p'}} \\
	&\lesssim  \lambda^{\kappa-\frac1p} \delta_1^\frac1p \delta_2^{\frac{d+2}{2p} }  \Gamma_{\lambda, \delta_1}^p(f)  \| u \|_{L^2} \|v \|_{L^2}.
\end{aligned}
\end{equation*} 
Combining this with  \eqref{u_del} we get \eqref{linear-key}.
\end{proof}

For $3\le d\le 8$, we define $\gamma_d: [d,\infty]\to [0,1]$.  For $3\le d\le 6$,  we set 
\begin{align} \label{gamma3-7}
\gamma_d(p) &= 
\begin{cases} 
\,	 1- \frac{d+5}{2p}   &\text{if }\ \   p \ge d+1,  \\[5pt]
\,	\frac {(2d-3)p- ( d^2+3d -6) }{2 (d-1)  p}   &\text{if }\ \   d+1  > p \ge \max\{ d,  \frac{d^2+3d-6}{2d-3} \} , \\[5pt] 
\,	0   &\text{if }\ \  \max\{ d,  \frac{d^2+3d-6}{2d-3}  \} > p \ge d,
\end{cases}
\intertext{and, for $d=7, 8$, we set}
\notag \gamma_d( p)  & =
\begin{cases} 
\,	 1- \frac{d+5}{2p}  &\text{if }\ \  p \ge \frac{d+9}{2}, \\[5pt]
\,	\frac{ 2(d+4)p- (d^2+9d +16) }{2 (d+5)  p}    &\text{if }\ \ \frac{d+9}{2} > p \ge d.
\end{cases}
\end{align}

Note that $\frac{d^2+3d-6}{2d-3} \le d$ if and only if $d \ge 5$, so there  is no $p$  that  belongs to the range of  third line in \eqref{gamma3-7} when $d=5,6$.  In higher dimensions $d\ge 9$ the bounds in Proposition \ref{bilinear-P}, Proposition \ref{bilinear-N},  Proposition \ref{bilinear-S} are already better than the estimates which we can deduce by combining the linear and bilinear estimates.   Improved bounds are possible  for all three cases \eqref{s-case}, \eqref{p-case}, and \eqref{n-case} by the similar argument. So we provide the details only for the  case \eqref{p-case} and  state the estimates for the other cases without providing the proof. 

\begin{prop}\label{bilinear-P2} 
Let $3\le d\le 8$, $\tau \gtrsim \epsilon_\circ^{-1}$, $\frac1\tau \le \lambda \lesssim 1$, $p\ge d$, and let $f\in L^p(\mathbb R^{d})$. Suppose that $\mathbb S_{\ell}$ and $\mathbb S_{\ell'}$ satisfy \eqref{p-case},  and   $u$, $v$ satisfy \eqref{uv_support} in place of $u_\ell$, $v_{\ell'}$, respectively. For $\nu={\kappa-\frac 2p}$ and $\mu>\gamma_d(p)$,  there is a constant $C>0$,  independent of $\tau$ and $\lambda$, such that 
\begin{equation} \label{bi-P2}
|\langle (m^\kappa_\tau (D) P_\lambda  f)  u ,  v \rangle | 
	\le C \lambda^\nu  \tau^{ \mu  }     
		\sup_{ \frac1\tau \le \delta\lesssim 1} \big( \Gamma_{\lambda, \delta }^{p, +}(f) + \Gamma_{\lambda, \delta }^{p }(f)\big) 
			\|u \|_{X_{\zeta(1),1/\tau}^{1/2}} \|v \|_{X_{\zeta(1),1/\tau}^{1/2}}.
\end{equation}
\end{prop}

\begin{proof}   
We first consider the case $3\le d\le 6$ under the assumption that  $d\le p < d+1$.  Recalling \eqref{dd}, it is sufficient to handle $I$ because  $I\!I$ can be handled symmetrically. So, we  assume  $\delta_2 \le \delta_1$.  From \eqref{key*+} and \eqref{linear-key}  we have, for $0\le \theta\le 1$,  
\Be\label{means}
|\langle (m^\kappa_\tau (D) P_\lambda  f)  u_{\delta_1},  v_{\delta_2}\rangle |  \\
		\lesssim \delta_2^{-\epsilon} \lambda^{\kappa-\frac2p} \delta_1^{b_1(\theta) } \delta_2^{b_2(\theta) }  
			\Gamma_\theta\|u \|_{X_{\zeta(1),1/\tau}^{1/2}} \|v \|_{X_{\zeta(1),1/\tau}^{1/2}} 
\Ee
where $\Gamma_\theta =  \big( \Gamma_{\lambda, \delta_1}^{p, +}(f))^{\theta}(\Gamma_{\lambda,\delta_1}^{p}(f) )^{1-\theta}$ and 
\[	b_1(\theta)   =  \frac{( d-1) \theta +2}{2p} -\frac 12, \quad
	b_2(\theta)   = \frac{ (d+2) - (d-2)\theta}{2p} - \frac 12.	\]
  
Note that $\Gamma_\theta\le \theta \Gamma_{\lambda, \delta_1}^{p, +}(f) +  (1-\theta) \Gamma_{\lambda,\delta_1}^{p}(f) $. Hence, in order to show \eqref{bi-P2} it is sufficient to find $\theta\in [0,1]$ such that  
\Be \label{target} 
	\Sigma(\theta):=\sum_{\frac1\tau \lesssim  \delta_2 \le \delta_1\lesssim 1}  \delta_1^{b_1(\theta) } \delta_2^{b_2(\theta) }\le C( \log \tau ) \tau^{\gamma_d(p)}.
\Ee
For the purpose  let $\theta_1$ and $\theta_2$ be such that $b_1(\theta_1) = 0$ and $b_2(\theta_2) =0$. That is to say,  $\theta_1 = \frac{p-2}{d-1}$ and $\theta_2 = \frac{d+2-p}{d-2}$.  Since $ d\le p<d+1$ we see that $0<\theta_1<1$ and $\theta_2>0$.  We consider the two cases
\[	(i)\colon  \theta_2 \le \theta_1<1,  \quad 
	(ii)\colon   \theta_1 < \theta_2,	\]
which are equivalent to
\[	(i)\colon  (d^2+3d-6)/(2d-3) \le p < d+1,  \quad  
	(ii)\colon  d\le p < (d^2+3d-6)/(2d-3),	\]
respectively. If $d=3$, the case $(i)$ is void. If $d\ge 5$, since $d\ge (d^2+3d-6)/(2d-3)$, the case $(ii)$ is void and the other $(i)$ is equivalent to $d\le p <d+1$.  For $d=3,4$, the case $(ii)$ is easy to handle. Indeed, if we choose any $\theta\in (\theta_1,\theta_2)$, \eqref{target} holds with $\gamma_d(p)=0$ since  $b_1(\theta), b_2(\theta)>0$. 
 
Now it remains to consider the case $(i)$ when $d\ge 4$.  We separately consider the following three cases: 
\[	(\rm A)\colon 0\le \theta \le \theta_2, \quad (\rm B)\colon \theta_2<  \theta \le \theta_1,\quad (\rm C)\colon \theta_1< \theta \le 1.	\]
We note that   $ b_1(\theta) + b_2(\theta) = \frac{\theta + d+4}{2p} -1$ is increasing in $\theta$. Hence for $0\le \theta\le\theta_1$, $b_1(\theta) + b_2(\theta)\le b_2(\theta_1)  \le 0 $. In the case $(\rm A)$,  $b_2(\theta)\ge 0$. Thus, 
\[	\Sigma(\theta)
 	=\sum_{\frac1\tau\lesssim \delta_2 \le \delta_1\lesssim 1} \delta_1^{ b_1(\theta) + b_2(\theta)} \Big(\frac{\delta_2}{\delta_1} \Big)^{b_2(\theta) }
	\lesssim \sum_{\frac1\tau\lesssim \delta_1\lesssim 1}  \delta_1^{ b_1(\theta) + b_2(\theta)}
	\lesssim  ( \log \tau ) \tau^{-( b_1(\theta)+ b_2(\theta) )}.
\]
In the case $(\rm B)$, $b_1(\theta)\le 0$ and $b_2(\theta) <0$. Thus, $\Sigma(\theta) \lesssim ( \log \tau )  \tau^{-( b_1(\theta)+ b_2(\theta) )}.$  In the case $(\rm C)$, $b_1(\theta) >0 $ and $b_2(\theta) <0$, so summation along $\delta_1$ is finite and we get  $\Sigma(\theta)  \lesssim \tau^{-b_2(\theta)}$.  Therefore, recalling $b_1(\theta)+b_2(\theta)$ is increasing and non-positive when $0\le \theta\le \theta_1$  and $b_2(\theta)$ is decreasing and non-positive when $\theta\ge\theta_1$,  we choose $\theta=\theta_1$ that makes  \eqref{target} true with  the smallest exponent of $\tau$, and we  have \eqref{target} with 
\begin{gather*}
\gamma_d(p) = - b_2(\theta_1)   = \frac{(2d-3)p - (d^2+3d -6)}{2p(d-1)} 
\end{gather*}
provided that  $\max\{d, \frac{d^2+3d-6}{2d-3} \} \le p < d+1$. 

If $p\ge d+1$, since $\theta_2\le1 \le \theta_1$, the case $(ii)$ is void and we need only to consider the cases $(\rm A)$ and $(\rm B)$. From the above computation we take $\gamma_d(p)=-(b_1(1)+b_2(1))=1-\frac{d+5}{2p}$, which corresponds to \eqref{means} with $\theta=1$. 

We now turn to the case $d=7, 8$. Combining \eqref{key} and \eqref{linear-key}, we  have
\[	|\langle (m^\kappa_\tau (D) P_\lambda  f)  u_{\delta_1},  v_{\delta_2}\rangle |  
		\lesssim \tau^\epsilon \lambda^{\kappa-\frac2p} \delta_1^{b_1 (\theta) } \delta_2^{b_2 (\theta) }   \, \Gamma_\theta
			\|u \|_{X_{\zeta(1),1/\tau}^{1/2}} \|v \|_{X_{\zeta(1),1/\tau}^{1/2}},	\]
where  $\theta \in [0,1]$ and $b_1 (\theta)  =  \frac{( d+5) \theta +4}{4p} -\frac 12,$ $b_2  (\theta)= \frac{ 2(d+2) - (d+3)\theta}{4p} - \frac 12.$  Once we have this estimate we can repeat the same argument to get the desired bound \eqref{bi-P2}. So, we omit the details. 
\end{proof}

For the other cases \eqref{n-case} and \eqref{s-case},  we apply the same argument to get improved estimates. In fact, for the antipodal case \eqref{n-case}, we use  \eqref{large_n}, \eqref{large2_n},  and \eqref{linear-key}. Thus, we  get
\begin{equation}\label{bi-N}
|\langle (m^\kappa_\tau (D) P_\lambda  f)  u ,  v \rangle | 
	\lesssim   \tau^{ \gamma_d(p)+ \epsilon} \sup_{ \frac1\tau \le \delta\lesssim 1} \big( \Gamma_{\lambda, \delta }^{p, -}(f) + \Gamma_{\lambda, \delta }^{p }(f)\big) 
		\|u \|_{X_{\zeta(1),1/\tau}^{1/2}} \|v \|_{X_{\zeta(1),1/\tau}^{1/2}}
\end{equation}
for any $\epsilon>0$. For the transversal case \eqref{s-case},  we have, for $\epsilon>0$, 
\begin{equation}\label{bi-S}
|\langle (m^\kappa_\tau (D) P_\lambda  f)  u ,  v \rangle | 
	\lesssim \lambda^{\kappa-\frac 1 p} \tau^{ \gamma_d(p)+ \epsilon}   \sup_{\frac1\tau \le \delta\lesssim 1}   \Gamma_{\lambda, \delta }^{p }(f) 
		\|u \|_{X_{\zeta(1),1/\tau}^{1/2}} \|v \|_{X_{\zeta(1),1/\tau}^{1/2}}.
\end{equation} 
As in Proposition \ref{bilinear-sum}, combining the estimates \eqref{bi-P2}, \eqref{bi-N}, and \eqref{bi-S} of the three cases \eqref{p-case},  \eqref{n-case}, and \eqref{s-case}, we obtain the following. 
 
\begin{prop} \label{bilinear-sum-2}  
Let $3\le d\le 8$, $\tau \gtrsim \epsilon_\circ^{-1}$, $\frac1\tau \le \lambda \lesssim 1$,  $p\ge d$, and let $f\in L^p(\mathbb R^{d})$.  Suppose  $u$ and $v$ satisfy \eqref{near}. Then \eqref{Large} holds for $\nu={\kappa-\frac 2 p}$ and $\mu>\gamma_d(p)$. 
\end{prop}

\section{Average over rotation and dilation} \label{sec_aver}
In  this section, we consider the average of $\mathcal A_{p}(f,\tau\la,\tau^{-1}Ue)$ over $U\in O_d$ and $\tau\in [1,2]$.  The projection operators engaged in the definition of  $\mathcal A_{p}(f,\tau\la,\tau^{-1}Ue)$ break the Fourier support of $f$ into small pieces.  Average over $U\in O_d$ and $\tau\in [1,2]$ makes it possible to exploit  such smallness of Fourier supports.  This gives considerably better bounds which are not viable when one attempts to control  $\mathcal A_{p}(f,\la,e)$ for a fixed $e$ with $\|P_\lambda f\|_{L^p}$.  

For an invertible $d\times d$ matrix $U$, let us define the projection operator   $(P_{k, k'}^{j, \pm })^U$ by 
\Be \label{projkku}
	\F\big((P_{k, k'}^{j, \pm})^U g \big)(\xi) = \phi_{k,k'}^{j,\pm}(\widetilde{U^t\xi})\, \widehat g(\xi),
\Ee 
where $U^t$ is the transpose of  $U$.  Let $dm$ be the normalized Haar measure on $O_d$.  Then we have, for any $\theta\in \mathbb S^{d-1}$ and $f\in L^1(\mathbb S^{d-1})$,
\Be\label{haar_msr}
	\int_{O_d}f(U\theta) dm(U) =  c_d \int_{\mathbb S^{d-1}} f(\omega)d\sigma(\omega)
\Ee
for some dimensional constant $c_d$. Let $\mathcal P_\lambda$ denote the operator given by $\mathcal F(\mathcal P_\lambda f)= \widetilde \beta(|\cdot|/\lambda) \widehat f$ with $\widetilde \beta\in C_c^\infty ((2^{-2}, 2^3))$ such that $\widetilde \beta=1$ on $[2^{-1}, 2^2]$. The following lemma can be obtained  in the same manner as in the proof of \cite[Lemma 5.1]{Haberman15}.  
 
\begin{lem}\label{s-average}   Let $\delta, \lambda>0$.    If $f\in L^p(\mathbb R^{d})$, $p\in[2,\infty)$, then  
\begin{equation} \label{s-avr}
	\int_1^2  \int_{O_d}     \|  P_{\le  \delta  }^{\tau^{-1} U e_1}  P_{\tau\lambda} f \|_{L^p(\R^d)}^p    d m(U) d\tau \lesssim(\delta/\lambda) \|\mathcal P_{\lambda} f\|_{L^p(\R^d)}^p, 
\end{equation} 
where the implicit constant is independent of $\delta, \lambda$.
\end{lem}

The following lemma is a consequence of Lemma \ref{s-average} and properties of the projection operator $P_{k,k'}^{j,+}$. Recall the definition of $\jc$ and $\ja=\ja(\delta)$ from Section \ref{dylocal} and Section \ref{sec_neigh}.

\begin{lem} \label{p-average}  
Let $0<\delta, \lambda\lesssim 1$ and $f\in L^p(\mathbb R^{d})$, $p \in [ 2,\infty)$. For  $(\mathbb S_\ell, \mathbb S_{\ell'})$ satisfying \eqref{p-case} and  $\jc<j<\ja=\ja(\delta)$, we have 
\begin{equation}\label{p-avr}
	\int_1^2  \int_{O_d} \sup_{k\sim k'} \|(P_{k,k'}^{j,+})^{\tau^{-1}U} P_{\le \delta }^{\tau^{-1}U e_1}  P_{\tau\lambda} f \|_{L^p(\R^d)}^p d m(U) d\tau  
	\lesssim 2^{j} \delta  \|\mathcal P_{\lambda} f\|_{L^p(\R^d)}^p,
\end{equation}
and 
\begin{equation} \label{p-avr2}
\int_1^2  \int_{O_d} \sup_{k\sim k'} \|(P_{k,k'}^{j_\ast,+})^{\tau^{-1}U} P_{\le \delta }^{\tau^{-1}U e_1}  P_{\tau\lambda} f \|_{L^p(\R^d)}^p d m(U) d\tau  
	\lesssim   ( \delta/\lambda)  \|\mathcal P_{\lambda} f\|_{L^p(\R^d)}^p.
\end{equation}
Here the implicit constants are independent of $\delta$, $\la$, and $j$. 
\end{lem}

\begin{proof}
For any $p$, the norms $\|(P_{k,k'}^{j,+})^{\tau^{-1 } U}\|_{L^p\to L^p}$ are bounded uniformly with respect to $\jc<j\le\ja$,\, $k\sim k',\, \tau\in [1,2],$ and $U\in O_d$. Hence 
\Be \label{easyeasy}
\sup_{k\sim k'} \|(P_{k,k'}^{j,+})^{\tau^{-1}U} P_{\le \delta }^{\tau^{-1}U e_1}  P_{\tau\lambda} f \|_{L^p(\R^d)}^p 
	\lesssim \| P_{\le \delta }^{\tau^{-1}U e_1}  P_{\tau \lambda}  f \|_{L^p(\R^d)}^p 
\Ee
holds uniformly for all $j$, $\tau$, and $U$.  When $\jc<j<\ja$ we note from  \eqref{thetatheta} that  the support of the multiplier of $ (P_{k,k'}^{j,+})^{\tau^{-1 } U} $  is contained  in the annulus $  \{ \xi : |\widetilde{U^t \xi}| \simeq 2^{-j}\}$ since $\tau\in [1,2]$. Hence, we may assume $\lambda \simeq 2^{-j}$ as seen in the proof of Lemma \ref{bilinear+}, and \eqref{p-avr} follows from \eqref{easyeasy} and  \eqref{s-avr}.  Also, \eqref{p-avr2} follows similarly by using \eqref{thetatheta1} instead of \eqref{thetatheta}.
\end{proof}

In Lemma \ref{p-average}, the average in $\tau$ does not have any significant role. However, in what follows, the average in dilation yields additional improvement.  To show this  we exploit the  support properties of the multiplier of $P_{k,k'}^{j,-}$.

\begin{lem} \label{n-average}  
Let $0<\delta,  \lambda\lesssim  1$ and $f\in L^p(\mathbb R^{d})$ for $p \in [ 2,\infty)$. For  $(\mathbb S_\ell, \mathbb S_{\ell'})$ satisfying \eqref{n-case} and  $\jc< j\le  \ja=\ja(\delta) $, we have 
\begin{equation}\label{n-avr}
	\int_1^2  \int_{O_d} \sup_{k\sim k'} \| (P_{k,k'}^{j,-})^{\tau^{-1}U} P_{\le \delta }^{\tau^{-1}U e_1}  P_{\tau\lambda} f \|_{L^p(\R^d)}^p d m(U) d\tau 
	\lesssim 2^{-2j} \delta  \|\mathcal P_\lambda f\|_{L^p(\R^d)}^p,
\end{equation}
where the implicit constant is independent of $\delta$, $\la$, and $j$. 
\end{lem}

\begin{proof} 
Note that the multiplier of the operator $(P_{k,k'}^{j,-})^{\tau^{-1}U} P_{\le \delta }^{\tau^{-1}U e_1} $ is supported in a dilation (from its center) of the rectangle $\tau U\big([-\delta, \delta]\times  \mathcal R_{k,k'}^{j,-}\big)$ of dimensions about  $\delta\times  2^{-2j} \times \overbrace{2^{-j} \times \dots \times 2^{-j}}^{ d-2 \text{ times}}$. From \eqref{thetatheta-} we note that   $\{ \tau \mathcal R_{k,k'}^{j,-}\}_{k\sim k'}$ are boundedly overlapping and contained in $ 2\tau\mathbb S^{d-2}+O(2^{-2j})$.  Thus we may assume  $\lambda \simeq 1$; otherwise the left side of \eqref{n-avr} vanishes, so the estimate is trivial. For the proof of \eqref{n-avr} it suffices to prove  that
\Be\label{easyl2}
\int_1^2  \int_{O_d} \sum_{k\sim k'} \| (P_{k,k'}^{j,-})^{\tau^{-1}U} P_{\le \delta }^{\tau^{-1}U e_1}  P_{\tau\lambda} f \|_{L^p(\R^d)}^p d m(U) d\tau 
	\lesssim 2^{-2j} \delta  \|\mathcal P_\lambda f\|_{L^p}^p.
\Ee

It is easy to see that 
\[
 \| (P_{k,k'}^{j,-})^{\tau^{-1}U} P_{\le \delta }^{\tau^{-1}U e_1}  P_{\tau\lambda}  f \|_{L^\infty(\R^d)} \le C \|\mathcal P_{\lambda} f \|_{L^\infty(\R^d)} 
\]
with $C$ independent of  $j$, $k$, $k'$, $\tau$, and $U$.  Hence, 
by interpolation  between $\ell_k^\infty L^\infty$  and $\ell_k^2 L^2$,  and  by Plancherel's theorem, to get \eqref{easyl2} for $2\le p<\infty$, it is enough to show 
\begin{align*}
\int_1^2 \int_{O_d} \int  \sum_{k\sim k'} \Big| \beta_0 \Big(\frac{U e_1\cdot \xi}{2 \delta}\Big)  \phi_{k,k'}^{j,-} 
 \Big(\frac {\widetilde{U^t \xi}}\tau \Big)  \beta\Big(\frac{|\xi|}{\lambda\tau} \Big) \widehat f(\xi) \Big|^2  d\xi d m (U)  d\tau 
\lesssim
 2^{-2j}  \delta   \|\mathcal P_\lambda f\|_{L^2}^2.
\end{align*}  
Again, by interpolation with the trivial $\ell_k^\infty L^\infty_{\xi, U, \tau}$ estimate,  
it is enough to show that for any $j$,
\[	\int_1^2 \int_{O_d} \int  \beta_0\Big(\frac{U e_1\cdot \xi}{2 \delta}\Big) \sum_{k\sim k'}  \phi_{k,k'}^{j,-} \Big(\frac {\widetilde{U^t \xi}}\tau \Big)    \Big|  \beta \Big( \frac{|\xi|}{\lambda\tau} \Big)  g(\xi) \Big|  d\xi dm(U)  d\tau \lesssim 2^{-2j} \delta    \|g\|_{L^1}.	\]
This follows if we show that, for $2^{-1}\lambda\le |\xi| \le 2^2 \lambda$, 
\begin{equation}
\label{11goal}
\int_1^2 \int_{O_d} 
                  \beta_0\Big(\frac{U e_1\cdot \xi}{2 \delta}\Big)  
                         \sum_{k\sim k'} \phi_{k,k'}^{j,-} 
                            \Big(\frac {\widetilde{U^t \xi}}\tau \Big)   dm(U)  d\tau 
\lesssim
2^{-2j} \delta .
\end{equation}

By \eqref{haar_msr} and Fubini's theorem, we have
\[	\iint \sum_{k\sim k'} \phi_{k,k'}^{j,-} \Big(\frac {\widetilde{U^t \xi}}\tau \Big) \beta_0 \Big(\frac{U e_1\cdot \xi}{2 \delta}\Big) dm(U)  d\tau \simeq
	\int_{\mathbb S^{d-1}}  \Big( \int_1^2 \sum_{k\sim k'} \phi_{k,k'}^{j,-} \B( \frac{|\xi|\widetilde \omega}{\tau}\B) d\tau\Big)  \beta_0 \Big(\frac{ e_1 \cdot |\xi| \omega }{2 \delta}\Big) d\sigma(\omega)  .
\]
Since  $|\xi| \simeq \lambda \simeq 1$ and $\bigcup_{k\sim k'} \mathcal R_{k,k'}^{j,-}$ is contained in $2\mathbb S^{d-2}+O(C2^{-2j}) $, we see that
\begin{equation} \label{tau-d}
\int_1^2  \sum_{k\sim k'}   \phi_{k,k'}^{j,-}  \Big(\frac {|\xi| \widetilde \omega} \tau \Big) d\tau
	 \lesssim  \int_1^2     \chi_{  \bigcup_{k\sim k'} \mathcal R_{k,k'}^{j,-} }\Big(\frac{|\xi| \widetilde{\omega}} \tau \Big) d\tau 
	 \simeq \int_{|\xi|/2}^{|\xi|}  \chi_{  \bigcup_{k\sim k'} \mathcal R_{k,k'}^{j,-} } ( t \widetilde{\omega} ) \frac{dt}{t^2} \lesssim  2^{-2j} 
\end{equation}
for $\omega \in \mathbb S^{d-1}$.  For $|\xi| \ge 2^{-1}\lambda$,   it is easy to see that
\begin{align}\label{hohoha}
\int_{\mathbb S^{d-1}} \beta_0\Big(\frac{ e_1\cdot |\xi| \omega }{2 \delta}\Big)  d \sigma(\omega) 
	\lesssim \min\{\delta,1\}.
\end{align}
Combining \eqref{tau-d} and \eqref{hohoha}, we get \eqref{11goal}. 
\end{proof}

Combining Lemma \ref{s-average}, Lemma \ref{p-average}, and  Lemma \ref{n-average},  we obtain the following.
\begin{prop}
Let  $0< \lambda \lesssim 1$ and let $f \in L^p(\mathbb R^{d})$ with $p \in [2,\infty)$.  
Then,
\begin{equation}\label{key-avr}
\int_1^2 \int_{O_d} \big[ \mathcal A_p (f_{\tau U}, \lambda, e_1 ) \big]^p dm(U) d\tau \le C  \| \mathcal P_\lambda f\|_{L^p(\R^d) }^p,
\end{equation}
where $f_{\tau U}(x)=\tau^{-d} f(\tau^{-1}Ux)$, and $C$ is independent of $\lambda$ and $e_1\in\mathbb S^{d-1}$.
\end{prop}
\begin{proof}
Since $\widehat{f_{\tau U}}(\xi) = \widehat f(\tau U\xi)$,  by changing variables $\xi \to \tau^{-1}U^t \xi$ in the frequency side,  we have
\begin{align*} 
	\| P_{k,k'}^{j,\pm} P_{\le \delta}^{e_1} P_\lambda f_{\tau U}  \|_{L^p}& = \tau^{-d + \frac d p} \| \big(P_{k,k'}^{j,\pm}\big)^{\tau^{-1  }U  } P_{\le \delta}^{\tau^{-1} U e_1} P_{\tau\lambda}  f  \|_{L^p}  , \\
	\| P_{\le \delta}^{e_1} P_\lambda f_{\tau U} \|_{L^p} &= \tau^{-d + \frac d p} \|   P_{\le \delta}^{\tau^{-1} U e_1} P_{\tau\lambda}  f \|_{L^p}.
\end{align*}
Since $\tau\simeq 1$, from  the definition  of $\Gamma_{\lambda, \delta}^{p, +}(f)$ and $ \Gamma_{\lambda, \delta}^{p, -}(f) $,  it follows that
\begin{equation}\label{A_delta}
\begin{aligned}
	\mathcal A_p (f_{\tau U},\lambda, e_1)  
	&\lesssim    
		\max_{\substack{(\mathbb S_\ell, \mathbb S_{\ell'}): \\ \text{neighboring}}} 
		\sup_{0<\delta\lesssim 1}   \sup_{\jc<j\le j_\ast} \sup_{k\sim k'} 
			\B\{(\lambda/\delta)^\frac1p \| \big(P_{k,k'}^{j,+}\big)^{\tau^{-1}U  } P_{\le \delta}^{\tau^{-1} U e_1} P_{\tau\lambda}  f \|_{L^p} \B\}   \\
	&\quad +	\max_{\substack{(\mathbb S_\ell, \mathbb S_{\ell'}): \\ \text{antipodal}}} 
		\sup_{0<\delta\lesssim 1}    \sup_{\jc<j \le j_\ast} \sup_{k\sim k'} 
			\B\{ (2^{2j}/\delta)^\frac 1p \| \big(P_{k,k'}^{j,-}\big)^{\tau^{-1  }U  } P_{\le \delta}^{\tau^{-1} U e_1} P_{\tau\lambda}  f \|_{L^p} \B\}  \\
	&\quad +	\sup_{0<\delta\lesssim 1}   (\lambda/\delta)^{\frac 1p} \|   P_{\le \delta}^{\tau^{-1} U e_1} P_{\tau\lambda}  f \|_{L^p}.
\end{aligned}
\end{equation}
Hence, we get  \eqref{key-avr} by Lemma \ref{s-average}, Lemma \ref{p-average}, and Lemma \ref{n-average}. 
\end{proof}

\section{Key estimates :  asymptotically   vanishing   averages }\label{sec_key}
In this section, we assemble the various estimates in the previous sections and obtain the estimates that are the key ingredients for the proofs of Theorem \ref{main} and Theorem \ref{Schrodinger}. 
 
\begin{prop} \label{low-high} 
Let $0\le \kappa \le 1$, $\tau\gg 1$, $\frac{d}{2-\kappa}\le p<\infty$, and let $g \in L^p(\mathbb R^d)$ with $\supp g\subset B_d(0,\tau)$.  Suppose \eqref{Large} holds. Then, we have 
\Be \label{result1} 
\|\mathcal M_{m^\kappa_\tau (D) g}\|_{X_{\zeta(1),1/\tau}^{1/2} \to X_{\zeta(1),1/\tau}^{-1/2}}   
	\lesssim \tau^{2-\kappa-\frac d p} \|g\|_{L^p} + \sum_{\frac1\tau \le \lambda \lesssim 1\colon \rm{dyadic} }   \lambda^{\nu} \tau^{\mu} \mathcal A_{p}(g, \lambda, e_1).
\Ee
\end{prop}

Therefore, by Proposition \ref{bilinear-sum}, the estimate \eqref{result1} holds provided that $d\ge3$, $0\le \kappa \le 1$, $\max\{\frac{d+1}2, \frac{d}{2-\kappa}\}\le p<\infty$, $\nu=\kappa-\frac2p$, and $\mu>\beta_d(p)$. Moreover, when $3\le d\le 8$ and $p\ge d$, we can exploit Proposition \ref{bilinear-sum-2} to obtain better bounds \eqref{result1} with $\mu>\gamma_d(p)$.

Recalling the definitions of $Q_\mu^\tau$ and $Q_{\le \mu}^\tau$ in Section \ref{loc_est},  we define the Fourier multiplier operator $Q^{\tau}_{> {\mu}}$ by $Q_{>\mu }^{\tau} u   := {u} -  Q_{\le \mu}^{\tau} u$. 

\begin{proof}[Proof of Proposition \ref{result1}]
To begin with, we fix a small number  $\delta_\circ\in [2^{-3}\ec, 2^{-2}\ec]$. It is easy to see that
\begin{gather}
\label{A}
	\| |D| Q^{\,1}_{> {\delta_\circ}} u\|_{L^2(\R^d)} \lesssim  \| u\|_{X_{\zeta(1),1/\tau}^{1/2}}, \\
\label{B}
	\|Q^{\,1}_{>  {\delta_\circ}} v\|_{L^\frac{2d}{d-2}(\R^d)} \lesssim \| Q^{\,1}_{>  {\delta_\circ}} v\|_{H^{1} } \lesssim \| v\|_{X_{\zeta(1),1/\tau}^{1/2}}, \\
\label{C}
	\| |D| Q^{\,1}_{\le {\delta_\circ}} u\|_{L^\frac{2d}{d-2}(\R^d)} \lesssim  \| u\|_{X_{\zeta(1),1/\tau}^{1/2}}, \quad 
	\| Q^{\,1}_{\le {\delta_\circ}} v\|_{L^\frac{2d}{d-2}(\R^d)} \lesssim  \|v\|_{X_{\zeta(1),1/\tau}^{1/2}}.
\end{gather}
The  estimate in  \eqref{A} follows from  \eqref{geom_symb}. The  estimate \eqref{B} is a consequence of the Hardy-Littlewood-Sobolev inequality, and   \eqref{C}  follows from \eqref{x-1/2-0}   and rescaling.  Setting
\begin{gather*}
	I =|\langle( m^\kappa_\tau (D)   g) Q^{\,1}_{> {\delta_\circ}} u,Q^{\,1}_{>  {\delta_\circ}} v\rangle |,  \quad 
	I\!I = |\langle (m^\kappa_\tau (D)   g) Q^{\,1}_{>  {\delta_\circ}} u, Q^{\,1}_{\le {\delta_\circ}} v\rangle |, \\ 
	I\!I\!I=|\langle (m^\kappa_\tau (D)   g) Q^{\,1}_{\le {\delta_\circ}} u,Q^{\,1}_{ >  {\delta_\circ}} v\rangle |,    \quad
	I\!V = |\langle (m^\kappa_\tau (D) g) Q^{\,1}_{\le {\delta_\circ}} u,Q^{\,1}_{\le {\delta_\circ}} v\rangle |  , 
\end{gather*}
we have
\[	|\langle (m^\kappa_\tau (D)   g) u,v\rangle | \le   I +I\!I+ I\!I\!I+ I\!V.	\]
The estimates for $ I,$ $ I\!I$, and $ I\!I\!I$ are easy to show. Indeed,  by H\"older's inequality  
\[	I\!=\big|\Langle {m}_\tau^\kappa(D) |D|^{-1}  g,   |D| ( \overline{Q^{\,1}_{> {\delta_\circ}} u}\,Q^{\,1}_{>  {\delta_\circ}} v)\Rangle \big| 
	\le \| m_\tau^\kappa(D)|D|^{-1}  g\|_{L^d}  \|   |D| ( \overline{Q^{\,1}_{> {\delta_\circ}} u}\,Q^{\,1}_{>  {\delta_\circ}} v)\|_{L^\frac d{d-1}}.	\]
The Hardy-Littlewood-Sobolev inequality and the fractional Leibniz rule  (see, for example, \cite{MusSch, FGO})  yield
\begin{align*}
I	&\lesssim \| g\|_{L^\frac{d}{2-\kappa}} \Big(  \|  |D| \overline{Q^{\,1}_{> {\delta_\circ}} u}\|_{L^2}  \|Q^{\,1}_{>  {\delta_\circ}} v\|_{L^{\frac{2d}{d-2}}} + \ \| Q^{\,1}_{> {\delta_\circ}}   u\|_{L^{\frac{2d}{d-2}}} \|| D| Q^{\,1}_{>{\delta_\circ}} v\|_{L^{2}}  \Big) \\
	& \lesssim \tau^{2-\kappa-\frac d p} \|g\|_{L^p} \|   u\|_{X_{\zeta(1),1/\tau}^{1/2}}  \|  v\|_{X_{\zeta(1),1/\tau}^{1/2}},
\end{align*}
where the last inequality follows from \eqref{A} and \eqref{B}. Using \eqref{A}, \eqref{B}, and \eqref{C}, the same argument gives 
\begin{align*}
I\!I	&\lesssim \| g\|_{L^\frac{d}{2-\kappa}} \Big( \| |D| \overline{Q^{\,1}_{> {\delta_\circ}} u}\|_{L^2}  \|    Q^{\,1}_{\le  {\delta_\circ}} v\|_{L^{\frac{2d}{d-2}}} + \ \|    Q^{\,1}_{> {\delta_\circ}}   u\|_{L^2} \| |D| Q^{\,1}_{\le {\delta_\circ}} v\|_{L^{\frac{2d}{d-2}}}  \Big) \\
	& \lesssim  \tau^{2-\kappa-\frac d p} \|g\|_{L^p} \|   u\|_{X_{\zeta(1),1/\tau}^{1/2}}  \|  v\|_{X_{\zeta(1),1/\tau}^{1/2}} .
\end{align*}
Similarly, we  have $ I\!I\!I \lesssim  \tau^{2-\kappa-\frac d p} \|g\|_{L^p} \|   u\|_{X_{\zeta(1),1/\tau}^{1/2}}  \|  v\|_{X_{\zeta(1),1/\tau}^{1/2}} $ by interchanging the roles of $u$ and $v$. Therefore 
\[	I +I\!I+ I\!I\!I \lesssim  \tau^{2-\kappa-\frac d p}  \|g\|_{L^p} \|   u\|_{X_{\zeta(1),1/\tau}^{1/2}}  \|  v\|_{X_{\zeta(1),1/\tau}^{1/2}}.	\] 

Now we consider ${I\!V}$ that is given by the low frequency parts $Q^{\,1}_{\le{\delta_\circ}} u$ and $Q^{\,1}_{ \le{\delta_\circ}} v$. By Littlewood-Paley decomposition, we have 
\[	|I\!V|	\le
	 \big|\Lang (m^\kappa_\tau (D) P_{\le \frac1\tau}  g)  Q^{\,1}_{\le{\delta_\circ}} u,  Q_{\le{\delta_\circ}}^{\,1}  v \Rang \big|   
	+ \sum_{\frac1\tau \le \lambda \lesssim 1\colon \rm{dyadic}  }   \big|\Lang (m^\kappa_\tau (D)  P_\lambda  g)  Q_{\le {\delta_\circ}}^{\,1} u,  Q_{\le {\delta_\circ}}^{\,1}  v \Rang \big|,	\]
where 
\begin{equation} \label{proj}
	\F(P_{\le r} u)(\xi) :=\bigg(1-\sum_{j\ge0}\beta(2^{-j} r^{-1}|\xi|) \bigg)\wh u(\xi) , \  \ \  P_{>r} u:= u-  P_{\le r} u.
\end{equation}
By the definition of $X_{\zeta(1),1/\tau}^{1/2}$  we see that 
 \begin{equation} 
 \label{x2scale}
   \| u\|_{L^2} \lesssim  \tau^{ 1/2} \|u\|_{X_{\zeta(1),1/\tau}^{1/2}}.
 \end{equation} 
  It follows from  Bernstein's inequality and Mikhlin's multiplier theorem  that
\begin{align} 
\notag
	\big|\Lang (m^\kappa_\tau (D)  P_{\le \frac1\tau}  g)   Q_{\le {\delta_\circ}}^{\,1}  u,  Q_{\le {\delta_\circ}}^{\,1} v \Rang \big|  
	&\le \|m^\kappa_\tau (D) P_{\le \frac1\tau g} \|_{ L^{\infty}} \| Q_{\le {\delta_\circ}}^{\,1} u\|_{L^2} \| Q_{\le {\delta_\circ}}^{\,1} v\|_{L^2} \\
\label{near-orgin}
	& \lesssim \tau^{1-\frac d p}  \| m^\kappa_\tau (D) P_{\le \frac1\tau} g\|_{L^p}  \|u\|_{X_{\zeta(1),1/\tau}^{1/2}} \|v\|_{X_{\zeta(1),1/\tau}^{1/2}} \\
\notag
	&  \lesssim \tau^{1-\kappa-\frac d p} \|g\|_{L^p} \| u \|_{X_{\zeta(1),1/\tau}^{1/2}} \| v \|_{X_{\zeta(1),1/\tau}^{1/2}}.
\end{align}
Finally, applying the assumption \eqref{Large},  we obtain 
\[	\sum_{\frac1\tau \le \lambda \lesssim 1 }    \big|\Lang (m^\kappa_\tau (D) P_\lambda  g)  Q_{\le {\delta_\circ}}^{\,1} u,  Q_{\le {\delta_\circ}}^{\,1}  v \Rang \big|    
	\lesssim
\!\!\!  \sum_{\frac1\tau \le \lambda \lesssim 1 } \!\!  
\lambda^{\nu} \tau^{\mu} {\mathcal A_{p}(g, \lambda,  e_1)}  \|u\|_{X_{\zeta(1),  1/\tau}^{1/2}}\|v\|_{X_{\zeta(1), 1/\tau}^{1/2}}.	\]
This completes the proof.
\end{proof}

As seen in Lemma \ref{equiv}, by scaling  we can obtain an estimate  in terms of  $X_{\zeta(\tau, U)}^{1/2}$ and $X_{\zeta(\tau, U)}^{-1/2}$ that is equivalent to  \eqref{result1}.  Here 
\Be\label{zeta}\zeta(\tau,U) = \tau U(e_1 -i e_2) \in \mathbb C^{d},  \quad  U\in O_d. \Ee
\begin{cor}\label{scaled}
Let $\tau \gg 1$ and $g \in L^p(\mathbb R^d)$ with $\supp g\subset B_d(0,C)$.  Suppose \eqref{result1} holds.  Then, with $  s =  \mu+\frac d p -2+\kappa$ we have
\begin{equation} \label{result2} 
\| \mathcal M_{m^\kappa  (D) g}  \|_{ X_{\zeta(\tau,U) }^{1/2} \to  X_{\zeta(\tau, U)}^{-1/2} }  
	\lesssim \| g\|_{L^p} + \sum_{\frac1\tau \le \lambda \lesssim 1\colon {\rm dyadic} }  \lambda^\nu   \tau^{s} \big[ \tau^{d-\frac d p } \mathcal A_{p}(g_{\tau U}, \lambda,  e_1) \big] .
\end{equation} 
\end{cor}

\begin{proof}
By Parseval's identity and change of variables $\xi \to \tau U\xi$,  we see 
\[
\langle  (m^\kappa  (D) g)  u,v \rangle  =\tau^{\kappa+2d}  \langle (  m^\kappa_\tau ( UD) g_{\tau U})  u_{\tau U}, v_{\tau U} \rangle .
\]
Since $g_{\tau U}$ is supported in a ball of radius $\simeq \tau$, 
applying \eqref{result1} to the right hand side of the above we have 
\[	|\langle (m^\kappa  (D) g) u,v\rangle| \lesssim 
        \tau^{\kappa+2d} \Big(  \tau^{2-\kappa-\frac d p} \|g_{\tau U}\|_{L^p} + \sum_{\frac1\tau \le \lambda \lesssim 1 }  
        \lambda^\nu \tau^\mu \mathcal A_p (g_{\tau U}, \lambda,  e_1) \Big)  \| u_{\tau U}\|_{X_{\zeta(1),1/\tau}^{1/2}} \| v_{\tau U}\|_{X_{\zeta(1),1/\tau}^{1/2}} .
\]
By \eqref{scaledX}, we have $\| u_{\tau U}\|_{X_{\zeta(1),1/\tau}^{1/2}}  \| v_{\tau U}\|_{X_{\zeta(1),1/\tau}^{1/2}} = \tau^{-d-2} \| u \|_{X_{\zeta(\tau,U) }^{1/2}} \| v \|_{X_{\zeta(\tau, U)}^{1/2}}$ and $\| g_{\tau U} \|_{L^p} $ $=\tau^{-d+\frac d p} \| g\|_{L^p}$. Thus \eqref{result2} follows. 
\end{proof}

Now  we extend  Corollary \ref{scaled} to $g \in H^{r,p}_c(\mathbb R^{d})$ with $r <0$. Naturally,  one may 
attempt to replace $g$ with $(1 + |D|^2)^{\frac r 2} g$ in \eqref{result1}  while taking $m^{\kappa}(D)=(1 + |D|^2)^{-\frac r 2}$. However, 
this simple strategy does not work  since compactness of the support of $(1 + |D|^2)^{\frac r 2} g$ is not guaranteed.  We need to slightly modify the argument using  the following easy lemma. 
 
\begin{lem}\label{elementary} Let $1<p\le q<\infty$. If $s_1<s_2$, then   $H^{s_2,q}_c\subseteq H^{s_1,p}_c$.  \end{lem}

Unlike the $L^p$ spaces over a compact set  the inclusion $H^{s,q}_c\subseteq H^{s,p}_c$ with $p< q$ does not seem to be true in general unless $s$ is an integer.  Failure of the embedding $W^{s,q}_c\subseteq W^{s,p}_c$ with $p< q$ and  non-integer $s$ was shown by  Mironescu and Sickel \cite{MS}. However, if we sacrifice a little bit of regularity such embedding remains true. Though this is easy to show, we  couldn't find a proper reference, so we include a proof. 

\begin{proof}[Proof of Lemma \ref{elementary}]
If $p=q$ the inclusion is clear by Mikhlin's multiplier theorem, so it is enough to consider the case $p<q$. Without loss of generality we may assume that $f$ is supported in $B_d(0,1)$.  Let  $\psi$ be a smooth function supported in  $B_d(0,3/2)$  and $\psi=1$ on $B_d(0,1)$.  We consider the operator $T(f)=\psi f$. It is  sufficient to show  $\|  Tf\|_{H^{s_1,p}} \lesssim \|f\|_{H^{s_2,q}}.$  Trivially $\|Tf\|_{H^{s,r}}\lesssim \|f\|_{H^{s,r}}$  for any $s$ and $1\le r\le \infty$. Thus by interpolation it is enough to show that, for any $\epsilon>0$,
\Be\label{e-sob}
	\|(1+|D|^2)^{-\frac{\epsilon}2} Tf \|_{L^1} \lesssim \|f\|_{L^\infty}.
\Ee 
Using the typical dyadic decomposition we write $(1+|\xi|^2)^{-\frac \epsilon 2}=  \beta_0(\xi)  + \sum_{k\ge 1} 2^{-\epsilon k} \beta_k(\xi),$ where  $\beta_0$ is a smooth function supported in $B_d(0,1)$, and $\beta_k$  is a smooth function supported in $\{\xi:  2^{k-2}\le |\xi|\le 2^{k} \}$ satisfying $|\partial^\alpha  \beta_k|\lesssim 2^{-|\alpha|k}$ for any multi-index $\alpha$.  Let us set  $P_k f=\mathcal F^{-1}(\beta_k \widehat f\,)$. Since $\psi$ is supported in  $B_d(0,3/2)$,  from the rapid decay of $\mathcal F^{-1}(\beta_k(2^k\cdot))$ we have, for  any $|x|\ge 2$ and $N$, 
\[  | P_k Tf(x)|  \lesssim 2^{-Nk}  \int \frac{|\psi(y) f(y)| }{(1+ 2^k|x-y|)^{N}}  dy.  \]
Thus, it  follows that 
\[	\| P_k Tf\|_{L^1} \le \| P_k Tf\|_{L^1(B_d(0,2))}+ \| P_k Tf\|_{L^1(\mathbb R^d\setminus B_d(0,2))}
  	\lesssim  (1+ 2^{-Nk})\|f\|_{L^\infty}.	\]
Clearly, $  \|(1+|D|^2)^{-\frac{\epsilon}2} Tf \|_{L^1}  \lesssim \sum_{k} 2^{-\epsilon k} \| P_k Tf\|_{L^1}$.   So, summation along $k$ gives the desired estimate \eqref{e-sob}. 
\end{proof}

\begin{cor}\label{scaled-n}
 Let $-1 \le  r \le 0$, $\frac d{2+r}\le p<\infty$, $\tau\gg 1$,  and let $m^r(\xi)=(1+|\xi|^2)^\frac r2$. Suppose that \eqref{Large} holds and  $f$ is supported in a bounded set. Then, for any $\epsilon>0$, 
\begin{equation}\label{result3}
\| \mathcal M_{f}  \|_{ X_{\zeta(\tau,U) }^{1/2} \to  X_{\zeta(\tau, U)}^{-1/2} }  
\lesssim 
 \| f\|_{H^{r +\epsilon,p}} +\!\!\! \sum_{\frac1\tau \le \lambda \lesssim 1\colon {\rm dyadic} }   
 \lambda^\nu \tau^{\mu +\frac d p -2}  \big[ \tau^{d-\frac d p } \mathcal A_{p}\big(  m^r_\tau(D) f_{\tau U}, \lambda,  e_1 \big) \big],
\end{equation}
where $m^r_\tau(\xi)=\tau^{-r}m^r(\tau\xi)=(\tau^{-2}+|\xi|^2)^{\frac r2}$ as in Definition \ref{mk}. 
\end{cor}

\begin{proof} 
Let us set $\kappa=-r$ and take $m^\kappa(D)=(1+|D|^2)^{\frac \kappa 2}$, $g=(1+|D|^2)^{-\frac \kappa 2} f$ so that $f=m^\kappa(D) g$. Scaling shows that $\langle f u,v\rangle=\tau^{\kappa+2d}\langle (m_\tau^{\kappa}(D)g_{\tau U})u_{\tau U}, v_{\tau U}\rangle$.  As before in the proof of Proposition \ref{result1}, we decompose frequencies of the bilinear operator to get
\[ |\langle f u,v\rangle |  \le  \tau^{\kappa+2d}\big( I +I\!I+ I\!I\!I+ I\!V \big), \] 
where 
\begin{gather*} 
I =|\langle (m_\tau^\kappa(D)g_{\tau U})  Q^{\,1}_{> {\delta_\circ}}  u_{\tau U} ,Q^{\,1}_{>   \delta_\circ } v_{\tau U} \rangle |, \quad
I\!I = |\langle (m_\tau^\kappa(D)g_{\tau U}) Q^{\,1}_{>  \delta_\circ}  u_{\tau U}, Q^{\,1}_{\le \delta_\circ} v_{\tau U} \rangle |, \\
I\!I\!I=|\langle  (m_\tau^\kappa(D)g_{\tau U}) Q^{\,1}_{\le \delta_\circ} u_{\tau U} , Q^{\,1}_{ > \delta_\circ} v_{\tau U}  \rangle |, \quad
I\!V = |\langle  (m_\tau^\kappa(D)g_{\tau U}) Q^{\,1}_{\le \delta_\circ} u_{\tau U} , Q^{\,1}_{\le \delta_\circ}  v_{\tau U} \rangle |.
\end{gather*}
Then, following the argument  in the proof of Proposition \ref{low-high} (then rescaling back), it is easy to see that, for $r\in [-1,0]$, 
\[ 
 {I +I\!I+ I\!I\!I}\lesssim \tau^{r-2d} \| (1+|D|^2)^{\frac r2} f\|_{L^\frac{d}{2+r}}   \|   u\|_{X_{\zeta(\tau, U)}^{1/2}}  \|  v\|_{X_{\zeta(\tau, U)}^{1/2}}. 
\] 
Using  Lemma \ref{elementary},  for any $\epsilon>0$, $-1\le r\le 0$, and $\frac{d}{2+r} \le p<\infty$,  we have
\[
{I +I\!I+ I\!I\!I}\lesssim \tau^{r-2d}  \|  f\|_{H^{r+\epsilon, p}} \|   u\|_{X_{\zeta(\tau, U)}^{1/2}}  \|  v\|_{X_{\zeta(\tau, U)}^{1/2}}.  
\]
For the remaining $I\!V$, we may routinely  repeat the same argument as before  making use of \eqref{Large} with $\kappa=-r$   to get 
\[   I\!V  \lesssim \tau^{r-2d} \bigg( \|f\|_{H^{r+\epsilon, p}} +   \sum_{\frac1\tau \le \lambda \lesssim 1 }   
 \lambda^\nu \tau^{\mu +\frac d p -2}  \big[ \tau^{d-\frac d p } \mathcal A_{p}\big(  m^r_\tau(D) f_{\tau U}, \lambda,  e_1 \big) \big] \bigg) \|u\|_{X^{1/2}_{\zeta(\tau, U)}} \|v\|_{X^{1/2}_{\zeta(\tau, U)}}.\]
Combining the estimates for $I+I\!I+I\!I\!I$ and $I\!V$ yields \eqref{result3}. 
\end{proof}

\subsection{Non-averaged estimates} 
We recall the following ((13), (15), and (16) in  \cite{Haberman15}) which are immediate consequences of \eqref{x2} and \eqref{x-1/2-0}:
\begin{align}
\label{ha1}
	\|\mathcal M_q\|_{X_{\zeta(\tau,U)}^{1/2}\to X_{\zeta(\tau,U)}^{-1/2}}& \lesssim \tau^{-1} \| q\|_{L^\infty} ,\\
\label{ha2}
	\|\mathcal M_q\|_{X_{\zeta(\tau,U)}^{1/2}\to X_{\zeta(\tau,U)}^{-1/2}} &\lesssim  \| q\|_{L^\frac d2(\R^d)}, \\
\label{x-1/2}
	\|f\|_{X_{\zeta(\tau,U)}^{-1/2}}&\lesssim \|f\|_{L^\frac{2d}{d+2}(\R^d)}.
\end{align}

From  \cite[Lemma 2.3]{HT} we also have
\begin{equation}  \|\mathcal M_{\nabla q}\|_{X_{\zeta(\tau,U)}^{1/2}\to X_{\zeta(\tau,U)}^{-1/2}} \lesssim  \| q\|_{L^\infty}.
 \label{hoho}
 \end{equation}
 
The estimates in the following proposition correspond to the estimates which one can get by formally interpolating the estimates \eqref{ha2} and \eqref{hoho}. 
\begin{prop}
\label{decay} 
Let  $d/2\le p\le\infty$, $\kappa=1-\frac d{2p}$, $\tau\gg 1$, and let $f\in L^p_c(B_d(0,C))$. Suppose
that  $m^\kappa(\xi)=(1+|\xi|^2)^\frac\kappa 2$ or $m^\kappa(\xi)=|\xi|^\kappa$. Then, 
\begin{equation}\label{interpol}
	\|\mathcal M_{m^\kappa (D) f}\|_{X_{\zeta(\tau,U)}^{1/2}\to X_{\zeta(\tau,U)}^{-1/2}} \lesssim  \| f\|_{L^p(\Rd)}.
\end{equation}
\end{prop}
  
\begin{proof}
When $\kappa=0$ the estimate \eqref{interpol} is identical with \eqref{ha2}, so it is enough to prove \eqref{interpol} with $\kappa=1$. For the purpose, it is more convenient to work with the rescaled space $X_{\zeta(1),1/\tau}^{1/2}$.  We note that $m^\kappa_\tau(\xi)$  takes the particular forms $(\tau^{-2}+|\xi|^2)^\frac \kappa2$,  $|\xi|^\kappa$. We regard $m^\kappa_\tau(D)$ as operators embedded in  an analytic  family  of operators $m^z_\tau(D)$  with complex parameter $z$. We claim that 
\Be\label{goal}
	|\inp{m^{\kappa+it}_\tau  (D) fu,v}|  \lesssim  (1+|t|)^{d+1} \tau^{\kappa} \|f\|_{L^p}   \|   u\|_{X_{\zeta(1),1/\tau}^{1/2}}  \|  v\|_{X_{\zeta(1),1/\tau}^{1/2}} 
\Ee 
whenever $f$ is supported in $B_d(0, C\tau)$.  This gives \eqref{interpol} by Lemma \ref{equiv}. By Stein's interpolation along analytic family  we only need to  show  \eqref{goal} for the cases $p=\infty$ ($\kappa=1$) and $p=d/2$ ($\kappa=0$).  H\"older's inequality, \eqref{x-1/2-0}, and  \eqref{scaledX}  yield  $|\inp{m^{\kappa+it}_\tau  (D) fu,v}|\le  \|m^{\kappa+it}_\tau  (D) f \|_{L^{d/2}}  \|u\|_{X_{\zeta(1),1/\tau}^{1/2}} \|  v\|_{X_{\zeta(1),1/\tau}^{1/2}} $. Thus,  \eqref{goal} with $\kappa=0$ follows from Mikhlin's multiplier theorem. It remains to show \eqref{goal} with $\kappa=1$. 
   
For the purpose we decompose 
\[ u= u_{0}  +  u_{1}  := P_{\le 8 } u+ P_{>  8} u,  \quad  v= v_{0}  +  v_{1}  :=  P_{\le 8} v+ P_{>  8} v.\] 
Here $P_{> r}$ and $P_{\le r}$ are given by \eqref{proj}. Since $\inp{ m^{\kappa+it}_\tau(D) fu_0,v_0}= \inp{ f, m^{\kappa-it}_\tau(D)(\overline{u_0}v_0)}$ and $\widehat {u_0}$ and $\widehat {v_0}$ are compactly supported (so, we may disregard the multiplier operator $m^{\kappa-it}_\tau  (D)$ because $\|m^{\kappa-it}_\tau  (D) g\|_{L^1}\lesssim(1+|t|)^{d+1} \|g\|_{L^1}$ whenever $\widehat g$ is supported in $B_d(0,C)$ and  $\kappa>0$\footnote{It is not difficult to show that $\| m^{\kappa-it}_\tau  (D) \psi\|_{L^1} \lesssim (1+|t|)^{d+1}$ if $ \psi\in \mathcal S(\mathbb R^d)$ provided that $\kappa>0$. }), we have the estimate 
\begin{align*}
|\inp{m^{1+it}_\tau  (D) fu_0,v_0}|  
	&\lesssim  (1+|t|)^{d+1}  \|f\|_{L^\infty} \|\overline{u_0}v_0\|_{L^1} \\
	&\lesssim  \tau  (1+|t|)^{d+1} \|f\|_{L^\infty} \| u_0\|_{X_{\zeta(1),1/\tau}^{1/2}}  \|  v_0\|_{X_{\zeta(1),1/\tau}^{1/2}} .  
\end{align*}
For the second inequality we used  \eqref{x2scale}.  

On the other hand, we have 
\[	|\inp{m^{1+it}_\tau  (D) fu_1,v_1}|  \le \|f\|_{L^d}  \big( \|m^{1-it}_\tau(D)P_{\le 8}(\overline{u_1}v_1)\|_{L^\frac{d}{d-1}} + \|m^{1-it}_\tau(D)P_{> 8}(\overline{u_1}v_1)\|_{L^\frac{d}{d-1}} \big). \]
For the low frequency part $P_{\le 8} (\overline{u_{1}}v_{1})$, we take $\psi \in \mathcal S(\mathbb R^d)$ such that $\widehat{\psi} =1$ on the Fourier support of $P_{\le 8} (\overline{u_{1}}v_{1})$.  By Young's and H\"older's inequalities and the estimate \eqref{B},
\begin{align*}
\| m_{\tau}^{1- i t}(D) P_{\le 8} (\overline{u_{1}}v_{1}) \|_{L^{ \frac{d}{d-1} }} 
& \le \| m_{\tau}^{1- i t}(D) \psi\|_{L^1} \| P_{\le 8} (\overline{u_{1}}v_{1}) \|_{L^{ \frac{d}{d-1} }} 
		\lesssim (1+|t|)^{d+1} \| \overline{u_1}v_1 \|_{L^{\frac{d}{d-1}}} \\
& \le (1+|t|)^{d+1} \| u_1\|_{L^2} \|v_1 \|_{L^{\frac{2d}{d-2}}} 
		\lesssim (1+|t|)^{d+1} \| u_1\|_{X^{1/2}_{\zeta(1),1/\tau}}  \|v_1 \|_{X^{1/2}_{\zeta(1),1/\tau}}.
\end{align*}
For the high frequency part $P_{>8}(\overline{u_1}v_1)$, we see that $m_{\tau}^{1- i t}(D) |D|^{-1}$ satisfies the assumption in Mikhlin's multiplier theorem. Hence, it follows from the fractional Leibniz rule and the estimate \eqref{x-1/2-0} (with \eqref{scaledX}) that
\begin{align*}
\| m_{\tau}^{1- i t}(D) P_{>8}(\overline{u_1}v_{1}) \|_{L^{ \frac{d}{d-1} }} & \le (1+|t|)^{d+1}  \| |D| ( \overline{u_1} v_1)  \|_{L^{\frac{d}{d-1}} }\\
& \lesssim  (1+|t|)^{d+1} \big(\| |D| \overline{u_1}\|_{L^2} \| v_1 \|_{L^{\frac{2d}{d-2}}} + \| \overline{u_1} \|_{L^{ \frac{2d}{d-2} }} \||D| v_1\|_{L^2} \big) \\
& \lesssim (1+|t|)^{d+1} \| u_1\|_{X^{1/2}_{\zeta(1),1/\tau}}  \|v_1 \|_{X^{1/2}_{\zeta(1),1/\tau}}.
\end{align*}
Since $f$ is supported in $B_d(0,C\tau)$, we obtain  
\begin{align*}
|\inp{m^{1+it}_\tau  (D) fu_1  ,v_1}| 
	&\lesssim    (1+|t|)^{d+1} \|f\|_{L^d}  \|   u_1\|_{X_{\zeta(1),1/\tau}^{1/2}}  \|  v_1\|_{X_{\zeta(1),1/\tau}^{1/2}} \\
	&\lesssim  \tau (1+|t|)^{d+1} \|f\|_{L^\infty}  \|   u_1\|_{X_{\zeta(1),1/\tau}^{1/2}}  \|  v_1\|_{X_{\zeta(1),1/\tau}^{1/2}}.
\end{align*} 

Finally, it is enough to consider $\inp{m^{1+it}_\tau  (D) fu_0,v_1}$ since the remaining $\inp{m^{1+it}_\tau  (D) fu_1,v_0}$ can be handled similarly. Since  $\| |D|\overline{u_0} \|_{L^\frac{2d}{d-2}} \lesssim   \|\overline{u_0}\|_{L^\frac{2d}{d-2}} $, repeating the above argument,   we have 
\begin{align*}
|\inp{m^{1+it}_\tau  (D) fu_0,v_1}|
	&\lesssim  (1+|t|)^{d+1} \|f\|_{L^d}  \| u_0\|_{X_{\zeta(1),1/\tau}^{1/2}}  \|  v_1\|_{X_{\zeta(1),1/\tau}^{1/2}} \\
	&\lesssim  \tau  (1+|t|)^{d+1}  \|f\|_{L^\infty}  \|   u_0\|_{X_{\zeta(1),1/\tau}^{1/2}}  \|  v_1\|_{X_{\zeta(1),1/\tau}^{1/2}}.
\end{align*} 
Thus, combining all the estimates together,  we see that \eqref{goal} holds with $\kappa=1$. 
\end{proof}

\begin{cor}
Let $p = \frac{d}{2(1-\kappa)}$ for $0\le \kappa <1$. If $f$ is supported in  a bounded set,  for any $\epsilon>0$ and $\tau \gg 1$
\begin{equation}\label{interpol2}
	\|\mathcal M_{ f}\|_{X_{\zeta(\tau,U)}^{1/2}\to X_{\zeta(\tau,U)}^{-1/2}} \lesssim  \| f\|_{H^{-\kappa+\epsilon, p}(\Rd)}.
\end{equation}
\end{cor}

\begin{proof}
Let $m^\kappa (D)= (1+|D|^2)^\frac \kappa 2$. Using rescaling and following the argument in the proof of Proposition \ref{decay}, 
we have
$
  |\inp{m^{1}  (D) gu,v}|  \lesssim   ( \| g\|_{L^\infty} +  \|g\|_{L^d}) \|   u\|_{X_{\zeta(\tau,U)}^{1/2}}  \|  v\|_{X_{\zeta(\tau,U)}^{1/2}} 
$
for any $g$ in the Schwartz class. 
 Obviously this gives 
\[
  |\langle gu,v\rangle|  \lesssim   \big( \| m^{-1} (D) g \|_{L^\infty} +  \|m^{-1} (D) g \|_{L^d} \big)   \|   u\|_{X_{\zeta(\tau,U)}^{1/2}}  \|  v\|_{X_{\zeta(\tau,U)}^{1/2}} .
\]
Since the Schwartz class is dense  in $H^{s,p}$, using Lemma \ref{elementary} and the  embedding $ H^{\frac dp+\epsilon,\, p}  \hookrightarrow L^\infty$ , we get
\[
  |\langle fu,v\rangle|  \lesssim  \big( \| m^{-1+\delta} (D) f \|_{L^p} +  \| f\|_{H^{-1+\epsilon, p}} \big)   \|   u\|_{X_{\zeta(\tau,U)}^{1/2}}  \|  v\|_{X_{\zeta(\tau,U)}^{1/2}}
\]
provided that $d\le p<\infty$, $\epsilon>0$, and $\delta>\frac dp$. 
Now, taking $p$ arbitrarily close to $\infty$ in the above and interpolating the estimate with   
\[ |\inp{ fu,v}|  \lesssim   \|f\|_{L^\frac d2} \|   u\|_{X_{\zeta(\tau,U)}^{1/2}}  \|  v\|_{X_{\zeta(\tau,U)}^{1/2}},\]
 which is equivalent to \eqref{ha2},  we get the bound \eqref{interpol2} on the desired range.
\end{proof}

\subsection{Convergence of the  averages to zero}  We now show that  averages of $\| q\| _{\bdot{X}_{\zeta(\tau,U)}^{-1/2}}$, $\|\mathcal M_q\|_{X_{\zeta(\tau,U)}^{1/2}\to X_{\zeta(\tau,U)}^{-1/2}}$ over $U$ and $\tau$ asymptotically vanish as $\tau\to\infty$.   Compared with the non-averaged counterpart, averaged estimates allow a considerable amount of regularity gain. 

\begin{prop} 
Let $d\ge 3$, $0\le \kappa \le 1$, and $\tau\gtrsim 1$.  Then,  we have
\begin{align} \label{g}
 \int_{O_d} \| m^\kappa (D)  g \|^2 _{X_{\zeta(\tau,U)}^{-1/2}} dm(U) \lesssim \| g\|_{L^\frac{2d}{d+2-2\kappa}(\Rd)}^2.
\end{align}
\end{prop}

\begin{proof} In order to show \eqref{g} it is enough to consider the case $\kappa=1$.  If $0\le \kappa<1$, from \eqref{g} with $\kappa=1$ and the Plancherel theorem, we have 
\[  \int_{O_d} \|  m^\kappa (D)  g \|^2 _{X_{\zeta(\tau,U)}^{-1/2}} dm(U) \lesssim \|(1+|D|^2)^{-\frac12}m^\kappa(D)g\|_{L^2}^2 \lesssim \|  |D|^{\kappa-1}   g\|_{L^2}^2. \] 
Thus, the desired estimate  follows by the Hardy-Littlewood-Sobolev inequality. 

To show \eqref{g} with $\kappa=1$, we break $g$ into $g=P_{\le 8\tau}\, g+P_{> 8\tau}\, g$ where  $P_{\le 8\tau}$ and $P_{> 8\tau}$ are given by \eqref{proj}.  Since  $ | p_{\zeta(\tau,U)}(\xi)| \gtrsim   |\xi|^2$   and  $|m^1(\xi)|\lesssim |\xi|$ for $ |\xi|\ge 4\tau $, 
\begin{equation}\label{g1} 
\|  m^1 (D)  P_{> 8\tau} \, g \| _{X_{\zeta(\tau,U)}^{-1/2}} \lesssim \|g\|_{L^2}.
\end{equation}
Noting that  $|p_{\zeta(\tau,U)}(\xi)| \simeq  |\xi| \big( |-|\xi|  + 2 \tau Ue_2 \cdot \frac{\xi }{|\xi|} | + |2\tau U e_1\cdot \frac{\xi} { |\xi|}| \big)$  and using \eqref{haar_msr},  we see 
\begin{align*}
\int_{O_d} \| m^1 (D)  P_{\le 8\tau} g \|^2 _{X_{\zeta(\tau,U)}^{-1/2}} dm(U) 
	&\lesssim \int_{O_d} \int_{|\xi|\le 8\tau} \frac{1+|\xi|^2}{\tau+ |p_{\zeta(\tau,U)}(\xi)|}\,   \, | \widehat{g}(\xi)|^2 d\xi \, dm(U) \\
	&\lesssim \bigg(1+\sup_{|\xi|\le 8\tau} \int_{\mathbb S^{d-1}} F(|\xi|/2\tau, \omega) d\sigma(\omega)\bigg) \| g\|_{L^2}^2 ,
\end{align*}
where $  F(r,  \omega ) = \big( |  e_2 \cdot \omega  -r  | + |  e_1\cdot \omega | \big)^{-1}.$  Taking into account symmetry of the sphere, it is clear that $\sup_{|\xi|\le 16\tau} \int_{\mathbb S^{d-1}} F(|\xi|/2\tau, \omega) d\sigma(\omega)\lesssim  \int_{\mathbb S^{d-1}} F(1, \omega) d\sigma(\omega).$  By change of variables $ \tau  \omega = \eta \in \mathbb R^{d}$, we see that  $\int_{\mathbb S^{d-1}} F(1, \omega) d\sigma(\omega)\le C_d$ for $d\ge 3$.  Thus, we get 
\[ \int_{O_d} \| m^1  (D)   P_{\le 8\tau} g \|^2 _{X_{\zeta(\tau,U)}^{-1/2}} dm(U) 
	\lesssim \| g\|_{L^2}^2 .\] 
This and \eqref{g1} yield \eqref{g} with $\kappa=1$. 
\end{proof}

\begin{cor}\label{with-d-cor}  Let $d\ge 3$, $\frac {2d}{d+2} \le p < \infty $, and $s\ge  \max\{-1, - \frac{d+2}{2}+\frac d{p} \}$.  
If  $ f\in H_c^{s,p}(\mathbb R^{d})$, then
\begin{align}
\label{q}
\lim_{\tau \to \infty} \int_{O_d} \| f\|^2 _{\bdot{X}_{\zeta(\tau,U)}^{-1/2}} dm(U) =0.
\end{align}
\end{cor}  
\begin{proof}
We may assume that $s = \max\{-1, - \frac{d+2}{2}+\frac d{p} \}$ since  $H^{t, p}\hookrightarrow H^{s,p}$ for $t \ge s$ and $1<p<\infty$. 
By \eqref{lemma2.1}, it is enough to show that \eqref{q} holds with  $\bdot{X}^{-1/2}_{\zeta(\tau,U)}$ replaced by ${X}^{-1/2}_{\zeta(\tau,U)}$.
Note that  $\|h\|_{H^{-1,2}} \lesssim \|h\|_{H^{-1, r}} $  for $2 \le r <\infty $ whenever $h$ is supported in a bounded set.\footnote{This is possible because the order is an integer. In fact, 
it follows from the embedding $H_c^{1, 2}(\Omega) \hookrightarrow H_c^{1,r'}(\Omega)$ for any bounded set $\Omega$ and duality.} Hence, it suffices to consider $s = -  \frac{d+2}{2} + \frac d p$ and $\frac{2d}{d+2} \le p  \le 2$, i.e.,
$f \in H^{-\kappa, \frac{2d}{d+2-2\kappa}}$ for $0\le \kappa \le 1$. 
From  \eqref{g} with $m^\kappa(D) = (1+ |D|^2)^{\frac\kappa2}$, we  have
\begin{equation}
\label{sobolev}\int_{O_d} \|  f \|^2 _{X_{\zeta(\tau,U)}^{-1/2}} dm(U) \lesssim \|f\|_{H^{-\kappa , p }}^2,   \quad  p= \frac{2d}{d+2-2\kappa}. 
\end{equation}
Let $\phi\in C_c^\infty(B_d(0,1))$ such that $\int \phi\, dx=1$. We  write  $f=(f-f\ast \phi_\epsilon)  +f\ast \phi_\epsilon$.  
By Young's convolution inequality and the embedding $H^{1,q} \hookrightarrow H^{\kappa,q} $ for $1<q<\infty$, 
we have 
\begin{align*}
\| f\ast \phi_\epsilon \|_ {X_{\zeta(\tau,U)}^{-1/2}} \! \lesssim \tau^{-\frac12} \| f\ast \phi_\epsilon \|_ {L^2}
	\lesssim   \tau^{-\frac12} \| f\|_{H^{-\kappa,p}}  \|  \phi_\epsilon  \|_{H^{1,\frac{2p }{3p -2}}}  
	\lesssim \tau^{-\frac12}   \epsilon^{-1  -\frac{d(2-p )}{2p }}  \| f\|_{H^{-\kappa,p}}. 
\end{align*}
Combining this with \eqref{sobolev}, we obtain
\[ \bigg( \int_{O_d} \| f  \|^2 _{X_{\zeta(\tau,U)}^{-1/2}} dm(U)\bigg)^{1/2}
\lesssim   \|f-f\ast \phi_\epsilon\|_{H^{-\kappa,p }} + \tau^{-\frac12} \epsilon^{-1  -\frac{d(2-p )}{2p }}  \| f\|_{H^{-\kappa,p }} .
\]
Since $\lim_{\epsilon \rightarrow 0} \| f - f\ast\phi_\epsilon\|_{H^{-\kappa,p}} = 0$,  \eqref{q} follows if we take  $\epsilon=\epsilon(\tau)>0$ such that   $\epsilon^{-1  -\frac{d(2-p )}{2p }} =\epsilon^{\kappa-2}< \tau^{\frac14 }$. 
\end{proof}

\subsection{Average over $\tau$ and $U$\!}  As we have seen in the proof of Proposition \ref{low-high}, to get the desired estimate we do not have to use the averaged estimate for the  high-high, low-high, high-low frequency interactions.  However, in the case of  low-low frequency interaction  we get  significantly improved bounds by means  of  average over  $\tau$ and $U$. 

For simplicity  we define 
\[\avint_{\!\!\!\! M}  f (\tau) d\tau:=  \frac 1M \int_M^{2M}  f (\tau) d\tau.\]
For $M \ge 2$, we set 
 \[   \mathfrak A_{M}^{p,\kappa} (f)= \bigg( \avm \int_{O_d} \|  \mathcal M_{m^\kappa  (D) f}  \|_{X_{\zeta(\tau,U)}^{1/2}\to  X_{\zeta(\tau,U)}^{-1/2}}^p dm(U) d\tau\bigg)^\frac1p.  \] 

\begin{lem} \label{with-d-av}  
Let  $2\le p<\infty$. Suppose we have \eqref{result2}.  If $\nu>s$, then for any  $f\in  W^{s,p}_c(B_d(0, 1))$ the estimate
\begin{equation}\label{a-sob}
\mathfrak A_{M}^{p,\kappa} (f)
\lesssim  
 \|f\|_{W^{s ,p}}
 \end{equation}
holds with
the implicit constant independent of $M$. The same remains valid with $W^{s,p}$  replaced by $ H^{s,p}$.
 \end{lem}

\begin{proof}  
It is well-known that if $1<p<\infty$, then $W^{k,p}=H^{k,p}$ for any $k=0,1,2, \ldots$, and $H^{s,p} \hookrightarrow W^{s-\epsilon,p}$ for any $\epsilon>0$, $s\in\mathbb R$ (see \cite[pp. 168--180]{Triebel}). Hence, it suffices to show that \eqref{a-sob} holds with $f \in W^{s,p}$.  Taking $p$-th power and  integrating  over $U$ and $\tau$ on both side of \eqref{result2}, by Minkowski's inequality we get
\begin{align*}
\mathfrak A_{M}^{p,\kappa} (f)
\lesssim
  \|f\|_{L^p} +  \mathbf A_M(f),
\end{align*}
where
\[
{\bA}_M(f)= 
   \sum_{\frac1{2M} <  \lambda \lesssim 1: dyadic} \lambda^{\nu} M^{s}
\bigg( \avm \int_{O_d} \big[\tau^{d-\frac d p}   \mathcal A_{p}  (f_{\tau U}, \lambda,  e_1)  \big]^p  dm(U) d\tau \bigg)^{\frac1p}.\\
\]
Thus, for \eqref{a-sob} it suffices to show that 
\Be  \label{sum} 
\bigg(\sum_{M\ge 2: dyadic} (\bA_M (f))^p \bigg)^{\frac1p} \le C  \|f\|_{W^{s,p}}.
\Ee
By scaling $\tau \to M \tau$ and applying \eqref{key-avr}, it follows that
\begin{align*}
\avm\! \int_{O_d}  \!\!\! \big[\tau^{d-\frac d p}   \mathcal A_{p}  (f_{\tau U}, \lambda, e_1)   \big]^p dm(U)  d\tau  
	&\lesssim M^{(d-\frac d p)p}   \!\! \int_1^{2}\! \!  \int_{O_d}    \!\!\! \big[ \mathcal A_{p}  (f_{ M\tau U}, \lambda,  e_1)  \big]^p  dm(U) d\tau  \\
	&\lesssim M^{(d- \frac d p)p}  \| \mathcal P_\lambda f_{M}\|_{L^p}^p = \| \mathcal P_{M\lambda}  f\|_{L^p}^p \,.
\end{align*} 
This yields 
$  \bA_M(f)\lesssim \sum_{(2M)^{-1} <  \lambda \lesssim 1: dyadic}   \lambda^{\nu} M^{s}  \| \mathcal P_{M\lambda} f\|_{L^p}.$
Reindexing  $\rho=\lambda M$, we see that 
$ \bA_M  (f) \lesssim \sum_{\frac12<  \rho \lesssim M: dyadic }   (\rho/M)^{\nu-s } \rho^{s}    \| \mathcal P_{\rho} f\|_{L^p}.$ 
Since $ \nu>s$, 
$\sup_M \sum_{ \frac12< \rho\lesssim M} (\rho/M)^{\nu-s}   \lesssim 1$ and $\sup_\rho \sum_{M \gtrsim \rho} (\rho/M)^{\nu-s} \lesssim 1$. 
So, by Schur's test,  
\begin{align*} 
\bigg(\sum_{M >2}   (\bA_M (f))^p\bigg)^\frac1p 
\lesssim \bigg( \sum_{\rho>1/2} \big( \rho^{s }  
\|\mathcal P_{\rho} f\|_{L^p} \big)^p \bigg)^{\frac 1p} \lesssim \|f\|_{W^{s, p}}.
\end{align*} 
Here the last inequality follows from the embedding of $W^{s,p}$ into the Besov space $B_{p,p}^s$ for $2\le p<\infty$ and $s\in\R$ (see \cite[pp. 179--180]{Triebel}).  Hence we get \eqref{sum}.
\end{proof}

In a similar way, using Corollary \ref{scaled-n} we obtain the following. 

\begin{lem} \label{huhu}
Let $2\le p<\infty$. Suppose \eqref{result3} holds with $\nu=-r-\frac 2p$.  If $\mu<2-\frac{d+2}p$ and $\mu+\frac dp-2\le r$, then, for every $\epsilon>0$ and any $f\in H_c^{r+\epsilon, p}(B_d(0,1))$, we have
\begin{equation}\label{a-sob-2}
\mathfrak A_M^p(f) := \bigg( \avm \int_{O_d} \|  \mathcal M_{ f}  \|_{X_{\zeta(\tau,U)}^{1/2}\to  X_{\zeta(\tau,U)}^{-1/2}}^p dm(U) d\tau\bigg)^\frac1p 
\lesssim \|f\|_{H^{r+\epsilon,p}}
\end{equation}
uniformly in $M$.
\end{lem}
 
\begin{proof}
By rescaling and using the identity $m^r_{M\tau}(D)f_{M\tau U} = \tau^{-r} (m^r_{M}(D)f_M)_{\tau U}$, it is not difficult to see that the estimates \eqref{result3} and \eqref{key-avr} imply
\[	\mathfrak A_M^p(f) \lesssim \|f\|_{H^{r+\epsilon, p}} +\sum_{\frac1{2M}<\lambda \lesssim 1} \lambda^\nu M^s  \| \mathcal P_{M\lambda} m^r(D) f\|_{L^p},	\]
where we set  $s = \mu +\frac d p -2-r$. By the assumption on $\mu$ and $\nu$ it is easy to check that $\nu>s$. Repeating the argument in the proof of Lemma \ref{with-d-av} immediately yields 
\[	\mathfrak A_M^p(f) \lesssim \|f\|_{H^{r+\epsilon,p}} + \|m^r(D)f\|_{B^s_{p,p}} \lesssim \|f\|_{H^{r+\epsilon, p}} + \|f\|_{H^{r+s, p}},   \]
where the last inequality follows from the embedding of $H^{s,p}$ into the Besov space $B_{p,p}^s$ for any $-\infty<s<\infty$ and $2\le p<\infty$ (\cite[p. 179]{Triebel}). Thus we get \eqref{a-sob-2}  since $s\le0$. 
\end{proof}

Now, we combine Proposition \ref{bilinear-sum}, Proposition \ref{low-high}, Corollary \ref{scaled}, and Lemma \ref{with-d-av} altogether to conclude the following: \emph{If $d\ge 3$, $0\le\kappa\le 1$, $\max\{\frac{d+1}2, \frac d{2-\kappa}\}\le p <\infty$,  and  $\kappa -\frac2p >s>\beta_d(p)+\frac dp -2+\kappa$,   then we have \eqref{a-sob} for any $f\in W_c^{s,p}(B_d(0, C))$.} Here, let us specify the range of $p$ for the estimate \eqref{a-sob} in Lemma \ref{with-d-av}.

\subsection*{When $3\le d \le 6$} Such $s$ exists only if $\kappa -\frac2p >\beta_d(p)+\frac dp -2+\kappa$, which is equivalent to $p>\frac{d+2}2$.  Hence we need to consider the cases $\frac{d}{2-\kappa}\le \frac{d+2}2$ and $\frac{d}{2-\kappa}> \frac{d+2}2$ separately. 
\vspace{-10pt}
\begin{itemize} [leftmargin=0.15in]
\item If  $0\le\kappa \le \frac4{d+2}$, then \eqref{a-sob} holds whenever $\frac{d+2}{2}< p <\infty$ and $s>\beta_d(p)+\frac dp-2+\kappa$.
\item If $\frac4{d+2}<\kappa\le 1$, then \eqref{a-sob} holds whenever $\frac{d}{2-\kappa}\le p<\infty$ and $s>\beta_d(p)+\frac dp-2+\kappa$.
\end{itemize}

\subsection*{When $d\ge 7$} In this case, such $s$ exists only if $-\frac2p >\beta_d(p)+\frac dp -2$, i.e., $p>\frac{3d+7}{6}$.
\vspace{-10pt}
\begin{itemize} [leftmargin=0.15in]
\item If  $0\le\kappa \le  \frac{14}{3d+7}$, then \eqref{a-sob} holds whenever $\frac{3d+7}{6}< p <\infty$ and $s>\beta_d(p)+\frac dp-2+\kappa$.
\item If $\frac{14}{3d+7}<\kappa\le 1$, then \eqref{a-sob} holds whenever $\frac{d}{2-\kappa}\le p<\infty$ and $s>\beta_d(p)+\frac dp-2+\kappa$.
\end{itemize}

On the other hand, \eqref{ha2} gives 
\begin{equation}\label{fff}
	\mathfrak A_{M}^{\frac d2,\kappa} (f) \lesssim \| m^{\kappa}(D)f\|_{L^\frac d2} \lesssim   \|f\|_{H^{\kappa,\frac d2}}
\end{equation}
for any $\kappa\ge 0$. Interpolating this bound and the estimate \eqref{a-sob} with endpoints $p$ in the above, we can extend the aforementioned ranges of $p$ to the range $p\ge\frac d2$ as the following.

\subsection*{When $3\le d \le 6$} By the definition of $\beta_d$ we need to consider the cases $\frac{d}{2-\kappa}\le 4$ and $\frac{d}{2-\kappa}> 4$ separately.
\vspace{-10pt}
\begin{itemize} [leftmargin=0.15in]
\item If $0\le\kappa \le \frac4{d+2}$, \eqref{a-sob} holds for $\frac d2 \le p \le \frac{d+2}2$ and $s> \frac dp-2+\kappa$.
\item If $\frac4{d+2}<\kappa\le \frac{8-d}{4}$,\footnote{This interval is empty when $d=6$.} \eqref{a-sob} holds for $\frac{d}2\le p\le\frac{d}{2-\kappa}$ and $s>\frac{d}{p}-2+\kappa$.
\item If $\frac{8-d}{4}<\kappa\le 1$,\footnote{This is void when $d=3$ or $4$.} \eqref{a-sob} holds for $\frac{d}2\le p\le\frac{d}{2-\kappa}$ and $s>\kappa+\frac1\kappa(2-\frac dp)(\frac12-\frac{2(2-\kappa)}d-\kappa)$.
\end{itemize}

\subsection*{When $d\ge 7$}  We consider the cases $\frac{d}{2-\kappa} \le \frac{d+9}{2}$ and $\frac{d}{2-\kappa}>\frac{d+9}{2}$ separately.
\vspace{-10pt}
\begin{itemize} [leftmargin=0.15in]
\item If  $0\le\kappa \le \frac{14}{3d+7}$, \eqref{a-sob} holds for $\frac d2\le p\le \frac{3d+7}{6}$ and $s>\kappa-\frac67(2-\frac dp)$.
\item If  $\frac{14}{3d+7}<\kappa \le \frac{18}{d+9}$, \eqref{a-sob} holds for $\frac d2\le p\le \frac d{2-\kappa}$ and $s>\kappa-(\frac{3d-1}{4}+\frac1{2\kappa})(\frac2d-\frac 1{p})$.
\item If $\frac{18}{d+9}<\kappa\le 1$,\footnote{This is void if $d\le 9$.} then \eqref{a-sob} holds for $\frac d2\le \kappa\le \frac{d}{2-\kappa}$ and $s>\kappa-(\frac{d-5}2+\frac5\kappa)(\frac2d-\frac1p)$.
\end{itemize}

Now, for simplicity, let us denote by $s>s_\circ(d,p,\kappa)$ the conditions on $s$ arranged above.
\begin{prop}\label{hohoho} 
Let $d\ge3$, $0\le \kappa \le 1$, and let $\frac d2\le p  <\infty$. Suppose that $s> s_\circ(d,p,\kappa)$ and $s\ge 0$. Then \eqref{a-sob} holds for all $f\in W_c^{s,p}(B_d(0,R))$ for any fixed $R>0$. Moreover, if $f\in W_c^{s,p}(\R^d)$ 
\Be\label{ave_decay_w} 
	\lim_{M\to \infty}  \mathfrak A_{M}^{p,\kappa} (f)=0.
\Ee
When $p=\frac d2$ and $\kappa=0$, \eqref{ave_decay_w} holds for any $f\in L^\frac{d}{2}(\R^d)$.
\end{prop}
For $4\le d\le 8$,  we can make the exponent  $s_\circ(d,p,\kappa)$ slightly smaller for $p\ge d$ if we use Proposition \ref{bilinear-sum-2} instead of Proposition \ref{bilinear-sum}. For the case $\kappa=1$, this will be done in Proposition \ref{hohoho2} below.
\begin{proof}[Proof of Proposition \ref{hohoho}]
We have already shown all the statement for the estimate \eqref{a-sob}. So, it remains to prove \eqref{ave_decay_w} and the proof is similar to that of  Corollary \ref{with-d-cor}. Writing  $f=f\ast \phi_\epsilon+(f- f\ast \phi_\epsilon)$ and using \eqref{a-sob}, we get 
\[ {\mathfrak A_M^{p,\kappa} } (f- f\ast \phi_\epsilon)\lesssim 
\| f- f\ast \phi_\epsilon \|_{W^{s, p}}. \] 
On the other hand, the estimate \eqref{ha1}  and the H\"older inequality give 
\[ \| \mathcal M_{m^\kappa (D) ( f\ast \phi_\epsilon) }  \|_{X_{\zeta(\tau,U)}^{1/2}\to X_{\zeta(\tau,U)}^{-1/2}} 
\lesssim \tau^{-1}  \| f\ast  (m^{\kappa} (D)\phi_\epsilon)\|_{L^\infty}  
\lesssim   \tau^{-1}  \epsilon^{-\kappa-\frac dp}  \|  f \|_{L^p}.\]
Here we used the Mikhlin multiplier theorem to see that $\|m^{\kappa} (D)\phi_\epsilon\|_{L^r}  \lesssim \|\phi_\epsilon\|_{H^{1,r}} \lesssim \epsilon^{-\kappa-d+\frac dr}$ for $1< r < \infty$. This immediately  yields  $\mathfrak A_M^{p, \kappa}( f\ast \phi_\epsilon)\lesssim M^{-1}  \epsilon^{-\kappa-\frac dp}  \|  f \|_{L^p}$.  
Thus,  
\begin{align*} 
\mathfrak A_M^{p, \kappa}(f) \le \mathfrak A_M^{p, \kappa} ( f\ast \phi_\epsilon)+   \mathfrak A_M^{p, \kappa}(f- f\ast \phi_\epsilon)
	\lesssim  M^{-1}\epsilon^{-\kappa-\frac dp} \|  f \|_{L^p}+  \| f- f\ast \phi_\epsilon \|_{W^{s, p}}.
\end{align*} 
Since $s\ge 0$, $\|f\|_{L^p}\lesssim \|f\|_{W^{s,p}}$. Taking  $  \epsilon= M^{-1/\beta}$ with $\beta > \kappa+\frac dp$,  we get \eqref{ave_decay_w}. The last statement follows  from \eqref{fff} in a similar manner.
\end{proof}

We now turn to \eqref{a-sob-2}. By Proposition \ref{bilinear-sum}, Corollary \ref{scaled-n}, and Lemma \ref{huhu}, we have the following: \emph{If $d\ge 3$, $\epsilon>0$, $\frac{d+1}2\le p <\infty$, $\beta_d(p)<\mu<2-\frac{d+2}p$, and $r\ge \max\{-1, \mu+\frac dp-2\}$, then \eqref{a-sob-2} holds for any $f\in H_c^{r+\epsilon,p}(B_d(0,C))$.}

Such $\mu$ exists only if $\beta_d(p)<2-\frac{d+2}p$, which is equivalent to $p>\frac{d+2}{2}$ if $3\le d \le 6$, and $p>\frac{3d+7}6$ if $d\ge 7$. As before, we interpolate \eqref{a-sob-2} and \eqref{fff} (with $\kappa =0$) to extend the range of $p$ and obtain the following proposition. The proof is very similar to and even simpler than that of Proposition \ref{hohoho}, so we omit it.

\begin{prop} \label{hohoho3} Let $\epsilon>0$.

\vspace{-8pt}  
$(\mathrm I)$ If $3\le d\le 6$, then \eqref{a-sob-2} holds for $r\ge \max\{-1, \beta_d(p)+\frac dp-2\}$ when $\frac{d+2}2<p<\infty$, and for $r\ge \frac dp -2$ when $\frac d2 \le p \le \frac{d+2}2$. 

\vspace{-8pt}  
$(\mathrm I\!\mathrm  I)$ If $d\ge 7$, then \eqref{a-sob-2} holds for $r\ge \max\{-1, \beta_d(p)+\frac dp-2\}$ when $\frac{3d+7}6<p<\infty$, and for $r\ge -\frac 67(2-\frac dp)$ when $\frac d2 \le p \le \frac{3d+7}6$. 

\vspace{-8pt}
$(\mathrm I\!\mathrm  I \!\mathrm I)$ If $d$, $p$, and $r$ are given as in $(\mathrm I)$, $(\mathrm I\!\mathrm  I)$,\footnote{Note that the conditions on $r$ in $(\mathrm I)$, $(\mathrm I\!\mathrm  I)$ can be written equivalently $r\ge \max\{-1, r_d(p)\}$ (see Section \ref{intro}).} then for any $f\in H_c^{r+\epsilon, p}(\R^d)$,
\Be\label{ave_decay_h} 
 \lim_{M\to \infty}  \mathfrak A_{M}^{p} (f)=0.
\Ee
\end{prop}

We close this section weakening the condition on $s$ in Proposition \ref{hohoho} in the case of $p\ge d$ and $\kappa=1$. For the definition of $\gamma_d(p)$, see Section \ref{sec_str}.
\begin{prop}\label{hohoho2} Let $3\le d\le 8$, $d\le p<\infty$, and $\kappa=1$. Suppose $s> \max\{0,  \gamma_d(p)+ \frac d p   -1 \}$. Then, for any $f\in  W^{s,p}_c(B_d(0, R))$ (for any constant $R>0$ fixed), \eqref{a-sob} and \eqref{ave_decay_w} hold.
\end{prop}

\begin{proof}
As before, from Proposition \ref{bilinear-sum-2}, Proposition \ref{low-high}, and Corollary \ref{scaled} we have 
\eqref{result2} with $\nu=1-\frac 2p$ and $s>\gamma_d(p)+\frac dp-1$ for $p\ge  d$. 
Thus,   by Lemma \ref{with-d-av}  we get \eqref{a-sob} since 
$2-\frac {d+2}p>  \gamma_d(p)$ whenever $p\ge d$.  Since we have \eqref{a-sob}, the same argument as in the proof of Proposition \ref{hohoho} gives 
\eqref{ave_decay_w}. 
\end{proof}

\section{Proof of Theorem \ref{main} and Theorem \ref{Schrodinger}  }\label{sec_pf_thm}

Once we have the key estimates in the previous sections, we can prove Theorem \ref{main} and Theorem \ref{Schrodinger} following   
the argument in \cite{Haberman15},  which   also relies on  the basic strategy due to Sylvester-Uhlmann \cite{SU}, and subsequent modifications due to 
Haberman-Tataru \cite{HT} and  Nguyen-Spirn \cite{NS}. We begin with recalling several basic theorems  which we need in what follows.

\subsection{Proof of Theorem \ref{main}}
Let $\Omega$, $s$, $p$, and $d$ be as in Theorem \ref{main}. We may assume that $s\le 1+\frac1p$ by the inclusion $W^{s_1,p}\subset W^{s_2,p}$ for $s_1\ge s_2$ and $1<p<\infty$. For $k=1,2$, assume that   $\gamma_k \in W^{s,p}(\Omega)\cap A(\Om)$ satisfy $\Lambda_{\gamma_1} = \Lambda_{\gamma_2}$. 
It is clear  that $W^{s,p}(\Omega) \subset W^{1,d}(\Omega) \subset W^{1,1}(\Omega)$ since $s\ge 1$ and $p \ge d$. 
The following is due to 
Brown \cite{Brown01}.
\begin{lem}\label{boundary}
If $\gamma_1$, $\gamma_2 \in W^{1,1}(\Omega)\cap A(\Omega)$ and $\Lambda_{\gamma_1} = \Lambda_{\gamma_2}$, then 
$\gamma_1 = \gamma_2$ on $\partial \Omega$.
\end{lem}

Hence $(\gamma_1-\gamma_2)\vert_{\partial \Omega}=0$ by Lemma \ref{boundary}. Since $s- \frac1p \le 1$, using \cite[Theorem 1]{Marschall},  we have   $\gamma_1, \gamma_2 \in W^{s,p}(\Omega)$ extended to the whole space $\mathbb R^d$ such that $\gamma_1 = \gamma_2$ outside of $\Omega$ and $\gamma_k$ is supported in a large ball $B_d(0,R)$ containing $\Omega$.  Thus,  from now on  we  assume that $\gamma_1$ and $\gamma_2$ are  in $W_c^{s,p}(\R^d)$.

Recall that the conductivity equation $\divergence(\gamma \nabla u)=0$ is equivalent to the equation $(\Delta - q)v = 0$, where $v= \gamma^{1/2}u$ and $q = \gamma^{-1/2}\Delta \gamma^{1/2}$.   

In what follows we also make use of the next two lemmas:

\begin{lem}{\cite[Proposition 2]{BT}}\label{identity}
Suppose that $\gamma_k \in W^{1,d}(\mathbb R^d)\cap A(\R^d)$, $\nabla \gamma_k^{1/2}$ is supported in a bounded set, and $\gamma_1 =\gamma_2$ outside of $\Omega$. 
If $v_k$ are solutions in $H^{1}_{loc}(\R^d)$ to $(\Delta - q_k) v_k=0$ with $q_k = \gamma_k^{-1/2}\Delta \gamma_k^{1/2}$, 
then $  ( q_1 , v_1 v_2)= ( q_2, v_1 v_2)$.
\end{lem}

\begin{lem}{\cite[Lemma 7.2]{Haberman15}}\label{equality}
Let $\gamma_k$ and $q_k$ be given as in Lemma \ref{identity}. 
If $q_1 =q_2$ in the sense of distributions, then $\gamma_1 = \gamma_2$. 
\end{lem}

 In order to prove Theorem \ref{main},  by Lemma  \ref{equality}  it suffices to show  that $\widehat{ q_1} =  \widehat{q_2}$.    
This will be done by  constructing the complex geometrical optics solutions $v_k$ with a parameter $\tau$ to the equations $(\Delta - q_k) v_k=0$ such that 
$v_1(x)v_2(x)$ converges to $e^{i\xi  \cdot  x}$ for a fixed $\xi \in \mathbb R^d$ as $\tau\to \infty$. 
In fact, the solutions $v_k$  take the form of $e^{x \cdot{\zeta_k'}(\tau,U)} (1 + \psi_k)$, where $\psi_k$'s are solutions to
\Be
\label{tildeeq}
(\Delta_{\zeta_k'(\tau,U)} - {q_k} ) \psi_k  = q_k,
\Ee
where $\Delta_\zeta=\Delta+2\zeta\cdot \nabla$. 

Fix orthonormal vectors $e_1,$ $e_2,$ $e_3$ in $\R^d$ and $r>0$. For $U\in O_d$ and $\tau\ge\max\{1,r\}$, we set, as in \cite{Haberman15},
\begin{equation} 
\label{tau-u}
\begin{aligned}
\zeta_1(\tau, U) &= \tau U ( e_1 - i e_2), & \zeta_1'(\tau, U) &= \tau U e_1 - i\sqrt{\tau^2 - r^2} U e_2  + i r Ue_3 , \\
\zeta_2(\tau, U) &= -\tau U ( e_1 - i e_2), & \zeta_2'(\tau, U) &= -\tau U e_1 + i\sqrt{\tau^2 - r^2} U e_2  + i r Ue_3  .
\end{aligned}
\end{equation}

We may write 
\[ q_k =\gamma_k^{-1/2}\Delta \gamma_k^{1/2} = \sum_j \partial_j f_{k,j} +h_k, \]
where   $f_{k,j} = \frac12\partial_j \log \gamma_k$ and $h_k= \frac14 |\nabla \log \gamma_k|^2$.  It is clear that  $f_{k,j} \in W_c^{s-1,p}(\R^d)$ and $h_k \in L_c^{ p /2}(\R^d)$ since $s\ge 1$. 

By Proposition \ref{hohoho2} (applied to $f_{k,j} \in W_c^{s-1,p}$ with $\kappa=1$), Proposition \ref{hohoho} (applied to $h_k\in L^\frac d2$), and Corollary \ref{with-d-cor} (applied to $\partial f_{k,j}\in H_c^{s-2-\epsilon,p}$ and  $h_k \in L^{\frac d 2}$),  we see that  
\Be
\label{exist-seq}
  \lim_{M\to \infty}\avm  \int_{O_d} \sum_{k, \ell=1,2}\bigg( \| \mathcal M_{q_k} \|_{X_{\zeta_\ell(\tau,U)}^{1/2}\to X_{\zeta_\ell(\tau,U)}^{-1/2}}^p 
  + \|q_k\|_{X_{\zeta_\ell(\tau,U)}^{-1/2}}^2 \bigg) dm(U) d\tau  = 0,
 \Ee
if $s-1 > \max\{  \gamma_d(p) +\frac d p -1,0\}$, i.e., $s > \max\{ \gamma_d(p) +\frac d p,1\} = s_d(p)$ when $d=5,6$. Hence, there exist sequences $\tau:=\tau_j>0$, $U:=U_j \in O_d$, and $\delta:=\delta_j>0$ (in what follows  we occasionally omit the subscript $j$ for simplicity of notation)
such that
\Be
\label{tau-delta}
\lim_{j\to\infty} \tau_j=\infty, \quad  \lim_{j\to \infty} \delta_j=0,
\Ee
and,   for $k,\ell=1,2$,
\begin{align} 
\label{inhomo}
\| \mathcal M_{q_k}\|_{ X_{ \zeta_\ell(\tau_j,U_j)}^{1/2}\to  X_{  \zeta_\ell( \tau_j,U_j)}^{-1/2}} <\delta_j , \quad
\|q_k\|_{ X_{ \zeta_\ell(\tau_j,U_j)}^{-1/2}} <\delta_j.
\end{align}
Since $|\zeta_\ell(\tau,U) - \zeta_\ell'(\tau,U)| \simeq r$, 
we have
\begin{equation}\label{norm_equiv}
\| u\|_{X_{\zeta_\ell(\tau,U)}^{b}} 
\simeq \| u\|_{X_{\zeta_\ell'(\tau,U)}^{b}}	
\end{equation}
for any $b \in \mathbb R$ with the implicit constant depending on $r$ (see \cite[Lemma 6.3]{Haberman15}).  
It follows from \eqref{inhomo} and Lemma \ref{properties}  that for $k,\ell =1,2$, 
\begin{equation}
\label{smallness} 
	\| \mathcal M_{q_k}\|_{ \bdot   X_{\zeta_\ell' (\tau,U)}^{1/2}\to \bdot X_{\zeta_\ell' ( \tau,U)}^{-1/2}}  \lesssim \delta, \quad 	\|q_k\|_{\bdot X_{\zeta_\ell' (\tau,U)}^{-1/2}} \lesssim \delta.
\end{equation}
Here and later on,  the implicit constants depend on $r$ but are independent of $\tau$. The precise dependence is not important since $\delta\to 0$ while $\tau\to \infty$.

With sufficiently large $j$,  by the contraction mapping principle  (or the operator $\Delta_{\zeta_k'(\tau,U)} - {q_k }$ is invertible  since $\|\Delta_{\zeta_k'(\tau,U)}^{-1} \mathcal M_{q_k }\|_{ \bdot X_{ \zeta_k'(\tau,U)}^{1/2}\to \bdot  X_{  \zeta_k'( \tau,U)}^{1/2}} $ is small)
we have solutions $\psi_k\in \bdot X_{\zeta_k'(\tau,U)}^{1/2}$ to the equations $(\Delta_{\zeta_k'(\tau,U)} - {q_k }) \psi_k = q_k$, $k=1,2$, such that 
\begin{equation}\label{psi}
\|\psi_k \|_{\bdot X_{\zeta_k'(\tau,U)}^{1/2}} 
\lesssim
\| q_k \|_{\bdot X_{\zeta_k'(\tau,U)}^{-1/2}}.
\end{equation}
Indeed, since $\psi_k=\Delta^{-1}_{\zeta_k' (\tau,U)} ({q_k} \psi_k + q_k)$ and $\| \Delta^{-1}_\zeta \|_{\bdot X_\zeta^{-1/2}  \to \bdot X_\zeta^{1/2}}=1$, we have that
\[
\|\psi_k \|_{\bdot X_{\zeta_k' (\tau,U)}^{1/2}} = \| \Delta^{-1}_{\zeta_k'(\tau,U)} ({q_k} \psi_k + q_k) \|_{\bdot X_{\zeta_k'(\tau,U)}^{1/2}} 
\le \| \mathcal M_{q_k} \psi_k + q_k  \|_{\bdot X_{\zeta_k'(\tau,U)}^{-1/2}} \lesssim \delta \| \psi_k\|_{\bdot X_{\zeta_k'(\tau,U)}^{1/2}} + \| q_k \|_{\bdot X_{\zeta_k'(\tau,U)}^{-1/2}}.
\]
If $j$ is large enough, \eqref{psi} follows. Furthermore, since $\| u \|_{H^{1,2}(\R^d)} \lesssim \tau^{1/2} \|u \|_{X_{\zeta(\tau)}^{1/2}}$, we have $\psi_k \in H^{1}(\R^d)$,  hence $v_k=e^{x \cdot{\zeta_k'}(\tau,U)} (1 + \psi_k) \in H_{loc}^{1}(\R^d)$. 

Therefore,  by the assumption  $\Lambda_{\gamma_1} = \Lambda_{\gamma_2}$ and Lemma \ref{boundary} we can apply Lemma \ref{identity} to get $(q_1 -q_2 ,v_1v_2) = 0$. 
Since $v_1v_2 = e^{ x\cdot(\zeta_1' + \zeta_2' )}(1+\psi_1)(1+\psi_2) = e^{i x \cdot 2rUe_3} (1+\psi_1)(1+\psi_2)$,  we now obtain  that
 \[
 ( q_1 -q_2 , e^{i x \cdot 2rUe_3}) = \sum_{k,n=1,2} (-1)^k ( q_k , e^{i x \cdot 2rUe_3} \psi_{n} )  
 + ( q_2-q_1 , e^{ i x \cdot 2rUe_3} \psi_1\psi_2).
\]
By \eqref{psi}, we have 
\begin{align*}
 | ( q_k ,\, e^{i x \cdot 2rUe_3} \psi_{n}) | 
\lesssim  
 \| q_k\|_{\bdot X_{\zeta_{n}'(\tau,U)}^{-1/2}} \| \psi_{n}\|_{\bdot X_{\zeta_{n}'(\tau,U)}^{1/2}}  
\lesssim 
 \| q_k\|_{\bdot X_{\zeta_{n}'(\tau,U)}^{-1/2}} \| q_{n}\|_{\bdot X_{\zeta_{n}' (\tau,U)}^{-1/2}}.
\end{align*}
Here we also use 
that $\|e^{\pm i x \cdot 2rUe_3} \phi \|_{ \bdot X_{ \zeta'(\tau,U)}^{1/2}}\simeq\| \phi \|_{ \bdot X_{ \zeta'(\tau,U)}^{1/2}}$ with the implicit constant which may  depend on $r$. This is easy to see  since the modulation $e^{i x \cdot 2rUe_3} $  is acting only on $\bar \xi$. 
Applying \eqref{smallness}, \eqref{norm_equiv}, and \eqref{psi}, successively, we get 
 \begin{align*}
|( q_1 , e^{i x \cdot 2rUe_3} \psi_1\psi_2) |
& \le   
 \| \mathcal M_{q_1}  ( e^{i x \cdot 2rUe_3}    \psi_1 ) \|_{ \bdot X_{ \zeta_2(\tau,U)}^{-1/2}}   
                   \| \psi_2 \|_{ \bdot X_{ \zeta_2(\tau,U)}^{1/2}}
                   \\
& \lesssim
\delta \|  e^{ i x \cdot 2rUe_3}  \psi_1   \|_{\bdot X_{\zeta_2(\tau, U)}^{ 1/2}}  \| \psi_2 \|_{\bdot X_{\zeta_2(\tau, U)}^{1/2}}  
 \simeq \,\,\delta \|  e^{- i x \cdot 2rUe_3} \overline{ \psi_1} \|_{\bdot X_{\zeta_1(\tau,U)}^{ 1/2}}   \| \psi_2 \|_{\bdot X_{\zeta_2'(\tau,U)}^{ 1/2}}  \\
& \lesssim  \delta \| \overline{\psi_1} \|_{\bdot X_{\zeta_1'(\tau,U)}^{1/2}}  \| q_2 \|_{\bdot X_{\zeta_2'(\tau,U)}^{-1/2}}  
\lesssim \delta  \| \overline{q_1} \|_{\bdot X_{\zeta_1'(\tau,U)}^{-1/2}}  \| q_2 \|_{\bdot X_{\zeta_2'(\tau,U)}^{-1/2}} 
\simeq \delta  \|  q_1  \|_{\bdot X_{\zeta_2'(\tau,U)}^{-1/2}}  \| q_2 \|_{\bdot X_{\zeta_2'(\tau,U)}^{-1/2}} ,
\end{align*} 
where we use the fact that $\zeta_1(\tau,U) = - \zeta_2(\tau,U)$.  
We also obtain similar estimate for $q_2$ replacing $q_1$. 
Hence, combining all the above estimates 
we get 
\[
|(\widehat{q_1} - \widehat{q_2}) (-2 r U e_3) | 
\lesssim \delta
\sum_{k,\ell, m, n} \|q_k\|_{\bdot X_{\zeta_\ell'(\tau,U)}^{-1/2}} \|q_m\|_{\bdot X_{\zeta_{n}'(\tau,U)}^{-1/2}} 
\lesssim 
 \delta^3.
\]
This shows that  $\lim_{j\to\infty} (\widehat{q_1} - \widehat{q_2})(-2 r U_j e_3)=0$ for the sequence $\{U_j \}$. Meanwhile, since $O_d$ is compact, we can pass to  a subsequence 
so that $U_{j}$ converges to a unitary matrix $U_\ast$. Thus we have  
$(\widehat{q_1} - \widehat{q_2})(-2 r U_\ast e_3)=0$. Therefore, we conclude that  $\widehat{q_1} - \widehat{q_2} =0$ as $e_3$ and $r$ are arbitrary.
This completes the proof of Theorem \ref{main}.

\begin{rmk}\label{sd} 
From Proposition \ref{hohoho2} (for $d=7,8$) and Proposition \ref{hohoho} (for $d\ge 9$), if we apply \eqref{ave_decay_w} (with $\kappa=1$) to the potentials $q_k=\gamma_k^{-1/2}\Delta \gamma_k^{1/2}$ where $\gamma_k\in W_c^{s,p}(\R^d)$, we see that \eqref{exist-seq} holds whenever $d\le p<\infty$ and  $s > s_d(p)$ that is given by
\begin{align*} 
s_d(p) 
   &=
\begin{cases}
\, 	1 +  \frac{d - 5}{2p}	&\text{if }\ \   \frac{d+9}2 \le  p <\infty ,  \\[5pt]
\, 	1 +  \frac{d^2 +d -16-2p}{2p(d+5)}	&\text{if }\ \ d\le  p < \frac{d+9}2, 
\end{cases}
\  \quad \text{ for }  d=7,8,
\\[4pt] 
s_d(p) 
    &=
\begin{cases}
\, 	1 +  \frac{d - 5}{2p}	&\text{if }\ \   \frac{d+9}2 \le  p <\infty ,  \\[5pt]
\, 	\frac12 +  \frac{3d -1}{4p}	&\text{if }\ \  d\le  p < \frac{d+9}2\,,
\end{cases}  
\qquad  \quad \quad  
\text{ for } d\ge9.
\end{align*}
As is mentioned in the introduction, if we have the additional condition ${\partial   \gamma_1}/{\partial \nu}={\partial   \gamma_2}/{\partial \nu}$ on the boundary, the zero-extension of $\gamma_1 - \gamma_2$ is valid if $s-\frac1p\le 2$.   Since 
$s_d(p)-\frac1p \le 2$ for $d\le p<\infty$ and  \eqref{exist-seq} is valid whenever $s > s_d(p)$, by the above argument, the injectivity of the mapping $W^{s, p}\ni \gamma \mapsto \Lambda_\gamma$ follows whenever  $d\le p<\infty$ and  $s > s_d(p)$. 
 \end{rmk}

\subsection{Proof of Theorem \ref{Schrodinger}}
Before we prove Theorem \ref{Schrodinger} we justify $\mathcal L_q$ is well defined with  $q\in H^{s,p}_c(\Omega)$  while $s,p$ satisfy  \eqref{concon}.

\begin{lem}\label{DtoN}  Let $p\ge \frac d2$ and  $\mathcal L_q$  be given  by \eqref{DtN}. 
Suppose  $q\in H^{s,p}_c(\Omega)$  and $s,p$ satisfy  \eqref{concon}.  
Then,  $\mathcal L_q$ is well defined and  continuous from $H^\frac12(\partial \Omega)$ to $H^{-\frac12}(\partial \Omega)$.
\end{lem}  

\begin{proof} We may  assume $s=\max\{-2+\frac dp,-1\}$ since $H^{s_2,p}\hookrightarrow H^{s_1,p}$ if $s_1\le s_2$. 
Let $u\in H^1(\Omega)$ be a solution to \eqref{seq} and $v\in  H^1(\Omega)$ with $v|_{\partial \Omega}=g$.  Then the quantity 
\[ S_u(v)=\int_\Omega \nabla u\cdot \nabla v+ q uv\, dx \]  
 is well defined.  In fact, we note that  $|(q,uv)|=|((1+|D|^2)^{\frac s2} q, (1+|D|^2)^{-\frac s2}(uv))|\le \|(1+|D|^2)^{\frac s2} q\|_{L^p} \| (1+|D|^2)^{-\frac s2}(uv)\|_{L^{p'}} $. By the Kato-Ponce inequality (\cite{KaPo, GuKo}) and the Hardy-Littlewood-Sobolev inequality we get  
\begin{align*}
 |(q,uv)|&
\lesssim    \|q\|_{H^{s, p}} \big(  \| (1+|D|^2)^{-\frac s2}u\|_{L^t} \|v\|_{L^\frac{2d}{d-2}}+  \| u \|_{L^\frac{2d}{d-2}}\|(1+|D|^2)^{-\frac s2}v\|_{L^t}\big)
 \\
 & \lesssim  \|q\|_{H^{s, p}} \|u\|_{H^{1}(\Omega)}\|v\|_{H^{1}(\Omega)} ,
 \end{align*}
where $s=\max\{-2+\frac dp,-1\}$ and $\frac 1t=\frac{d+2}{2d}-\frac 1p$.  Thus we have 
\Be   |S_u(v)|\lesssim  \|u\|_{H^{1}(\Omega)}\|v\|_{H^{1}(\Omega)} \label{prep}. \Ee
Since $u\in H^1(\Omega)$ is a solution to \eqref{seq}, $S_u(v_\circ -v)=0$ for all  $v_\circ \in  H^1(\Omega)$ with $v_\circ|_{\partial \Omega}=g$ because 
$v_\circ -v\in H_0^1(\Omega)$. This shows $S_u(v)$ does not depend on particular choices of $v$, that is to say, $\mathcal L_q$ is well defined. 

To show  $\mathcal L_q:H^\frac12(\partial \Omega)\to H^{-\frac12}(\partial \Omega)$ is continuous, by duality it is sufficient to show that 
\[  |(\mathcal L_q f, g)|\lesssim \|f\|_{H^\frac12(\partial \Omega)} \|g\|_{H^\frac12(\partial \Omega)}.  \] 
 From \eqref{DtN}  and \eqref{prep}  we have $ |(\mathcal L_q f, g)|\lesssim \|u\|_{H^{1}(\Omega)}\|v\|_{H^{1}(\Omega)}.$
Using the right inverse of the  trace operator we have $\|u\|_{H^{1}(\Omega)}\lesssim  \|f\|_{H^\frac12(\partial \Omega)} $ and $\|v\|_{H^{1}(\Omega)}\lesssim  \|g\|_{H^\frac12(\partial \Omega)} $. This gives the desired estimate.
\end{proof}

From the standard argument, similarly handling $(q,uv)$ as in the above, 
it is also easy to see the following.  (See Section 2.7 in \cite{FSU}, for example.) 
\begin{prop}\label{zero1}  Let $p\ge \frac d2$. Suppose $q_1, q_2\in H^{s,p}_c(\Omega)$ and $s,p$ satisfy  \eqref{concon} and  suppose  $\mathcal L_ {q_1}=\mathcal L_{q_2}$. Then, 
$(q_1-q_2, u_1 u_2)=0$ whenever  $u_i\in H^1(\Omega)$  is a solution to $\Delta u-q_i u= 0$  for each $i=1,2$. 
\end{prop}
 
 Now we show Theorem \ref{Schrodinger} by  constructing the complex geometrical  optics solutions. We follow the lines of arguments in the proof of Theorem \ref{main}.  Let $s$, $p$ be given as in Theorem \ref{Schrodinger} and   $q_1,$ $q_2\in H_c^{s,p}(\Omega)$.  
By Lemma \ref{elementary}, it is enough to consider the case $s \le 0$.
Then, by Proposition \ref{hohoho3} we have, for $k=1,2$, 
\begin{equation}\label{mqk} \lim_{M\to \infty} \bigg( \avm \int_{O_d} \| \mathcal M_{q_k} \|_{X_{\zeta(\tau,U)}^{1/2}\to  X_{\zeta(\tau,U)}^{-1/2}}^p dm(U) d\tau\bigg)^\frac1p =0   \end{equation} 
for $0 \ge s > \max\{-1, r_d(p)\}$. From Corollary \ref{with-d-cor} we see that  \eqref{q} holds for $s\ge\max\{-1,\frac d p -2\}$ and $ \frac d 2 \le p<\infty$.  
We also have,  for $k=1,2,$ 
 \begin{equation}\label{qk} 
  \lim_{\tau \to \infty} \int_{O_d} \| q_k\|^2 _{\bdot{X}_{\zeta(\tau,U)}^{-1/2}} dm(U) =0 . 
 \end{equation} 
 Let $\zeta_1(\tau, U)$, $\zeta_1'(\tau, U)$, $\zeta_2(\tau, U)$, and $\zeta_2'(\tau, U)$ be given by   \eqref{tau-u}.
   Combining these two,  we have 
 $\tau = \tau_j>0$, $U = U_j \in O_d$, and $\delta_j>0$  such that 
\eqref{tau-delta} and \eqref{inhomo} hold.   Once we have  $\tau = \tau_j>0$, $U = U_j \in O_d$, then the rest of argument works without modification. So we omit the details.

\begin{rmk}\label{nachman}
When $d\ge3$ the above argument provide a different proof of  the uniqueness result for  $q\in L^{d/2}$ (\cite{N}) by using \eqref{ha2} instead of \eqref{a-sob-2}.
As observed above, it is sufficient to show \eqref{mqk} and \eqref{qk} for $q\in L^{d/2}$.
Following the proof of Proposition \ref{hohoho} (\eqref{ave_decay_w}), we write $q = q\ast \phi_\epsilon + (q-q\ast\phi_\epsilon)$. 
By \eqref{ha1} and \eqref{ha2}, we obtain 
\[
 \mathfrak A_M^{p,0} (q) \lesssim \tau^{-1} \| q \ast \phi_\epsilon\|_{L^\infty} + \| q- (q\ast\phi_\epsilon)\|_{L^{\frac d 2}}  \lesssim \tau^{-1} \epsilon^{-2} \| q \|_{L^{\frac d2}} + \| q- (q\ast\phi_\epsilon)\|_{L^{\frac d 2}}. 
\]
Taking $\epsilon =\tau^{-\frac 14}$, we see that \eqref{mqk} holds for $q_k \in L^{\frac d 2}$. 
Meanwhile \eqref{qk} is immediate by \eqref{q}  with $s =0$  and $p=\frac d 2$. The remaining is identical with the previous argument. 
\end{rmk}
 
\subsection*{Acknowledgement} 
This work was supported  by NRF-2017R1C1B2002959  (S. Ham), a KIAS Individual Grant (MG073701) at Korea Institute for Advanced Study (Y. Kwon), and  NRF-2018R1A2B2006298 (S. Lee).

\bibliographystyle{plain}

\end{document}